\documentclass[oneside,english]{amsart}
\usepackage[latin9]{inputenc}
\usepackage{amstext}
\usepackage{amsthm}
\usepackage{amssymb}
\usepackage{graphicx}
\usepackage{wasysym}

\makeatletter

\providecommand{\tabularnewline}{\\}

\numberwithin{equation}{section}
\numberwithin{figure}{section}
\theoremstyle{plain}
\newtheorem{thm}{\protect\theoremname}
\theoremstyle{plain}
\newtheorem{prop}[thm]{\protect\propositionname}
\theoremstyle{remark}
\newtheorem{rem}[thm]{\protect\remarkname}
\theoremstyle{plain}
\newtheorem{lem}[thm]{\protect\lemmaname}
\theoremstyle{definition}
\newtheorem{defn}[thm]{\protect\definitionname}
\theoremstyle{plain}
\newtheorem{cor}[thm]{\protect\corollaryname}
\theoremstyle{definition}
\newtheorem{example}[thm]{\protect\examplename}

\usepackage{graphicx,xcolor}
\numberwithin{thm}{section}

\@ifundefined{showcaptionsetup}{}{%
 \PassOptionsToPackage{caption=false}{subfig}}
\usepackage{subfig}
\makeatother

\usepackage{babel}
\providecommand{\corollaryname}{Corollary}
\providecommand{\definitionname}{Definition}
\providecommand{\examplename}{Example}
\providecommand{\lemmaname}{Lemma}
\providecommand{\propositionname}{Proposition}
\providecommand{\remarkname}{Remark}
\providecommand{\theoremname}{Theorem}

\begin{document}
\title{The enumeration of extreme rigid honeycombs}
\author{Hari Bercovici and Wing Suet Li}
\address{Mathematics Department\\
Indiana University\\
Bloomington, IN 47405}
\email{bercovic@indiana.edu}
\address{School of Mathematics\\
Georgia Institute of Technology\\
Atlanta, GA 30332}
\email{li@math.gatech.edu}
\thanks{WSL was supported in part by Simons Foundation collaborative grant
416045.}
\begin{abstract}
Rigid tree honeycombs were introduced by Knutson, Tao, and Woodward
\cite{KTW} and they were shown in \cite{bcdlt} to be sums of extreme
rigid honeycombs, with uniquely determined summands up to permutations.
Two extreme rigid honeycombs are essentially the same if they have
proportional \emph{exit multiplicities} and, up to this identification,
there are countably many equivalence classes of such honeycombs. We
describe two ways to approach the enumeration of these equivalence
classes. The first method produces a (finite) list of all rigid tree
honeycombs of fixed \emph{weight} by looking at the \emph{locking
patterns} that can be obtained from a certain quadratic Diophantine
equation. The second method constructs arbitrary rigid tree honeycombs
from \emph{rigid overlays} of two rigid tree honeycombs with strictly
smaller weights. This allows, in principle, for an inductive construction
of all rigid tree honeycombs starting with those of unit weight. We
also show that some rigid overlays of two rigid tree honeycombs give
rise to an infinite sequence of rigid tree honeycombs of increasing
complexity but with a fixed number of nonzero exit multiplicities.
This last result involves a new inflation/deflation construction that
also produces other infinite sequences of rigid tree honeycombs.
\end{abstract}

\maketitle

\section{Introduction \label{sec:Introduction}}

The notion of a \emph{rigid honeycomb} was introduced in \cite{KTW}
in relation to the Horn problem. It also played a central role in
proving that the Horn inequalities hold in an arbitrary factor of
type II$_{1}$ in \cite{bcdlt}, where it was shown that every rigid
honeycomb arises as a special kind of \emph{overlay} (more general
than the \emph{clockwise overlays }of \cite{KTW}) of \emph{extreme}
rigid honeycombs. Extreme rigid honeycombs are scalar multiples of
the rigid tree measures (which we now call \emph{rigid tree honeycombs})
studied in \cite{blt}.

In retrospect, the results of \cite{bcdlt} depended on the study
of these extreme rigid honeycombs, and one of the initial difficulties
was the lack of examples of sufficiently complicated such honeycombs,
which in turn hampered effective experimentation. (Producing just
six such honeycombs of weight four required an extensive search; an
intersection problem generated by one of these six examples is fully
analyzed at the end of \cite[Section 6]{bcdlt}.) Subsequently, several
infinite families of extreme rigid honeycombs were constructed in
\cite{blt}.

Our purpose is to address the following two questions:
\begin{enumerate}
\item [(a)]Can one effectively enumerate all the rigid tree honeycombs
of a given weight $\omega?$
\item [(b)]How can one generate rigid tree honeycombs of high weight from
rigid tree honeycombs of smaller weight?
\end{enumerate}
In answer to (a), we offer a method based on the study of certain
combinatorial objects called \emph{locking patterns.} This makes the
enumeration of (types of) rigid tree honeycombs possible, though still
tedious for large weights. For (b), we show that every rigid tree
honeycomb of weight at least two can be constructed from an overlay
of honeycombs with strictly lower weights. In principle, this allows
one to use the enumeration of rigid tree honeycombs of weight less
than a fixed $k\in\mathbb{N}$ to construct all possible such honeycombs
of weight $k$. In practice, this is difficult to achieve because
there are usually many overlays of two honeycombs of given type and
these overlays are hard to construct directly. However, once an overlay
of two honeycombs, or even a single honeycomb, is known, we show how
to construct honeycombs of much greater complexity and weight. These
constructions recover, for instance, some of the infinite families
constructed in \cite{blt} and produce many additional infinite families
of rigid tree honeycombs. The honeycombs constructed this way should
be instrumental in the experimental study of related conjectures,
such as those of \cite{danilov}.

The paper is organized as follows. Basic definitions about honeycombs,
puzzles, and duality are covered in Sections \ref{sec:Honeycombs-and-rigidity}
and \ref{sec:Puzzles-and-duality}. Locking patterns and their connection
to the enumeration of rigid tree honeycombs are also covered in Section
\ref{sec:Puzzles-and-duality}. (The term \emph{locking pattern} was
suggested by Ken Dykema before we really understood what these overlays
are, and this is why the notion is not mentioned in \cite{bcdlt}.)
In Section \ref{sec:Honeycombs-compatible-with-a}, we describe two
constructions that help calculate the exit multiplicities of honeycombs
supported in the edges of the puzzle of a given rigid honeycomb. One
of these results yields the basic inductive construction of rigid
tree honeycombs from other rigid tree honeycombs of smaller weights.
Section \ref{sec:Degeneration-of-a rigid honey} is dedicated to the
study of degenerations of a rigid honeycomb, particularly \emph{simple
degenerations.} These simple degenerations, applied to a rigid tree
honeycomb, produce either an extreme rigid honeycomb or an overlay
of two extreme rigid honeycombs. Necessary conditions are derived
on such overlays and Section \ref{sec:de-degeneration} is dedicated
to showing that these necessary conditions are sufficient as well.
The arguments are based on a new construction that starts with a puzzle
and a compatible honeycomb and produces another puzzle and compatible
honeycomb using an inflation and a partial deflation. The constructions
produce infinite sequences of rigid tree honeycombs whose weights
increase rapidly (like the powers of a number $>1$) but with a fixed
number of nonzero exit multiplicities. Some (but not all) of the infinite
families produced in \cite{blt} are shown to arise from this construction
and some new infinite families are produced.

\section{Honeycombs and rigidity\label{sec:Honeycombs-and-rigidity}}

We review the definition of a honeycomb as presented in \cite{kt}.
Fix vectors $u_{1},u_{2},u_{3}$ of length $\sqrt{3}/3$ in the plane,
arranged in counterclockwise order such that $u_{1}+u_{2}+u_{3}=0$
(see Figure \ref{fig:The-vectors-uj and wj}), and define unit vectors
\[
w_{1}=u_{3}-u_{2},\quad w_{2}=u_{1}-u_{3},\quad w_{3}=u_{2}-u_{1}.
\]
Every point $X\in\mathbb{R}^{2}$ can be written uniquely as $X=x_{1}u_{1}+x_{2}u_{2}+x_{3}u_{3}$
such that $x_{1},x_{2},x_{3}\in\mathbb{R}$ and $x_{1}+x_{2}+x_{3}=0$.
We view the numbers $x_{j}$ as the coordinates of $x$. Given a fixed
$c\in\mathbb{R}$, the equation $x_{j}=c$ represents a line parallel
to $w_{j}$.
\begin{figure}
\begin{picture}(100,100)
\put(0,0){\includegraphics{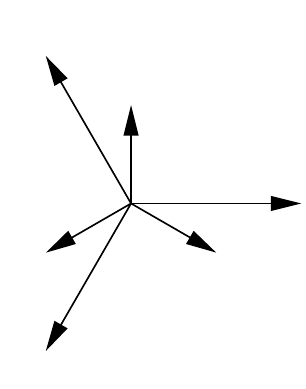}} 
\put(2,3){$w_1$} 
\put(90,52){$w_2$} 
\put(10,100){$w_3$}
\put(0,35){$u_3$}
\put(35,85){$u_2$}
\put(64,35){$u_1$}
\end{picture}%

\caption{\label{fig:The-vectors-uj and wj}The vectors $u_{j}$ and $w_{j}$}

\end{figure}
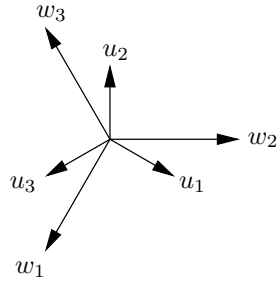

A honeycomb $\mu$ is described by two elements, namely,
\begin{enumerate}
\item the support of $\mu$. This consists of a finite union $I_{1}\cup\cdots\cup I_{k}\subset\mathbb{R}^{2}$
where
\begin{enumerate}
\item each $I_{j}$ is a closed segment, a closed ray, or a line,
\item each $I_{j}$ is parallel to one of the vectors $w_{1},w_{2},w_{3}$,
and
\item if $i\ne j$ then $I_{i}$ and $I_{j}$ have at most one point in
common that is an endpoint of both $I_{i}$ and $I_{j}$.
\end{enumerate}
\item the multiplicities of the honeycomb. These are (not necessarily integer)
numbers $\mu(I_{1}),\dots,\mu(I_{k})\in(0,+\infty)$. 
\end{enumerate}
It is convenient to consider that a honeycomb $\mu$ assigns zero
multiplicity to segments parallel to some $w_{i}$ that intersect
the support of $\mu$ in finitely many points, and that $\mu$ assigns
multiplicity $\mu(I_{j})$ to all subintervals of $I_{j}$. A \emph{branch
point} of a honeycomb is a point where at least three of the sets
$I_{j}$ meet. The segments $I_{j}$ are called \emph{edges }of\emph{
$\mu$}. The multiplicities of a honeycomb are subject to the following
\emph{balance }condition: Suppose that that $J_{1},J_{2},J_{3},J_{4},J_{5},J_{6}$
are six segments, each parallel to some $w_{i}$, containing no branch
points in their interior, having a common endpoint $X$ and arranged
clockwise around $X$. Then
\begin{equation}
\mu(J_{1})-\mu(J_{4})=\mu(J_{5})-\mu(J_{2})=\mu(J_{3})-\mu(J_{6}).\label{eq:balance condition}
\end{equation}
It may be easier to think of a honeycomb as a Borel measure in the
plane that assigns $\mu(I_{j})\times\text{length}(A)$ to each Borel
subset $A\subset I_{j}$. This allows us to add honeycombs, thus giving
the collection of all honeycombs the structure of a convex cone.

In this work, we do not distinguish between two honeycombs that are
simply translates of each other, with the translation preserving multiplicities.

Suppose that $\mu$ is a honeycomb with support $I_{1}\cup\cdots\cup I_{k}$.
It is easy to see that any endpoint of some $I_{j}$ must also be
an endpoint of another $I_{\ell}$. The support of a honeycomb $\mu$
must contain lines or rays. Each ray points in the direction of one
of the vectors $\pm w_{j}$, $j=1,2,3$. Figure \ref{fig:Some-honeycombs}
shows a few examples of honeycombs. All edge multiplicities, with
one exception indicated by a thicker line, are equal to $1$.
\begin{figure}
\includegraphics[scale=0.4]{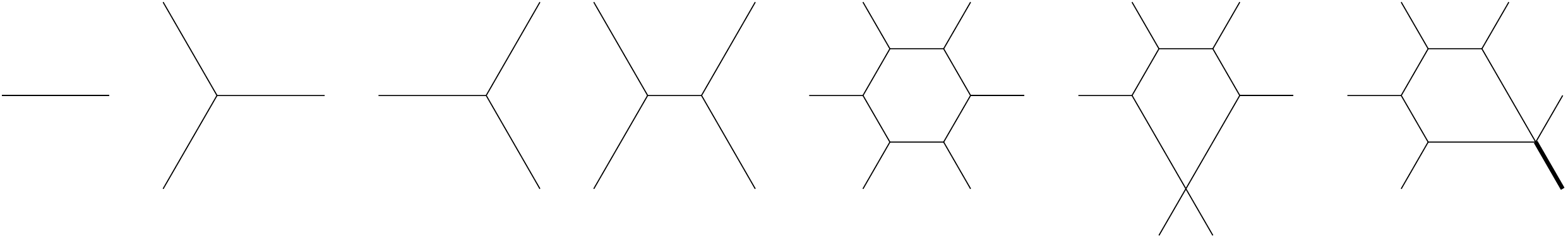}

\caption{Some honeycombs\label{fig:Some-honeycombs}}

\end{figure}

We denote by $\mathcal{M}$ (respectively, $\mathcal{M}_{*}$) the
collection of all honeycombs with the property that the rays in the
support of $\mu$ point in the direction of $w_{1},w_{2},$ or $w_{3}$
(respectively $-w_{1},-w_{2}$, or $-w_{3}$). The second and third
examples in Figure \ref{fig:Some-honeycombs} represent a honeycomb
in $\mathcal{M}$ and its mirror image (relative to a vertical line)
in $\mathcal{M}_{*}$. Suppose that $\mu$ is a honeycomb in $\mathcal{M}$,
$c\in\mathbb{R}$, and $j\in\{1,2,3\}.$ If the support of $\mu$
contains a ray $I\subset\{x_{j}=c\}$, we define
\begin{equation}
\mu^{(j)}(c)=\mu(I).\label{eq:exit densities}
\end{equation}
Otherwise, we set $\mu^{(j)}(c)=0.$ The numbers $\mu^{(j)}(c)$ are
called the \emph{exit multiplicities} of $\mu$. (The ordered collection
of the nonzero exit multiplicities will later be called the \emph{exit
pattern }of $\mu.)$ Of course, $\mu$ only has finitely many nonzero
exit multiplicities and the balance condition (\ref{eq:balance condition})
is easily seen (see, for instance, \cite{kt}) to imply the identity
\begin{equation}
\sum_{c\in\mathbb{R}}\mu^{(1)}(c)=\sum_{c\in\mathbb{R}}\mu^{(2)}(c)=\sum_{c\in\mathbb{R}}\mu^{(3)}(c).\label{eq:omega is defined}
\end{equation}
This common value is called the \emph{weight }of $\mu$ and is denoted
\begin{equation}
\omega(\mu)=\sum_{c\in\mathbb{R}}\mu^{(1)}(c).\label{eq:definition of weight}
\end{equation}
The exit multiplicities also satisfy the equality
\begin{equation}
\sum_{c\in\mathbb{R}}c(\mu^{(1)}(c)+\mu^{(2)}(c)+\mu^{(3)}(c))=0,\label{eq:trace identity}
\end{equation}
sometimes known as the \emph{trace identity} on account of its connection
with linear algebra.

A honeycomb $\mu\in\mathcal{M}$ (or $\mu\in\mathcal{M}_{*}$) is
said to be \emph{rigid} if there is no other honeycomb that has the
same exit multiplicities as $\mu$. It is easy to see that an arbitrary
honeycomb is not rigid if it has a branch point adjacent to six edges
in its support; see Figure \ref{fig:Three-honeycombs-with}.
\begin{figure}
\includegraphics{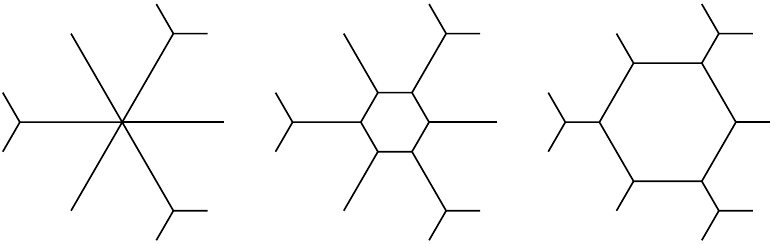}

\caption{\label{fig:Three-honeycombs-with}Three honeycombs with identical
exit multiplicities}

\end{figure}
 In fact, one can determine whether a honeycomb $\mu$ is rigid just
by looking at its support. Let $\mu$ be a honeycomb and let $A,B,C$
be branch points of $\mu$ such that the segments $AB$ and $BC$
are edges of $\mu$. We say that $ABC$ is an \emph{evil turn }\cite[p. 1592]{bcdlt}
if one of the following situations occurs:
\begin{enumerate}
\item $C=A$, and the support of $\mu$ contains edges $BX,BY,BZ$ that
are $120^{\circ},180^{\circ}$, and $240^{\circ}$ clockwise from
$AB$.
\item $BC$ is $120^{\circ}$ clockwise from $AB$.
\item $C\ne A$ and $A,B,C$ are collinear.
\item $BC$ is $120^{\circ}$ counterclockwise from $AB$ and the support
of $\mu$ contains an edge $BX$ that is $120^{\circ}$ clockwise
from $AB$.
\item $BC$ is $60^{\circ}$ counterclockwise from $AB$ and the support
of $\mu$ contains edges $BX,BY$ which are $120^{\circ}$ and $180^{\circ}$
clockwise from $AB$.
\end{enumerate}
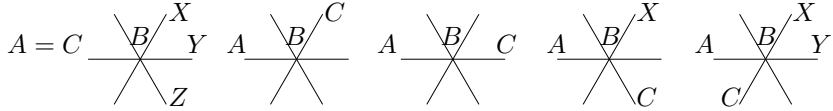
\begin{figure}

\begin{picture}(275,100)
\put(0,0){\includegraphics[scale=.8]{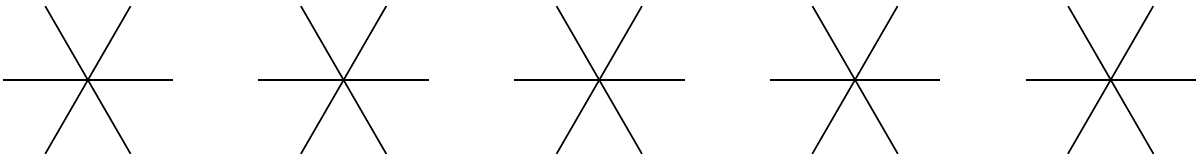}}
\put(-30,20){$A=C$}
\put(16,23){$B$}
\put(38,20){$Y$}
\put(31,32){$X$}
\put(31,0){$Z$}
\put(53,20){$A$}
\put(75,23){$B$}
\put(90,32){$C$}
\put(110,20){$A$}
\put(134,23){$B$}
\put(155,20){$C$}
\put(174,20){$A$}
\put(194,23){$B$}
\put(208,32){$X$}
\put(208,0){$C$}
\put(230,20){$A$}
\put(253,23){$B$}
\put(240,0){$C$}
\put(274,20){$Y$}
\put(267,32){$X$}
\end{picture}
\caption{\label{fig:Evil-turns}Evil turns}
\end{figure}
The possible evil turns are illustrated in Figure \ref{fig:Evil-turns}
where the edges that \emph{must }be in the support of\emph{ $\mu$}
are labeled, but the support may contain all six edges incident to
$B$. A sequence $A_{1},A_{2},\dots,A_{n},A_{n+1}=A_{1}$ of branch
points of $\mu$ is called an \emph{evil loop} if $A_{j-1}A_{j}A_{j+1}$
is an evil turn for every $j=1,\dots,n$. Each of the first three
configurations in Figure \ref{fig:Some-evil-loops} contains exactly
one evil loop (up to cyclic permutations) but the fourth one contains
none. 
\begin{figure}
\includegraphics[scale=0.8]{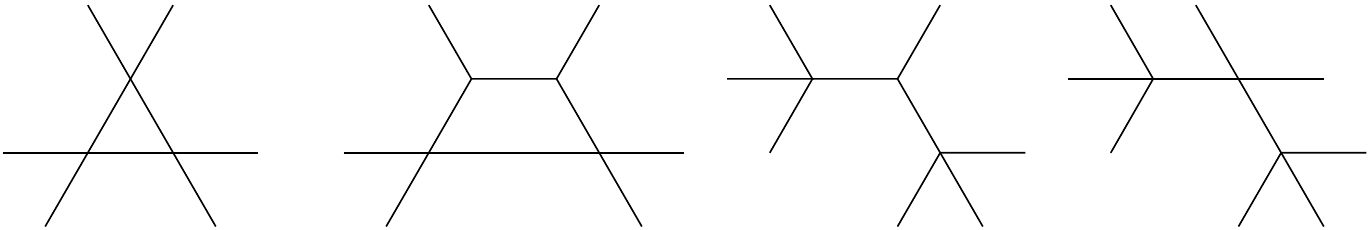}

\caption{Some loops\label{fig:Some-evil-loops}}

\end{figure}

The following result is proved as in \cite[Proposition 2.2]{bcdlt}.
(The proof uses a characterization of rigidity via puzzles; see \cite[Theorems 4 and 7]{KTW}.)
\begin{prop}
\label{prop:rigid in terms of evil}A honeycomb $\mu\in\mathcal{M}$
is rigid if and only if its support contains none of the following
configurations:
\begin{enumerate}
\item Six edges meeting in one branch point.
\item An evil loop.
\end{enumerate}
\end{prop}

Suppose that $e=AB$ and $f=BC$ are edges of a rigid honeycomb $\mu$.
As in \cite{bcdlt}, we write $e\to_{\mu}f$ (or $e\to f$ when $\mu$
is understood; see Figure \ref{fig:The-relation-=00005Cto}) if one
of the following two situations arises:
\begin{enumerate}
\item $\varangle ABC=120^{\circ}$ and the segment $BX$ opposite $AB$
is not in the support of $\mu$.
\item $e$ and $f$ are opposite and there exists a segment $BX$ not in
the support of $\mu$ such that $\varangle CBX=60^{\circ}$.
\end{enumerate}
\begin{figure}
\begin{picture}(100,100)
\put(0,0){\includegraphics{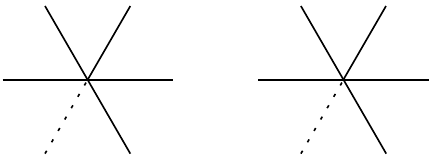}}
\put(13,23){$B$}
\put(30,44){$A$}
\put(-3,23){$C$}
\put(5,0){$X$}
\put(86,23){$B$}
\put(68,23){$C$}
\put(120,23){$A$}
\put(80,0){$X$}
\end{picture}

\caption{\label{fig:The-relation-=00005Cto}The relation `$\to$'}

\end{figure}
The balance condition shows that $\mu(e)\le\mu(f)$ whenever $e\to f$.
In fact, $\mu(f)$ equals $\mu(e)$ if $e\to f$ and $f\to e$, and
it equals $\mu(e)+\mu(e')$ when $e\to f$, $e'\to f$, and $e\ne e'$.
The transitive closure $\Rightarrow$ of the relation $\to$ was called
\emph{descendant} in \cite{bcdlt}. Thus, $e\Rightarrow f$ if there
exists a \emph{descendance path}, that is, a sequence $e_{1}\dots e_{n}$
of edges such that $e_{1}=e,e_{n}=f,$ and $e_{j}\to e_{j+1}$, $j=1,\dots,n-1$.
Given such a path, $e$ is called an \emph{ancestor} of $f$ and $f$
is called a \emph{descendant} of $e$. We write $e\Leftrightarrow f$,
and we say that $e$ and $f$ are \emph{equivalent}, if either $e=f$
or both $e\Rightarrow f$ and $f\Rightarrow e$ hold. An edge $e$
of $\mu$ is said to be a \emph{root edge} if the relation $f\Rightarrow e$
implies $e\Leftrightarrow f$. The above observations show that the
multiplicity of an arbitrary edge of $\mu$ can be written as a linear
combination with positive integer coefficients of the multiplicities
of root edges. More precisely, suppose that $\mu$ is a rigid honeycomb
and that $e_{1},\dots,e_{n}$ is a maximal sequence of pairwise inequivalent
root edges. Then an arbitrary edge $f$ of $\mu$ satisfies
\[
\mu(f)=\sum_{j=1}^{n}d_{j}\mu(e_{j}),
\]
where $d_{j}\ge0$ is the number of distinct descendance paths from
$e_{j}$ to $f$. Another way of formulating this is to say that
\[
\mu=\sum_{j=1}^{n}\mu(e_{j})\mu_{j},
\]
where each $\mu_{j}$ is itself a rigid honeycomb (see \cite[Section 3]{bcdlt}).
The summands $\mu_{j}$ belong to extreme rays in the cone $\mathcal{M}$.
Extreme rigid honeycombs are characterized by the fact that they have
a unique equivalence class of root edges. An extreme rigid honeycomb
is a \emph{tree honeycomb} precisely when the multiplicity assigned
to its root edges is equal to $1$. Alternatively, rigid tree honeycombs
are obtained as immersions of a special kind of tree that we now define.

The trees relevant to the construction of rigid honeycombs are finite
trees whose edges are labeled $1,2,$ or $3$, and the following conditions
are satisfied:
\begin{enumerate}
\item each vertex has order $1$ or $3$, and
\item the three edges adjacent to a vertex of order $3$ have distinct labels.
\end{enumerate}
An \emph{immersion }of a tree is a map $f$ that associates to each
vertex $v$ of order 3 of a tree a point $f(v)$ in the plane such
that the following conditions are satisfied:
\begin{enumerate}
\item [(i)]if $a$ and $b$ are joined by an edge labeled $j$, then $f(a)\ne f(b)$
and the segment $f(a)f(b)$ is parallel to $w_{j}$, and
\item [(ii)] if $xa,xb$, and $xc$ are three edges adjacent to a vertex
$x$ then the angle between any two of the segments $f(x)f(a),f(x)f(b)$,
and $f(x)f(c)$ equals $120^{\circ}$.
\end{enumerate}
Given a tree $T$ that has at least two vertices of order $3$, and
an immersion $f$ of $T$, we define a honeycomb $\mu_{f}$ as follows.
Let $ab$ be an edge of $T$. If $a$ and $b$ have order $3$, we
add the segment $f(a)f(b)$ to the support of $\mu_{f}$ and we add
$1$ to the multiplicity of this segment. If $a$ is of order $1$,
$ab$ is labeled $j$, and $c,d$ are the other vertices connected
to $b$, we add to the support of $\mu_{f}$ a ray $h$ with endpoint
$f(b)$ and parallel to $w_{j}$ and we add $1$ to the multiplicity
of this ray. The ray is chosen such that the angles between any two
of $h,f(b)f(c)$, and $f(b)f(d)$ equal $120^{\circ}$. The segments
$f(a)f(b)$ and the rays constructed above can be viewed as images
under the immersion $f$ of the edges of $T$. The honeycomb $\mu_{f}$
is simply determined by adding all of the multiplicities just defined.
If $T$ has exactly one vertex of order $3$, there are two possible
honeycombs associated to an immersion $f$, one in $\mathcal{M}$
and the other in $\mathcal{M}_{*}$. The honeycombs obtained from
immersions of trees are called \emph{tree honeycombs.}

\begin{figure}
\begin{picture}(250,100)
\put(0,0){\includegraphics[scale=.5]{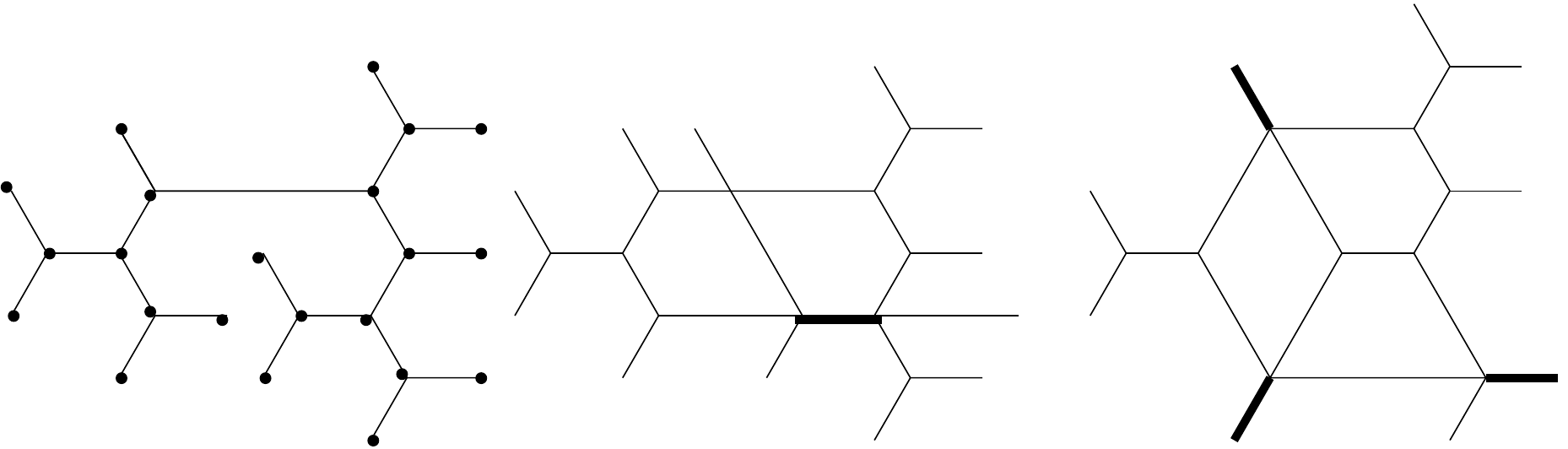}}
\put(12,35){$2$}
\put(73,35){$2$}
\put(73,56){$2$}
\put(73,13){$2$}
\put(30,24){$2$}
\put(55,24){$2$}
\put(42,46){$2$}
\put(6,23){$1$}
\put(68,23){$1$}
\put(68,45){$1$}
\put(68,1){$1$}
\put(49,12){$1$}
\put(25,12){$1$}
\put(25,33){$1$}
\put(17,45){$3$}
\put(-1,35){$3$}
\put(60,35){$3$}
\put(60,13){$3$}
\put(60,56){$3$}
\put(43,23){$3$}
\put(18,23){$3$}
\end{picture}

\caption{A tree and two of its immersions\label{fig:A-tree-and-immersions-1}}
\end{figure}
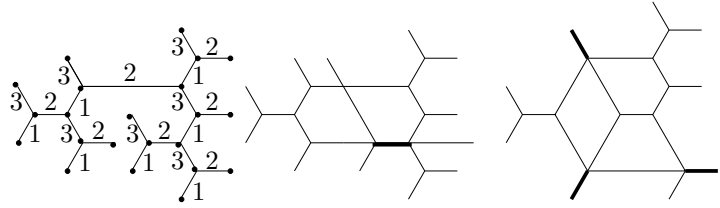

Not every tree honeycomb is rigid. Figure \ref{fig:A-tree-and-immersions-1}
represents a tree $T$ and two immersions of $T$, the second of which
yields a rigid honeycomb.

Suppose that $\mu\in\mathcal{M}$ is a rigid honeycomb with a unique
equivalence class of root edges, all of which have multiplicity $1$.
We construct a tree $T_{0}$ as follows. Fix a root edge $e_{0}=XY$
of $\mu$. The vertices of $T_{0}$ are $x,y$, and the descendance
paths $e_{0}\dots e_{n}$. Given such a descendance path, the vertices
$v=e_{0}\dots e_{n-1}$ and $v'=e_{0}\dots e_{n}$ are joined by an
edge of $T_{0}$ labeled $j$ if $e_{n}$ is parallel to $w_{j}$.
In addition, $x$ is joined to $y$ by an edge labeled $j$ if $XY$
is parallel to $w_{j}$, $x$ is joined to $e_{0}e_{1}$ if $X$ is
a an endpoint of $e_{1}$, and $y$ is joined to $e_{0}e_{1}$ if
$Y$ an endpoint of $e_{1}$. The same labeling rule applies to the
last two kinds of edges. The tree $T_{0}$ obtained this way may have
some vertices of order $2$. To obtain a tree $T$ satisfying condition
(1) we simply remove the vertices of order $2$ as follows: suppose
that $v_{0}v_{1}\dots v_{n}$ is a path in $T_{0}$ such that $v_{0}$
and $v_{n}$ have order $1$ or $3$ and $v_{1},\dots,v_{n-1}$ have
order $2$. Then $v_{1},\dots,v_{n-1}$ are removed from $T_{0}$
and an edge joining $v_{0}$ and $v_{n}$, carrying the label of $v_{0}v_{1}$,
is added instead. Clearly, we have $\mu=\mu_{f}$ for some immersion
$f$ of $T$. The construction of $T_{0}$ and $T$ is illustrated
in Figure \ref{fig:The-construction-ofT} for a particular rigid tree
honeycomb.
\begin{figure}
\begin{picture}(250,100)
\put(0,0){\includegraphics[scale=.7]{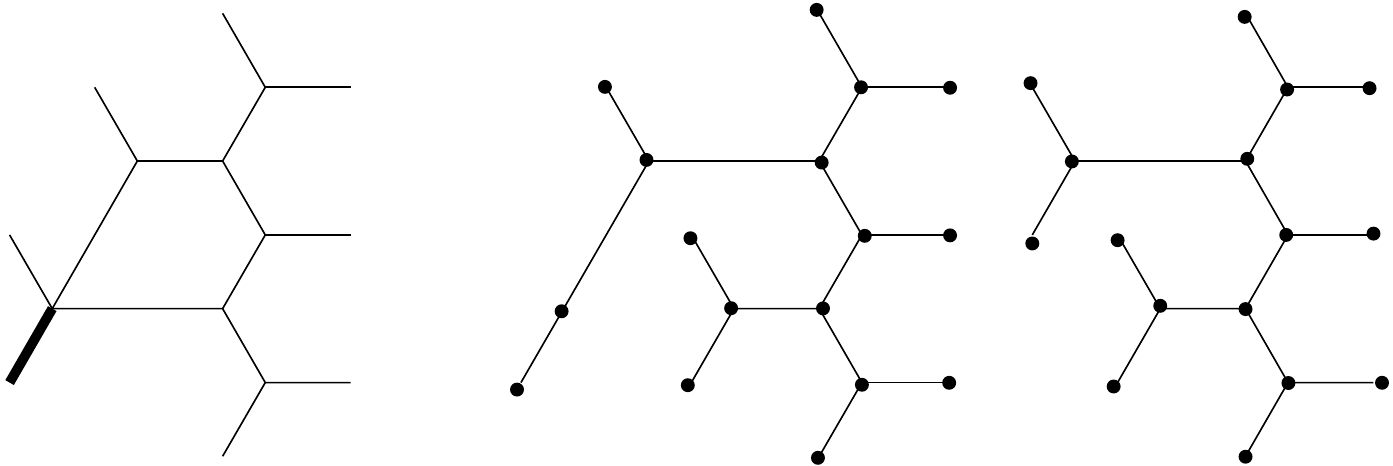}}
\put(17,59){$X$}
\put(49,59){$Y$}
\put(121,59){$x$}
\put(169,59){$y$}
\put(208,59){$x$}
\put(256,59){$y$}

\put(145,63){$2$}
\put(232,63){$2$}
\put(181,78){$2$}
\put(267,78){$2$}
\put(181,49){$2$}
\put(267,49){$2$}
\put(181,18){$2$}
\put(267,18){$2$}
\put(155,33){$2$}
\put(241,33){$2$}

\put(128,49){$1$}
\put(214,49){$1$}
\put(112,20){$1$}
\put(146,20){$1$}
\put(232,20){$1$}
\put(172,35){$1$}
\put(258,35){$1$}
\put(249,67){$1$}
\put(258,6){$1$}
\put(163,67){$1$}
\put(172,6){$1$}

\put(117,67){$3$}
\put(203,67){$3$}
\put(137,35){$3$}
\put(223,35){$3$}
\put(163,49){$3$}
\put(163,20){$3$}
\put(163,78){$3$}
\put(249,49){$3$}
\put(249,20){$3$}
\put(249,78){$3$}
\end{picture}

\caption{\label{fig:The-construction-ofT}The construction of $T_{0}$ and
$T$}

\end{figure}
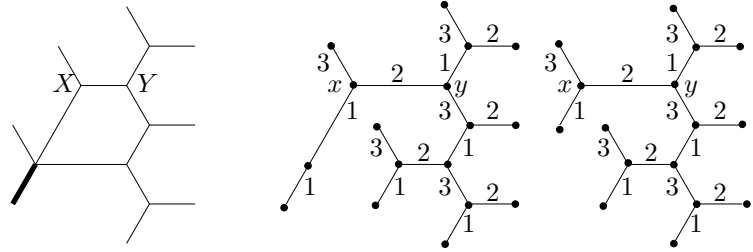

Suppose that $T$ is a tree, that $f$ is an immersion of $T$, and
that $\mu=\mu_{f}$ belongs to $\mathcal{M}$. Suppose also that there
exists an edge $xy$ of $T$ such that $\mu$ assigns unit multiplicity
to part of $f(x)f(y)$. Orient the edges of $T$ other than $xy$
away from $xy$. We say that the immersion $f$ is \emph{coherent}
if the following condition is satisfied: if a segment $I$ is contained
in the images under $f$ of two different edges of $T$, then the
two orientations induced on $I$ are the same. The following result
is \cite[Theorem 4.1]{blTR}.
\begin{thm}
\label{thm:rigid immersion} Suppose that $T$ is a tree, that $f$
is an immersion of $T$, and that $\mu=\mu_{f}\in\mathcal{M}$. Then
$\mu$ is rigid if and only if the following three conditions are
satisfied.
\begin{enumerate}
\item There exists an edge $xy$ of $T$ such that $\mu$ assigns multiplicity
$1$ to part of $f(x)f(y)$.
\item The immersion $f$ is coherent.
\item There is no branch point $X$ of $\mu$ such that four segment $XA_{1},XA_{2},XB_{1},$
and $XB_{2}$ have nonzero multiplicity, the points $A_{1}XA_{2}$
are collinear, and the points $B_{1}XB_{2}$ are collinear; see Figure
\emph{\ref{fig:A-crossing}}.
\end{enumerate}
\end{thm}

\begin{figure}
\begin{picture}(80,100)
\put(0,0){\includegraphics[scale=.35]{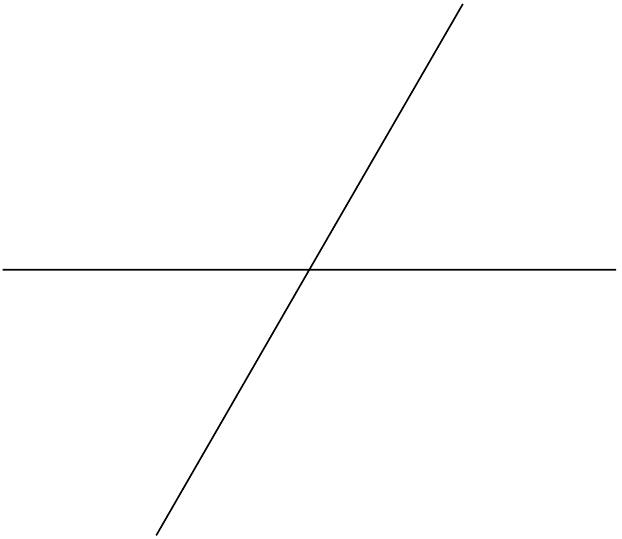}}
\put(-11,25){$A_1$}
\put(64,25){$A_2$}
\put(24,30){$X$}
\put(17,-3){$B_1$}
\put(46,55){$B_2$}

\end{picture}

\caption{\label{fig:A-crossing}A crossing}

\end{figure}
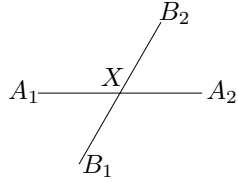
Having obtained a description of rigid tree honeycombs, it is important
to determine which linear combinations $c_{1}\mu_{1}+\cdots+c_{n}\mu_{n}$,
$c_{1},\dots,c_{n}>0$, of rigid tree honeycombs are rigid. Such combinations
are called \emph{rigid overlays} and their description involves a
bilinear map \cite{blt} defined on honeycombs in $\mathcal{M}$ as
follows: given $\mu,\nu\in\mathcal{M}$, we set
\begin{equation}
\Sigma_{\nu}(\mu)=\sum_{c<d}[\mu^{(1)}(c)\nu^{(1)}(d)+\mu^{(2)}(c)\nu^{(2)}(d)+\mu^{(3)}(c)\nu^{(3)}(d)]-\omega(\mu)\omega(\nu),\label{eq:the function sigma}
\end{equation}
where $\mu^{(j)}(c)$ is as defined above (\ref{eq:exit densities}).
It was seen in \cite[end of Section 2]{blt} that 
\begin{equation}
\Sigma_{\mu}(\mu)=\frac{1}{2}\left[\omega(\mu)^{2}-\sum_{c\in\mathbb{R}}[\mu^{(1)}(c)^{2}+\mu^{(2)}(c)^{2}+\mu^{(3)}(c)^{2}]\right].\label{eq:sigms(mu,mu)}
\end{equation}
The following result is \cite[Theorem 4.2]{blt}.
\begin{thm}
\label{thm:sigma(tree,tree)}Let $\mu$ and $\nu$ be two distinct
rigid tree honeycombs in $\mathcal{M}$. Then $\Sigma_{\nu}(\mu)\ge0$
and $\Sigma_{\nu}(\nu)=-1$.
\end{thm}

It is possible that $\Sigma_{\nu}(\mu)>0$ and $\Sigma_{\mu}(\nu)>0$
for two rigid tree honeycombs $\mu$ and $\nu$, see Figure \ref{fig:positive-sigmas}.
\begin{figure}
\includegraphics[scale=0.8]{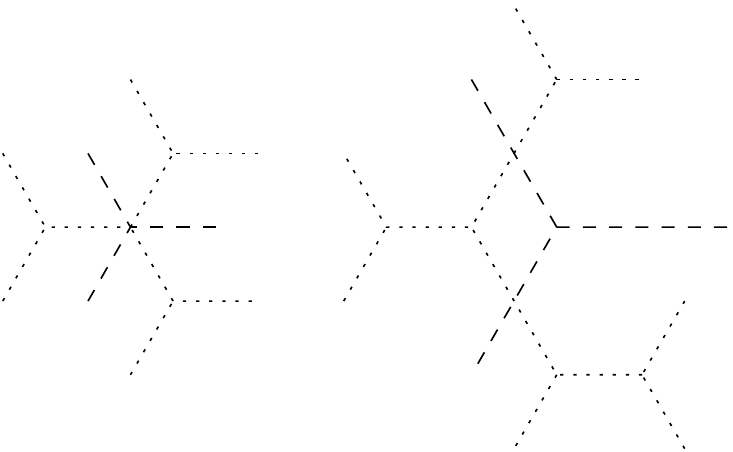}

\caption{\label{fig:positive-sigmas}$\Sigma_{\mu}(\nu)>0$ and $\Sigma_{\nu}(\mu)>0$}

\end{figure}
 It follows from (\ref{eq:sigms(mu,mu)}) that the second assertion
in Theorem \ref{thm:sigma(tree,tree)} can be written as
\begin{equation}
\sum_{c\in\mathbb{R}}[\mu^{(1)}(c)^{2}+\mu^{(2)}(c)^{2}+\mu^{(3)}(c)^{2}]=\omega(\mu)^{2}+2.\label{eq:omega square +2}
\end{equation}

\begin{rem}
\label{rem:calculation of Sigma} Suppose that $\mu$ and $\nu$ are
distinct rigid honeycombs in $\mathcal{M}$ obtained from immersions
$f$ of $T$ and $g$ of $S$, respectively. Then \cite[Theorem 5.1]{blt}
provides a method for the calculation of $\Sigma_{\nu}(\mu)$. Fix
an edge $e$ in $S$ that is mapped by $g$ onto a root edge of $\nu$
that is not contained in the support of $\mu$. Orient all the other
edges of $S$ away from $e$. Then $\Sigma_{\nu}(\mu)$ is the total
number of occurrences of events of the four types indicated in the
Figure \ref{fig:Calculation-of-sigma}.
\begin{figure}
\begin{picture}(250,100)
\put(0,0){\includegraphics[scale=.60]{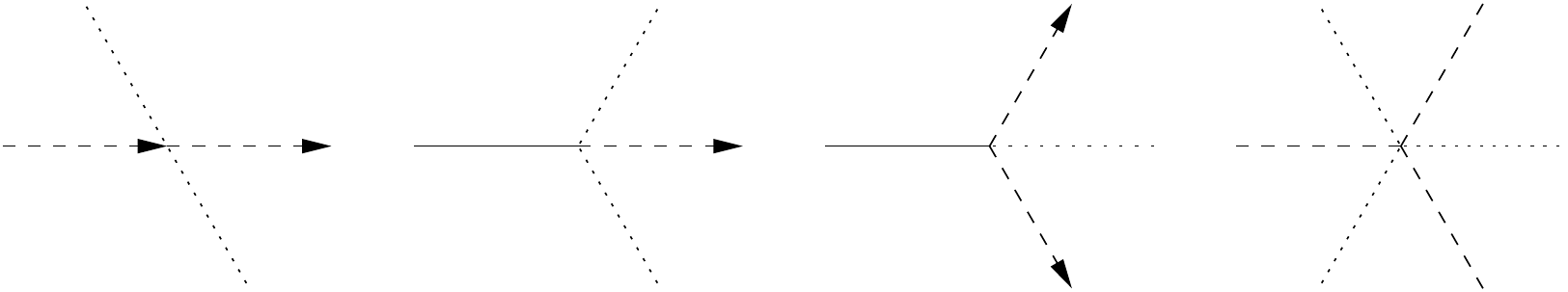}}
\put(-10,25){$\mu$}
\put(10,52){$\nu$}

\put(66,25){$\mu$}
\put(119,52){$\nu$}

\put(194,52){$\mu$}
\put(209,25){$\nu$}

\put(268,52){$\mu$}
\put(285,25){$\nu$}
\end{picture}

\caption{\label{fig:Calculation-of-sigma}Calculation of $\Sigma_{\nu}(\mu)$}
\end{figure}
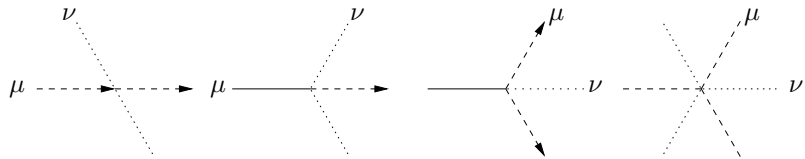
 In these diagrams, the dashed lines represent the images under $g$
of edges of $S$, the dotted lines represent the images under $f$
of edges of $T$, and the solid lines represent common edges of $\mu$
and $\nu$. (Note that, in counting these events we consider edges
of the trees rather than edges of the honeycombs themselves. Thus,
when looking at edges of the honeycombs, one must take into account
their multiplicities.) 
\begin{figure}
\begin{picture}(80,100)
\put(0,0){\includegraphics[scale=1]{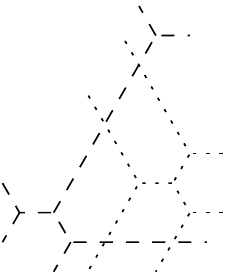}}
\put(-10,25){$\mu$}
\put(70,32){$\nu$}
\end{picture}\caption{$\mu$ is clockwise from $\nu$\label{fig:clockwise overlay}}

\end{figure}
\end{rem}

The following characterization of rigid overlays is \cite[Corollary 4.5]{blt}
(see also \cite[Theorem 3.8]{bcdlt}).
\begin{thm}
\label{thm:rigid overlays}Let $\mu_{1},\dots,\mu_{n}$ be distinct
rigid tree honeycombs and let $c_{1},\dots,c_{n}$ be positive numbers.
Then the following are equivalent:
\begin{enumerate}
\item The honeycomb $\sum_{j=1}^{n}c_{j}\mu_{j}$ is rigid.
\item There exists a permutation $\nu_{1},\dots,\nu_{n}$ of $\mu_{1},\dots,\mu_{n}$
such that $\Sigma_{\nu_{i}}(\nu_{j})=0$ for $i>j$.
\end{enumerate}
\end{thm}

\section{Puzzles and duality \label{sec:Puzzles-and-duality}}

Suppose that $\mu\in\mathcal{M}$. The \emph{puzzle} of $\mu$ is
obtained by a process of \emph{inflation} that we describe next. Cut
the plane along the edges of $\mu$ to obtain a collection of \emph{white
}puzzle pieces and translate these pieces away from each other in
such a way that, given an edge $AB$ of $\mu$, the parallelogram
formed by the two translates of $AB$ has two sides of length $\mu(AB)$
that are $60^{\circ}$ clockwise from $AB$. We call this parallelogram
the \emph{inflation }of\emph{ $AB$} and we color it \emph{dark gray}.
The balance condition for $\mu$ implies that these white and dark
gray pieces fit together and leave a space corresponding to each branch
point of $\mu$. These extra spaces are colored \emph{light gray}.
The light gray piece associated to a branch point $X$ of $\mu$ is
a convex polygon with as many sides as there are edges incident to
$X$. One can recover the honeycomb $\mu$, up to a translation, from
its puzzle by \emph{deflation}. This process consists simply of moving
the white pieces together and replacing each dark gray parallelogram
(or half-infinite strip) by a segment parallel to its white sides
and with multiplicity equal to the length of its light gray side.
There are similar processes of $*$\emph{-inflation} and\emph{ $*$}-\emph{deflation
}which are the same as inflation and deflation except that `clockwise'
is replaced by `counterclockwise' in the construction of the $*$-inflation.

Starting with a nonzero honeycomb $\mu\in\mathcal{M}$, one can construct
the puzzle of $\mu$ and apply $*$-deflation to this puzzle. The
only difficulty is that the segments arising from dark gray infinite
strips would be assigned infinite multiplicity. To overcome this difficulty,
choose a closed equilateral triangle $\Delta$ containing all the
branch points of $\mu$ whose sides are $60^{\circ}$ clockwise from
those rays in the support of $\mu$ that they intersect. Apply now
the process of inflation only to the part of the support of $\mu$
that is contained in $\Delta$ and apply $*$-deflation to the resulting
puzzle. The result of these operations is the part of a honeycomb
$\mu_{\Delta}^{*}\in\mathcal{M}_{*}$ contained in an equilateral
triangle $\Delta_{*}$ with sides equal to $\omega(\mu)$. (A simple
example is illustrated in Figure \ref{fig:honey and dual}.) 
\begin{figure}
\includegraphics{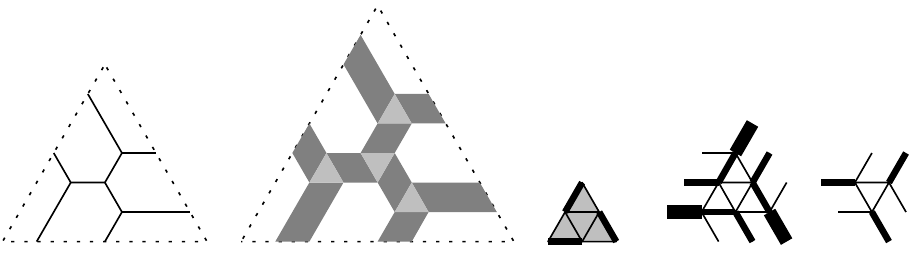}

\caption{\label{fig:honey and dual}A honeycomb $\mu$ in $\Delta$, its puzzle,
the $*$-deflation, and the honeycombs $\mu_{\Delta}^{*}$ and $\mu^{*}$}

\end{figure}
Choosing a different triangle $\Delta$ changes the multiplicities
of $\mu_{\Delta}^{*}$ only on the boundary of $\Delta_{*}$. Denote
by $\Delta_{0}$ the smallest triangle satisfying the requirements
above. We define $\mu^{*}$ to be equal to $\mu_{\Delta_{0}}^{*}$.
The honeycomb $\mu^{*}$ has the property that each side of $\Delta_{*}$
contains a segment with zero multiplicity.

In general, we observe that the sides of $\Delta_{*}$ are divided
by the rays in the support of $\mu^{*}$ into segments with lengths
equal to the exit multiplicities of $\mu$. (The triangle $\Delta_{*}$
is partitioned by the support of $\mu^{*}$ into light gray pieces
that form the \emph{degeneracy graph} of $\mu$ introduced in \cite{kt}.)
If $\mu$ has only one branch point, we define $\mu^{*}=0$.
\begin{rem}
\label{rem:stuff about mu versus mu*} As just observed, the nonzero
exit multiplicities of a honeycomb $\mu$ are precisely the lengths
of the segments on $\partial\Delta_{*}$ between consecutive exit
rays of $\mu^{*}$. For instance, suppose that the support of $\mu^{*}$
is contained in the support of $\nu^{*.}$ If ${\rm exit(\mu)={\rm exit}(\nu)},$then
$\mu^{*}$ and $\nu^{*}$ share their exit rays, and therefore $\mu$
and $\nu$ have precisely the same nonzero exit multiplicities. On
the other hand, if ${\rm exit}(\mu)={\rm exit}(\nu)-1$, then $\mu^{*}$
has one fewer exit ray, and thus $\mu$ has the same nonzero exit
multiplicities, except that the two nonzero exit multiplicities of
$\nu$ corresponding to segments separated by the extra exit ray are
replaced by their sum. (If ${\rm exit}(\mu)$ is even smaller, more
of the multiplicities of $\nu$ are aggregated in this manner into
multiplicities of $\mu$.)
\end{rem}

It is easily seen that a honeycomb $\mu\in\mathcal{M}$ is rigid if
and only if $\mu^{*}\in\mathcal{M}_{*}$ is rigid (see \cite[p. 1591]{bcdlt}).
Given a rigid honeycomb $\mu\in\mathcal{M}$, we denote by ${\rm exit}(\mu)$
the number of nonzero exit multiplicities of $\mu$, and we denote
by ${\rm root}(\mu)$ the number of equivalence classes of root edges
in the support of $\mu$. As noted earlier, $\mu$ is a linear combination
with positive coefficients of ${\rm root}(\mu)$ distinct rigid tree
honeycombs. The following result is equivalent to \cite[Theorem 1.1]{ab}.
(To see the equivalence, observe that the number of \emph{attachment
points} of $\mu$ relative to the minimal triangle $\Delta_{0}$ considered
above is equal to ${\rm exit}(\mu)-3$.)
\begin{thm}
\label{thm:exit(mu)=00003Droot(mu) etc} For every rigid honeycomb
$\mu\in\mathcal{M}$ we have
\[
{\rm root}(\mu)+{\rm root}(\mu^{*})={\rm exit}(\mu)-2.
\]
\end{thm}

In order to study the exit multiplicities of a rigid tree honeycomb,
we require an auxiliary result of convex analysis.
\begin{lem}
\label{lem:affine map}Let $K_{1}$ and $K_{2}$ be two convex sets
such that $K_{2}$ is finite-dimensional, and let $f:K_{1}\to K_{2}$
be a surjective affine map. Suppose that there exists an interior
point $b_{0}\in K_{2}$ such that $f^{-1}(b_{0})$ is a singleton.
Then $f$ is injective.
\end{lem}

\begin{proof}
Let $a_{0}\in K_{1}$ satisfy $f(a_{0})=b_{0}$ and suppose that $a'_{1},a_{1}''\in K_{1}$
are such that $f(a_{1}')=f(a_{1}'')=b_{1}$. Since $b_{0}$ is an
interior point of $K_{2}$, there exist $b_{2}\in K_{2}$ and $t\in(0,1)$
such that $b_{0}=tb_{1}+(1-t)b_{2}$. Choose $a_{2}\in f^{-1}(b_{2})$.
Since $f$ is affine, we deduce that
\begin{align*}
f(ta_{1}'+(1-t)a_{2}) & =tb_{1}+(1-t)b_{2}=b_{0},\\
f(ta_{1}''+(1-t)a_{2}) & =tb_{1}+(1-t)b_{2}=b_{0},
\end{align*}
and thus
\[
ta_{1}'+(1-t)a_{2}=ta_{1}''+(1-t)a_{2}=a_{0}.
\]
It follows immediately that $a'_{1}=a_{1}''$.
\end{proof}
Fix a positive number $r$ and consider an equilateral triangle $\Delta=ABC$
with side length $r$ and sides parallel to $w_{1},w_{2},$ and $w_{3}$.
For simplicity, suppose that $A,B,$ and $C$ are arranged counterclockwise
around $\Delta$ and $AB$ is parallel to $w_{1}$. 
\begin{figure}
\begin{picture}(80,100)
\put(0,0){\includegraphics[scale=1.2]{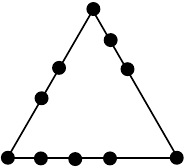}}
\put(5,35){$A_1$}
\put(0,25){$A_2$}
\put(8,-9){$B_1$}
\put(21,-9){$B_2$}
\put(34,-9){$B_3$}
\put(48,30){$C_1$}
\put(43,42){$C_2$}
\end{picture}

\caption{\label{fig:a pattern}The pattern $2,1,2|1,1,1,2|3,1,1$}

\end{figure}
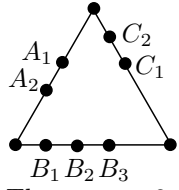
 A \emph{pattern }consists of a sequence $A_{0}=A,A_{1},\dots,A_{k_{1}-1},A_{k_{1}}=B=B_{0},B_{1},\dots,B_{k_{2}-1},B_{k_{2}}=C=C_{0},C_{1},\dots,C_{k_{3}-1},C_{k_{3}}=A$
of points arranged counterclockwise on the boundary of $\Delta$;
see Figure \ref{fig:a pattern}. Alternatively, a pattern is determined
by the distances $\alpha_{1}^{(1)},\dots,\alpha_{k_{1}}^{(1)}|\alpha_{1}^{(2)},\dots,\alpha_{k_{2}}^{(2)}|\alpha_{1}^{(3)},\dots,\alpha_{k_{3}}^{(3)}$
defined by 
\[
\alpha_{j}^{(1)}=\text{length}(A_{j-1},A_{j}),\ \alpha_{j}^{(2)}=\text{length}(B_{j-1},B_{j}),\ \alpha_{j}^{(3)}=\text{length}(C_{j-1},C_{j}).
\]
These positive numbers are subject to
\[
\sum_{j=1}^{k_{1}}\alpha_{j}^{(1)}=\sum_{j=1}^{k_{2}}\alpha_{j}^{(2)}=\sum_{j=1}^{k_{3}}\alpha_{j}^{(3)}=r.
\]

\begin{defn}
We say that a pattern is a \emph{locking pattern} if the following
condition is satisfied: Suppose that $\mu\in\mathcal{M}$ is such
that the branch points of $\mu$ are contained in $\Delta$ and that
the boundary of $\Delta$ only intersects the rays in the support
of $\mu$ in a subset of $\{A_{j}\}_{j=0}^{k_{1}}\cup\{B_{j}\}_{j=0}^{k_{2}}\cup\{C_{j}\}_{j=0}^{k_{3}}.$
Then $\mu$ is rigid.
\end{defn}

It is easy to see that $\mathcal{M}$ can be replaced by $\mathcal{M}_{*}$
in the above definition. Moreover, the above condition need only be
satisfied for honeycombs $\mu\in\mathcal{M}$ with rays intersecting
$\partial\Delta$ in a subset of $\{A_{j}\}_{j=0}^{k_{1}-1}\cup\{B_{j}\}_{j=0}^{k_{2}-1}\cup\{C_{j}\}_{j=0}^{k_{3}-1}$;
such honeycombs are allowed to have a ray meeting $\partial\Delta$
in $A_{0}(=C_{k_{3}})$ that is parallel to the other rays passing
through some $A_{j}$ with $j\ne0$ but not such a ray that is parallel
to the ones passing through some $C_{j}$ for $j\ne k_{3}.$ A similar
observation (with $A_{0},B_{0},C_{0}$ in place of $A_{k_{1}},B_{k_{2}},C_{k_{3}}$)
applies to honeycombs in $\mathcal{M}_{*}$. It is also obvious that
a pattern that has a locking refinement must itself be locking.

Suppose now that $\mu\in\mathcal{M}$ and the only nonzero exit multiplicities
of $\mu$ are $\alpha_{j}^{(1)}=\mu^{(1)}(c_{j}^{(1)})$, $\alpha_{j}^{(2)}=\mu^{(2)}(c_{j}^{(2)})$,
and $\alpha_{j}^{(3)}=\mu^{(3)}(c_{j}^{3})$, where $c_{1}^{(1)}<\cdots<c_{k_{1}}^{(1)}$,
$c_{1}^{(2)}<\cdots<c_{k_{2}}^{(2)}$, and $c_{1}^{(3)}<\cdots<c_{k_{3}}^{(3)}$.
These numbers define a pattern with $r=\omega(\mu)$ that we call
the \emph{exit pattern} of $\mu$.
\begin{prop}
\label{prop:locking pattern} Suppose that $\mu\in\mathcal{M}$ is
a rigid tree honeycomb. Then the exit  pattern of $\mu$ is a locking
pattern.
\end{prop}

\begin{proof}
Choose rigid tree honeycombs $\rho_{1},\dots,\rho_{n}\in\mathcal{M}_{*}$
and $c_{1},\dots,c_{n}>0$ such that $\mu^{*}=\sum_{j=1}^{n}c_{j}\rho_{j}$.
Theorem \ref{thm:exit(mu)=00003Droot(mu) etc} and the hypothesis
imply that $n=\text{exit}(\mu)-3$. Recall that all the branch points
of $\mu^{*}$ are contained in a triangle $\Delta_{*}$ with side
lengths $\omega(\mu)$ and that the rays of $\mu^{*}$ determine precisely
the exit  pattern of $\mu$ on the boundary of $\Delta_{*}$. Denote
by $K_{1}$ the collection of all honeycombs in $\mathcal{M}_{*}$
whose branch points are contained in $\Delta_{*}$ and whose rays
intersect the boundary of $\Delta_{*}$ in a subset of $\{A_{j}\}_{j=1}^{k_{1}}\cup\{B_{j}\}_{j=1}^{k_{2}}\cup\{C_{j}\}_{j=1}^{k_{3}}.$
Observe that $k_{1}+k_{2}+k_{3}={\rm exit}(\mu)$. Given $\nu\in K_{1}$,
denote by $f(\nu)\in\mathbb{R}^{{\rm exit}(\mu)}$ the vector $(a_{1},\dots,a_{k_{1}},b_{1},\dots,b_{k_{2}},c_{1},\dots,c_{k_{3}})$,
where $a_{j}$ is the exit multiplicity that $\nu$ assigns to the
ray passing through $A_{j}$, and $b_{j},c_{j}$ are defined similarly.
Denote by $K_{2}\subset\mathbb{R}^{{\rm exit}(\mu)}$ the range of
$f$. In order to apply Lemma \ref{lem:affine map}, we verify that
$f(\mu^{*})$ is an interior point of $K_{2}$. First, we note that
the dimension of $K_{2}$ is at most ${\rm exit}(\mu)-3$ because
every honeycomb must satisfy (\ref{eq:omega is defined}) and (\ref{eq:trace identity}).
On the other hand, $f$ is injective on the set $S=\{\sum_{j=1}^{n}t_{j}\rho_{j}:t_{1},\dots,t_{n}>0\}$
and thus $f(K_{1})$ has dimension at least ${\rm exit}(\mu)-3$.
It follows that $K_{2}$ has dimension exactly ${\rm exit}(\mu)-3$
and that each point in $f(S)$ is an interior point of $K_{2}$. The
proposition follows from Lemma \ref{lem:affine map}.
\end{proof}
The multiplicity pattern of a rigid tree honeycomb has an additional
property that helps determine whether a given pattern arises from
such a honeycomb. Suppose that an arbitrary pattern $A_{0},\dots,A_{k_{1}},B_{0},\dots,B_{k_{2}},C_{0},\dots C_{k_{3}}$
is given on a triangle $\Delta$ with sides parallel to $w_{1},w_{2},$
and $w_{3}$. A convex region $\Omega\subset\Delta$ is said to be
\emph{flat }relative to this pattern if every honeycomb $\mu\in\mathcal{M}$,
whose  rays intersect $\partial\Delta$ in a subset of $\{A_{j},B_{j},C_{j}\}$,
assigns zero multiplicity to all segments contained in the interior
of $\Omega$.
\begin{prop}
\label{prop:flat pieces}Let $\mu\in\mathcal{M}$ be a rigid tree
honeycomb. Then the light gray pieces determined in $\Delta_{*}$
by the support of $\mu^{*}$ are flat relative to the exit  pattern
of $\mu$.
\end{prop}

\begin{proof}
We denote by $\mathcal{N}\subset\mathcal{M}_{*}$ the collection of
those honeycombs that have branch points contained in $\Delta_{*}$
and exit points among $A_{0},\dots,A_{k_{1}},B_{0},\dots,B_{k_{2}},C_{0},\dots C_{k_{3}}$.
Suppose that $I\subset\Delta_{*}$ is a segment parallel to $w_{j}$
for some $j=1,2,3$. We say that $I$ is \emph{flat} if the following
condition is satisfied for every $\nu\in\mathcal{N}$:
\begin{enumerate}
\item [($\dagger$)]There is no edge $J$ of $\nu$ such that $J\cap I$
is a single point in the interior of $I$.
\end{enumerate}
If $AB$ and $AC$ are two flat edges, each parallel to $w_{1},w_{2},$
or $w_{3},$ and $\varangle BAC=60^{\circ}$, then the smallest convex
polygon with sides parallel to $w_{1},w_{2},$ or $w_{3}$ containing
$A,B,$ and $C$ is flat relative to the exit pattern of $\mu$. This
polygon is either an equilateral triangle or a trapezoid. It suffices
to show that all the edges of $\mu^{*}$ in $\Delta_{*}$ are flat.
These edges are of the form $e^{*}$, with $e$ an edge of $\mu$,
where $e^{*}$ is a segment of length $\mu(e)$ that is $60^{\circ}$
clockwise from \emph{$e$. }Denote by $S$ the set of all edges of
$\mu$ and by $S_{1}\subset S$ the collection of those edges $e$
with the property that $e^{*}$ is flat. It is clear that the rays
of $\mu$ belong to $S_{1}$. We show that the set $S_{1}$ has the
following properties (see Figure \ref{fig:e1e2e3}):
\begin{figure}
\begin{picture}(80,100)
\put(0,0){\includegraphics[scale=2]{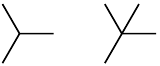}}
\put(-3,-5){$e_1$}
\put(34,17){$e_2$}
\put(-3,39){$e_3$}
\put(55,-5){$e_1$}
\put(92,17){$e_2$}
\put(55,39){$e_3$}
\put(80,37){$e_4$}
\end{picture}

\caption{\label{fig:e1e2e3}The edges $e_{j}$}

\end{figure}
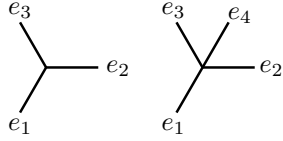
\begin{enumerate}
\item If $e_{1},e_{2},$ and $e_{3}$ are the only three edges meeting at
a branch point, and if $e_{1},e_{2}\in S_{1}$, then $e_{3}\in S_{1}$.
\item If $e_{1},e_{2},e_{3}$, and $e_{4}$ are four edges meeting at a
branch point of $\mu$ such that $\mu(e_{1})=\mu(e_{2})+\mu(e_{4})$,
and if $e_{1},e_{2}\in S_{1}$, then $e_{3},e_{4}\in S_{1}$.
\item If $e_{1},e_{2},e_{3}$, and $e_{4}$ are four edges meeting at a
branch point of $\mu$ such that $\mu(e_{1})=\mu(e_{2})+\mu(e_{4})$
and $e_{1}$ is collinear with $e_{4}$, and if $e_{1}\in S_{1}$,
then $e_{4}\in S_{1}$.
\end{enumerate}
Supposing for the moment that properties (1)--(3) have been proved,
we argue by contradiction that $S_{1}=S$. If $S_{1}\ne S$, choose
an edge $e\in S\backslash S_{1}$ with the property that the length
of a descendance path from some root edge to $e$ is the largest possible.
All the descendants of $e$ have longer descendance paths from the
same root edge and hence they belong to $S_{1}$. Since $e$ is not
a ray, it must have descendants. There are three possibilities:
\begin{enumerate}
\item [(a)]$e$ has two descendants $e_{2}$ and $e_{3}$ and there is
no other edge adjacent to their common endpoint. This possibility
is excluded by property (1) above because $e_{2},e_{3}\in S_{1}$.
\item [(b)]$e$ has two descendants $e_{1},e_{3}$, and there is a fourth
edge $e_{4}$ adjacent to their common endpoint. This possibility
is excluded by (2) for similar reasons.
\item [(c)]$e$ has only one descendant $e_{1}$, in which case there must
exist $e_{2}$ and $e_{3}$ adjacent to their common point such that
$\mu(e_{1})=\mu(e)+\mu(e_{2})$. This is excluded by (3).
\end{enumerate}
All three possibilities being excluded, we conclude that $S_{1}=S$.

Finally, we verify properties (1)--(3). In case (1), the edges $e_{1}^{*},e_{2}^{*},$
and $e_{3}^{*}$ form an equilateral triangle. Suppose that $\nu\in\mathcal{N}$
has an edge that intersects $e_{3}^{*}$ in a single interior point.
It is easy to see that in this case there must also be some edge of
$\nu$ that intersects either $e_{1}^{*}$ or $e_{2}^{*}$ in its
interior, thus contradicting the flatness of these segments.

In cases (2) and (3), the edges $e_{1}^{*}$ and $e_{4}^{*}$ are
the two parallel edges of an isosceles trapezoid and $e_{2}^{*},e_{3}^{*}$
are the other two edges. Suppose that $\nu\in\mathcal{N}$ has an
edge that intersects $e_{3}^{*}$ in a single interior point. Then
$\nu$ must also have an edge that intersects either $e_{1}^{*}$
or $e_{2}^{*}$ in a single interior point, so $e_{3}\in S_{1}$ if
$e_{1},e_{2}\in S_{1}$. Finally, suppose that $\nu\in\mathcal{N}$
has an edge that intersects $e_{4}^{*}$ in a single interior point.
Then $\nu$ must also have an edge that intersects $e_{1}^{*}$ in
a single interior point, so $e_{4}\in S_{1}$ if $e_{1}\in S_{1}$.
This concludes the proof.
\end{proof}
The last two results indicate an algorithm for checking whether a
given pattern arises from the exit multiplicities of a rigid tree
honeycomb. Thus, starting with a pattern, we construct flat segments
and flat regions in stages, as follows. At stage $0$, we have the
collection of flat segments determined by the pattern on the boundary
of the triangle. Suppose that the flat segments and flat regions of
stage $k$ have been constructed, and let $AB$ and $CD$ be two nonparallel
flat edges that appear in this configuration. Suppose that $AB$ and
$CD$ lie on different sides of a $60^{\circ}$ angle, and that the
interiors of two equilateral triangles $\Delta_{1}$ and $\Delta_{2}$,
based on $AB$ and $CD$ respectively, intersect. Then we construct
segments $A'B'\subset\Delta_{1}$ parallel to $AB$ with $A',B'\in\partial\Delta_{1}$
and $C'D'\subset\Delta_{2}$ parallel to $CD$ with $C',D'\subset\partial\Delta_{2}$
(see Figure \ref{fig:Flat-regions}) 
\begin{figure}
\begin{picture}(160,100)
\put(0,0){\includegraphics[scale=1.5]{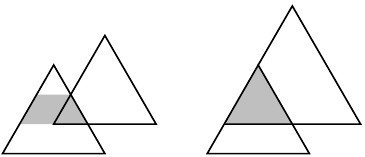}}
\put(-8,0){$A$}
\put(47,0){$B$}
\put(81,0){$A$}
\put(135,0){$B$}
\put(0,15){$A'$}
\put(38,16){$B'$}
\put(89,15){$A'$}
\put(127,16){$B'$}
\put(68,15){$D$}
\put(157,15){$D$}
\put(43,55){$C$}
\put(123,68){$C$}
\put(27,34){$C'$}
\put(104,41){$C'$}
\end{picture}\caption{Flat regions\label{fig:Flat-regions}, $D'=B'$}

\end{figure}
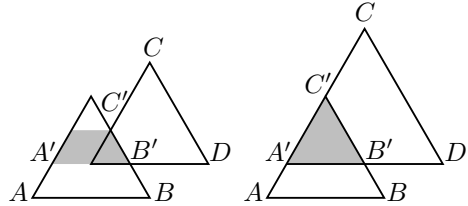
that have one endpoint in common. Denote by $\Omega$ the smallest
trapezoid (or equilateral triangle) with edges parallel to $w_{1},w_{2,}$
or $w_{3}$ that contains $A'B'$ and $C'D'$. Then, in stage $k+1$
we declare that $\Omega$ and all of its sides are flat. If the pattern
is in fact the exit  pattern of a rigid tree honeycomb, then the preceding
result shows that this process ends in a finite number of stages and
that the flat regions it creates (which are only equilateral triangles
and trapezoids) tile the original triangle. Moreover, the shortest
flat segments, corresponding to root edges of the rigid tree honeycomb,
must have unit length. 

In order to verify that a given rigid pattern, for which the preceding
process yields a collection of flat triangles and trapezoids that
tile the entire triangle $\Delta$, really arises as the exit  pattern
of some rigid tree honeycomb $\mu$, we must construct $\mu$ or some
honeycomb homologous to it. One way to do this is to find a honeycomb
$\nu$ in $\mathcal{M}_{*}$ whose support consists of the flat segments
created in the process, construct $\mu$ such that $\mu^{*}=\nu$,
and verify that it is in fact a rigid tree honeycomb. This may be
a laborious process. Fortunately, there exists a simpler process that
yields such as honeycomb $\mu$: each flat region being a triangle
or a trapezoid, it has a circumcenter. The segment joining the circumcenters
of two adjacent flat regions is part of the perpendicular bisector
of the common edge. Add to these segments rays contained in the perpendicular
bisector of each flat segment in $\partial\Delta$ terminating at
the circumcenter of the adjacent flat region. Assign to each bisecting
segment a multiplicity equal to the length of the original flat segment.
This way, we obtain a honeycomb, rotated counterclockwise by $30^{\circ}$.
It is fairly easy to check whether this is a rigid tree honeycomb
since we know what its root edges are.
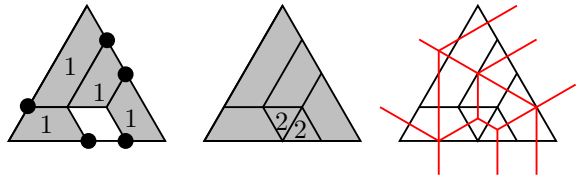
\begin{figure}
\begin{picture}(200,100)
\put(0,0){\includegraphics[scale=1.5]{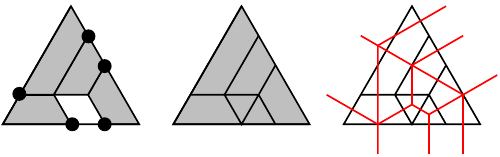}}
\put(102,18){$2$}
\put(109,15){$2$}

\put(22,39){$1$}
\put(13,17){$1$}
\put(45,19){$1$}
\put(33,29){$1$}
\end{picture}

\caption{Determining flat regions\label{fig:Determining-flat-regions}}

\end{figure}

Figure \ref{fig:Determining-flat-regions} illustrates the two above
processes for the pattern $3,1|2,1,1|2,1,1$. The flat regions are
determined in two stages, they tile $\Delta$, and the shortest flat
edges have unit length. However, the (red) honeycomb constructed using
circumcenters is rigid but not extreme because not all the root edges
are equivalent. For the pattern $1,2,1|1,2,1|1,2,1$ shown in Figure
\ref{fig:Flat-regions-do-not-cover}, all the flat regions are determined
in stage $1$ and they do not cover $\Delta$, so this is not the
exit  pattern of a rigid tree honeycomb.
\begin{figure}
\includegraphics[scale=1.5]{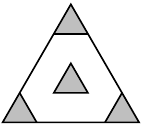}

\caption{\label{fig:Flat-regions-do-not-cover}Flat regions do not cover $\Delta$}

\end{figure}
 On the other hand, for the pattern $4,2|2,2,2|2,2,2$ the flat regions
do cover $\Delta$ but the shortest flat edge has length $2$. In
this case, the circumcenter construction yields twice a rigid tree
honeycomb. There are locking patterns for which the flat regions created
by the method described above overlap and thereby create greater flat
regions that may not by triangles or trapezoids. Such an example is
$3,2|1,3,2|2,2,1$, illustrated in Figure \ref{fig:A-rigid-pattern}.
\begin{figure}
\includegraphics{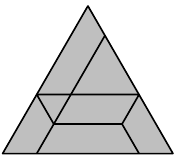}

\caption{\label{fig:A-rigid-pattern}A rigid pattern}

\end{figure}

Suppose that $\mu\in\mathcal{M}$ is a rigid tree honeycomb and its
dual is written as in the above proof $\mu^{*}=\sum_{j=1}^{n}c_{j}\rho_{j}$.
If $t_{1},\dots,t_{n}>0$ and $\nu\in\mathcal{M}$ is such that $\nu^{*}=\sum_{j=1}^{n}t_{j}\rho_{j}$
then $\mu$ and $\nu$ are homologous in the sense defined in \cite{blt}.
In other words, up to a translation, $\nu$ can be obtained by stretching
or shrinking the edges of $\mu$ while preserving their multiplicities
(see Figure \ref{fig:Two-homologous-honeycombs}).
\begin{figure}
\includegraphics[scale=0.7]{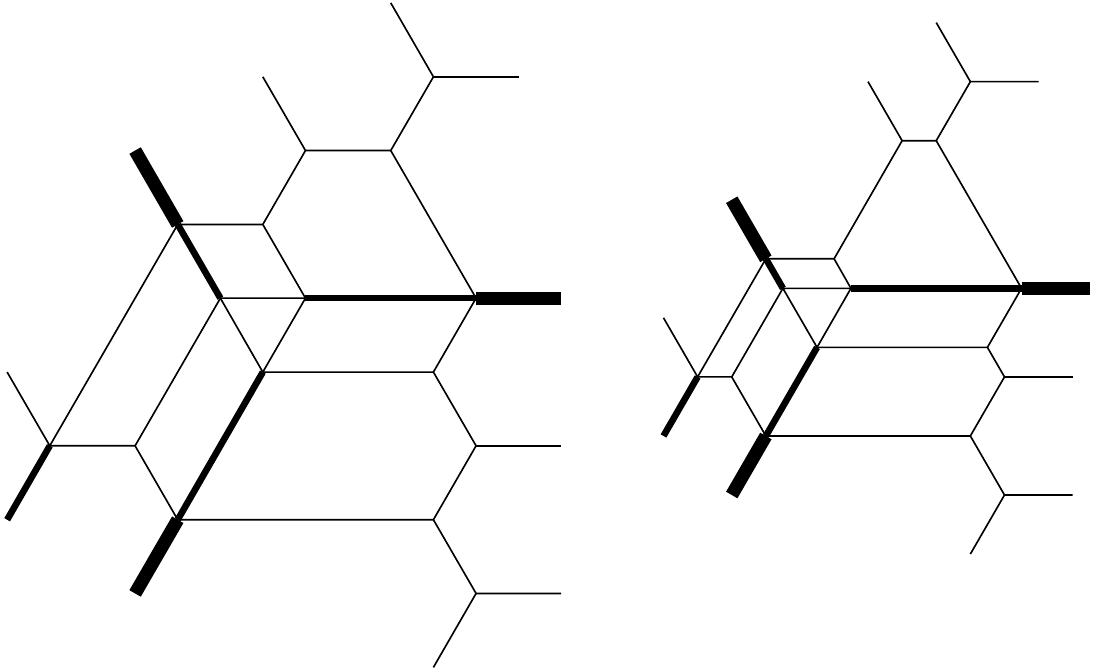}

\caption{\label{fig:Two-homologous-honeycombs}Two homologous honeycombs with
$\omega=6$}

\end{figure}
 Conversely, if $\nu$ is an arbitrary rigid tree honeycomb of weight
$\omega(\mu)$ that determines the same pattern on $\Delta_{*}$,
then $\nu$ must satisfy $\nu^{*}=\sum_{j=1}^{n}t_{j}\rho_{j}$ for
some $t_{1},\dots,t_{n}>0$. Thus, in order to classify the rigid
tree honeycombs of a given weight $\omega$, it suffices to find the
locking patterns on a triangle $\Delta$ with sides equal to $\omega$
that also correspond to rigid tree honeycombs. There is only a finite
number of such patterns because the lengths of the segments in the
boundary of $\Delta$ determined by the relevant patterns are exit
multiplicities of a tree honeycomb and are therefore integers.

Let $\mu\in\mathcal{M}$ be a rigid tree honeycomb of weight $\omega$.
According to (\ref{eq:omega square +2}), the integers $\mu^{(j)}(c)$,
$j=1,2,3,c\in\mathbb{R}$, are subject to the following requirements:

\begin{align*}
\sum_{c\in\mathbb{R}}\mu^{(j)}(c) & =\omega,\quad j=1,2,3,\\
\sum_{c\in\mathbb{R}}[\mu^{(1)}(c)^{2}+\mu^{(2)}(c)^{2}+\mu^{(3)}(c)^{2}] & =\omega^{2}+2.
\end{align*}
If we arrange the nonzero exit multiplicities of $\mu$ in a nonincreasing
list $\alpha_{1},\dots,\alpha_{N}$, we have
\begin{equation}
\sum_{j=1}^{N}\alpha_{j}=3\omega,\quad\sum_{j=1}^{N}\alpha_{j}^{2}=\omega^{2}+2.\label{eq:diophantus}
\end{equation}
This system of Diophantine equations has finitely many solutions that
are fairly easy to list. We observe that subtracting the first equation
above from the second we obtain
\begin{equation}
\sum_{j=1}^{N}\frac{\alpha_{j}(\alpha_{j}-1)}{2}=\frac{(\omega-1)(\omega-2)}{2}.\label{eq:triangular sum}
\end{equation}
In other words, we need to write the triangular number $(\omega-1)(\omega-2)/2$
as a sum of triangular numbers, and this sum will determine precisely
the values of those $\alpha_{j}$ that are greater than $1$. Once
numbers $\alpha_{j}$ have been determined, there may be several ways
that they can be arranged to make a pattern on a triangle with sides
$\omega$. For each of these arrangements, we can use the algorithm
outlined above to determine whether it is the exit  pattern of a rigid
tree honeycomb. We illustrate this for small values of $\omega$.
For $\omega\le2$ we have $(\omega-1)(\omega-2)/2=0$ and thus $\alpha_{j}=1$
for all $j$. The corresponding types of rigid tree honeycombs are
pictured in Figure \ref{fig:Rigid-tree-homeycombs-small-omega}(A)
and (B).
\begin{figure}
\subfloat[$\omega=1$]{\includegraphics{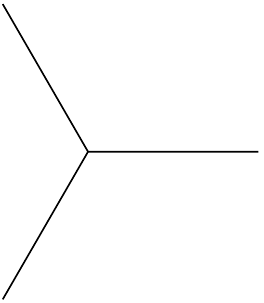}}\subfloat[$\omega=2$]{

\includegraphics[scale=0.5]{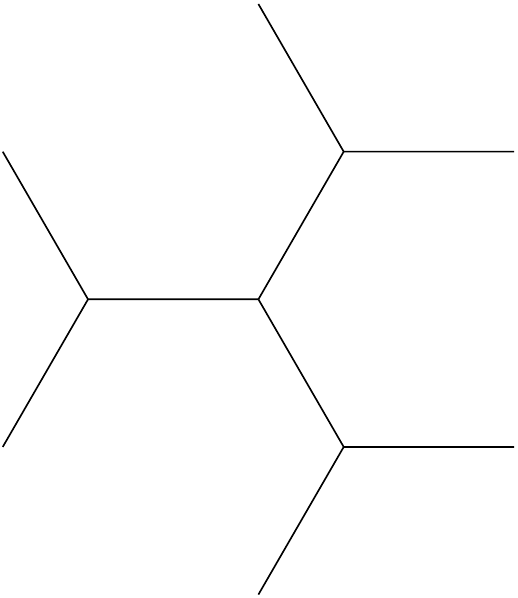}}\subfloat[$\omega=3$]{

\includegraphics[scale=0.4]{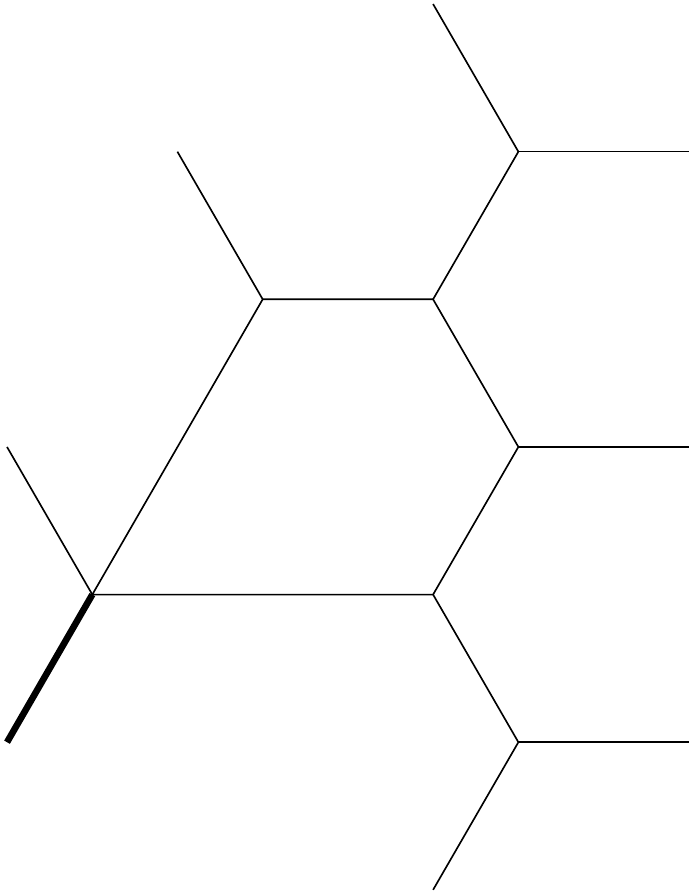}

}

\caption{\label{fig:Rigid-tree-homeycombs-small-omega}Rigid tree honeycombs
with small weight}

\end{figure}
 For $\omega=3$ we have $(\omega-1)(\omega-2)/2=1$ so the only numbers
$\alpha_{j}$ satisfying (\ref{eq:diophantus}) are $2,1,1,1,1,1,1,1$.
These lengths can be arranged around a triangle $\Delta$ in six different
ways but all these arrangements can be obtained from one of them by
rotations or reflections. It is seen immediately that these arrangements
are in fact the exit  patterns of rigid tree honeycombs. Thus, there
are exactly six types of rigid tree honeycombs of weight 3, all of
which can be obtained by rotations and reflections from the one pictured
in Figure \ref{fig:Rigid-tree-homeycombs-small-omega}(C) (the thicker
edge has multiplicity 2).

For $\omega=4$, we have $(\omega-1)(\omega-2)/2=3$, and this can
be written as $3=3$ and $3=1+1+1$. The corresponding solutions $\alpha_{j}$
are 
\begin{align*}
3,1,1,1,1,1,1,1,1,1, & \text{ and}\\
2,2,2,1,1,1,1,1,1.
\end{align*}
We find that there are $44$ distinct types of rigid tree honeycombs,
all of which can be obtained from one of the eight depicted in
\begin{figure}
\includegraphics[scale=0.2]{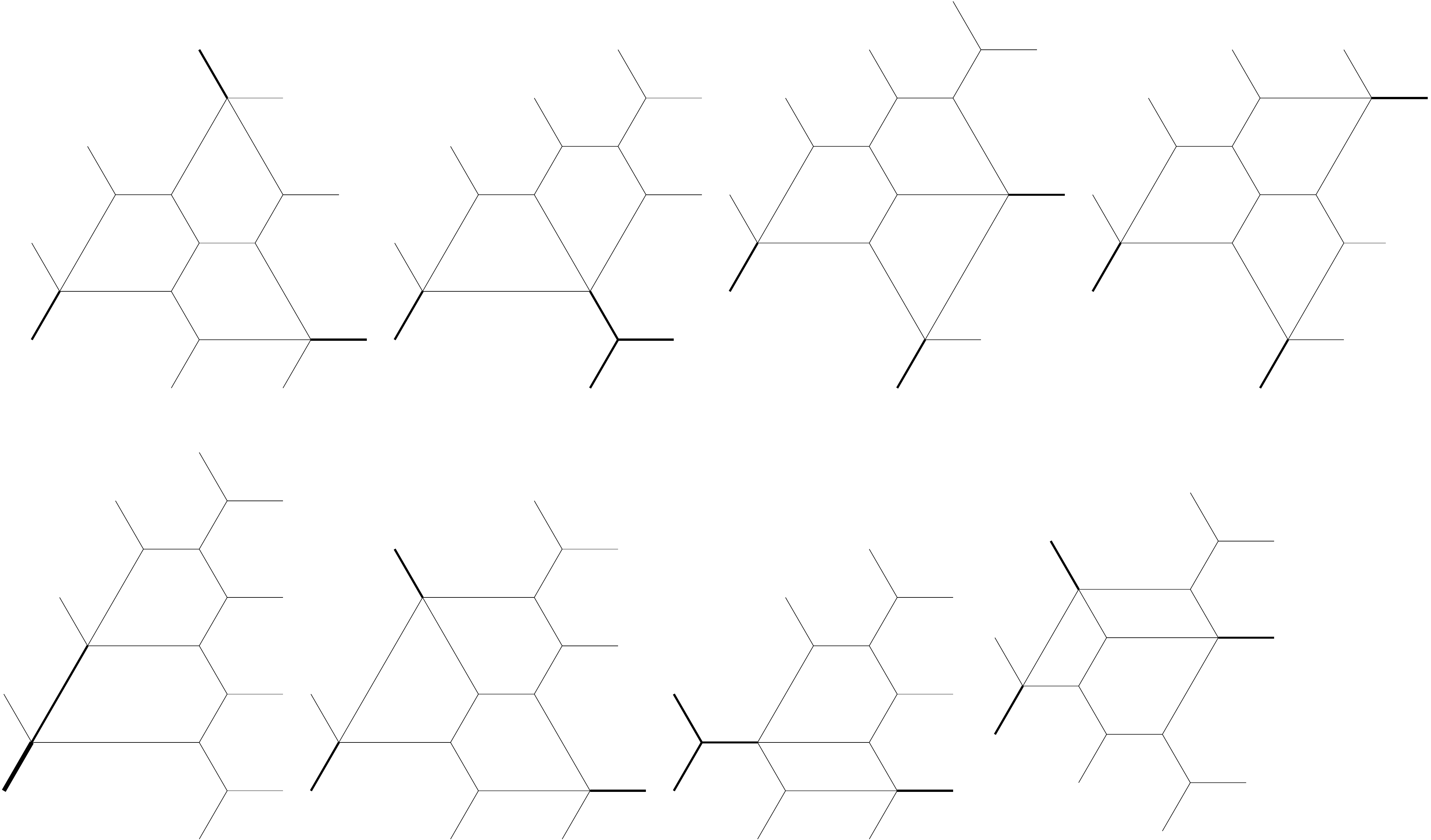}

\caption{\label{fig:honey with omega=00003D4}Rigid tree honeycombs with $\omega=4$}

\end{figure}
 Figure \ref{fig:honey with omega=00003D4} by rotations and reflections.
Each of the types pictured, except for the first one, yields six different
types via rotations and reflections. The first one is invariant under
rotations by $120^{\circ}$. For larger values of $\omega$, the number
of types of rigid tree honeycombs grows fairly rapidly. For instance,
we get 272 such types for $\omega=5$ and more than 1000 for $\omega=6$.
Rigid tree honeycombs with a rotational symmetry exist for every weight
$\omega$ that is not a multiple of $3$. Some examples of such honeycombs
are provided in \cite{blt} but for large $\omega$ there are other
examples. To find them we denote by $\beta_{1}\ge\dots\ge\beta_{k}$
the possible exit multiplicities of such a honeycomb in the direction
of $w_{1}$, and note that these numbers satisfy
\[
\sum_{j=1}^{k}\beta_{j}=\omega,\quad\sum_{j=1}^{k}\beta_{j}^{2}=\frac{\omega^{2}+2}{3}.
\]
 These equations combine to yield
\[
\sum_{j=1}^{k}\frac{\beta_{j}(\beta_{j}-1)}{2}=\frac{(\omega-1)(\omega-2)}{6}.
\]
For instance, if $\omega=8$, we must write $(\omega-1)(\omega-2)/6=7$
as a sum of triangular numbers. The acceptable possibilities are
\[
6+1,3+3+1,
\]
and they correspond to the following sequences of $\beta_{j}$:
\begin{align*}
4,2,1,1\\
3,3,2.
\end{align*}
(The representation $7=1+1+1+1+1+1+1$ requires that $\beta_{1}=\cdots=\beta_{7}=2$
but $7\times2>8$. Similarly, $7=3+1+1+1+1$ requires $\beta_{1}=3$
and $\beta_{2}=\cdots=\beta_{5}=2$.) We obtain, up to symmetry, the
4 distinct rigid tree honeycomb types depicted in Figure \ref{fig:symmetric-8}.
\begin{figure}
\includegraphics[scale=0.7]{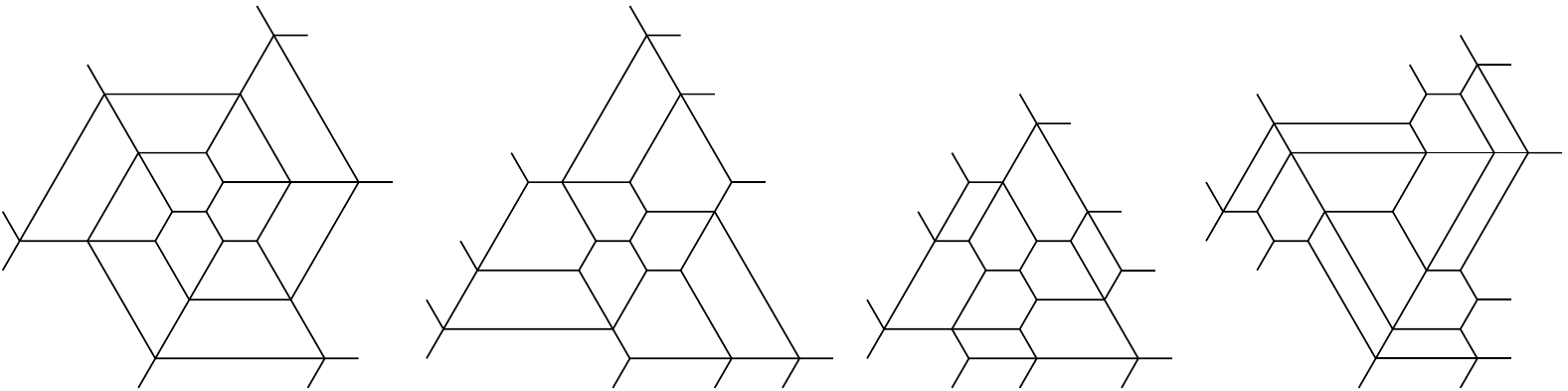}

\caption{\label{fig:symmetric-8}Four honeycombs with $\omega=8$}

\end{figure}

\section{Honeycombs compatible with a puzzle\label{sec:Honeycombs-compatible-with-a}}

Suppose that we have the puzzle of a honeycomb $\nu\in\mathcal{M}$
and, in addition, a honeycomb $\mu\in\mathcal{M}$ superposed on it.
We say that $\mu$ is \emph{compatible} with the puzzle of $\nu$
if the dark gray parallelogram contain no branch points of $\mu$
in their interior and the edges of $\mu$ only cross such a parallelogram
along segments parallel to an edge of the parallelogram. This concept
appeared in \cite{KTW} where it was used to explain saturated Horn
inequalities. It was also employed in \cite[Section 7]{blt} to give
a different argument for some factorizations of Littlewood-Richardson
coefficients first found in \cite{king-tollu,derksen}. Here we are
mostly interested in the existence and construction of such compatible
honeycombs $\mu$, particularly rigid tree honeycombs. It is useful
to observe that a tree honeycomb that is compatible with the puzzle
of $\nu$ crosses the gray parallelograms along segments that are
either all parallel to the white edges or all parallel to the light
gray edges (see 
\begin{figure}
\includegraphics[scale=0.7]{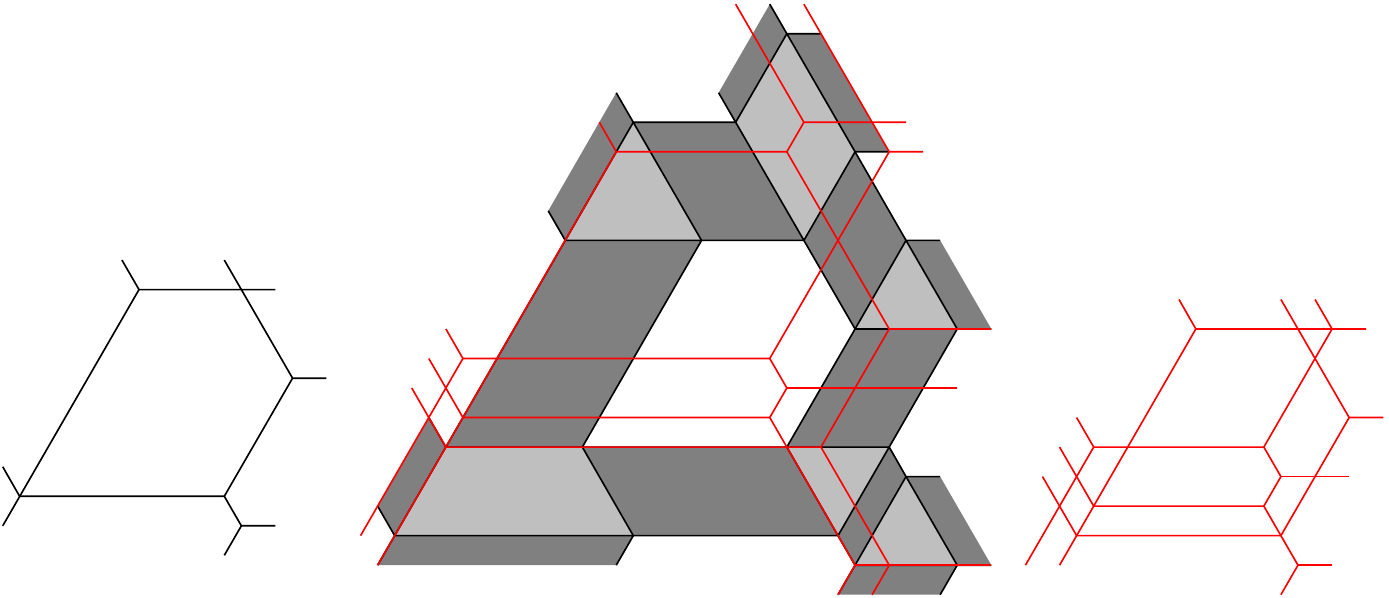}

\caption{\label{fig:compatible honeycomb, deflation}$\nu$, its puzzle, a
compatible $\mu$, and $\mu_{\nu}$}

\end{figure}
Figure \ref{fig:compatible honeycomb, deflation}).  Deflating the
puzzle of $\nu$ yields, in addition to the original honeycomb $\nu$,
a honeycomb $\mu_{\nu}$ constructed as follows: 
\begin{enumerate}
\item the parts of the support of $\mu$ inside the white puzzle pieces
are translated along with those white pieces,
\item the parts of the support of $\mu$ that cross dark gray parallelograms
between white puzzle pieces are discarded, and
\item if $AB$ is an edge of $\nu$ and the support of $\mu$ crosses the
inflation of $AB$ on segments $I$ parallel to $AB$, then $AB$
is also contained in the support of $\mu_{\nu}$ and its multiplicity
is the sum of the multiplicities of all such segments $I$.
\end{enumerate}
A similar construction, using the $*$-deflation of the puzzle of
$\nu$, yields a honeycomb $\mu_{\nu^{*}}$ with $\omega(\mu_{\nu^{*}})=\omega(\mu_{\nu})=\omega(\mu)$.
Moreover, it was shown in \cite{blt} that
\begin{equation}
\mu^{*}=(\mu_{\nu})^{*}+(\mu_{\nu^{*}})^{*}.\label{eq:mu* when mu is compatible with the puzzle of nu}
\end{equation}

Our inductive procedure for constructing honeycombs from rigid overlays
depends on the construction of honeycombs that are compatible with
the puzzle of a \emph{rigid} honeycomb. We begin with compatible honeycombs
whose support does not intersect the interior of any puzzle piece.
Such a support must be contained in the edges of the dark gray parallelograms.
Thus, suppose that $\nu\in\mathcal{M}$ is a rigid honeycomb and regard
the union of the edges of the dark gray parallelograms in its puzzle
as a prospective support for a honeycomb. For simplicity, we simply
call puzzle edges the edges of the dark gray parallelogram. These
edges are either white or light gray, according to the color of the
other neighboring puzzle piece.
\begin{lem}
\label{lem:evil loops in puzzles}Let $\nu$ be a honeycomb in $\mathcal{M}$.
Every evil loop contained in the puzzle edges of $\nu$ is a gentle
loop, that is, consecutive edges in the loop form $120^{\circ}$ angles.
\end{lem}

\begin{proof}
Since all the branch points in this support are `rakes' (as in the
second part of Figure \ref{fig:e1e2e3}), it follows that every evil
turn of an evil loop is either a gentle turn (of $120^{\circ}$) or
a $180^{\circ}$ turn along the `handle' of a rake. Suppose that an
evil loop $e_{1},\dots,e_{n}=e_{1}$ contains at least one $180^{\circ}$
turn. We may assume that $e_{1}=AB$ and $e_{2}=BA$, in which case
it follows that the edges $e_{2},\dots,e_{k}$ point to the acute
angles of the adjacent parallelograms until we meet again a $180^{\circ}$
turn $e_{k}e_{k+1}$ (see Figure \ref{fig:Evil-turns-in-a-puzzle}).
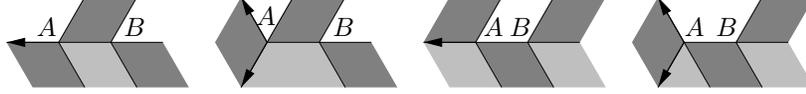
\begin{figure}
\begin{picture}(320,100)
\put(0,0){\includegraphics[scale=0.8]{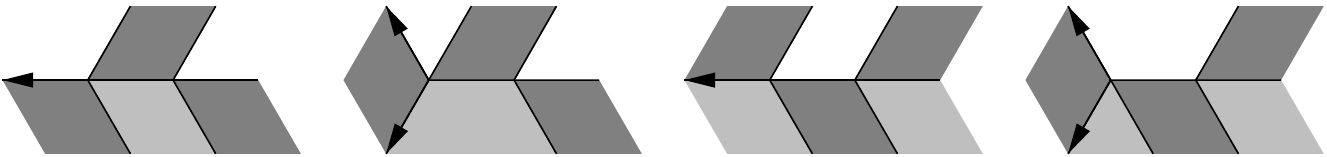}}
\put(12,20){$A$}
\put(95,24){$A$}
\put(181,20){$A$}
\put(257,20){$A$}
\put(45,20){$B$}
\put(124,20){$B$}
\put(191,20){$B$}
\put(269,20){$B$}
\end{picture}\caption{\label{fig:Evil-turns-in-a-puzzle}Evil turns in a rigid puzzle}

\end{figure}
 Continuing this way around the loop, we conclude that $e_{1}$ also
points to the acute angle of the adjacent parallelogram, an obvious
contradiction.
\end{proof}
According to \cite[Theorems 4 and 7]{KTW}, the puzzle of a rigid
honeycomb $\nu$ contains no gentle (and, by Lemma \ref{lem:evil loops in puzzles},
no evil) loops. It follows that any honeycomb supported by the puzzle
edges of  $\nu$ is necessarily rigid. Suppose for a moment that there
exists a honeycomb $\mu$ whose support is the union of all the edges
of the puzzle of $\nu$. The root edges of such a honeycomb are easily
identified. For each ray $e$ in the support of $\nu$ there is one
equivalence class of root edges containing the ray in the inflation
\begin{figure}
\includegraphics[scale=0.5]{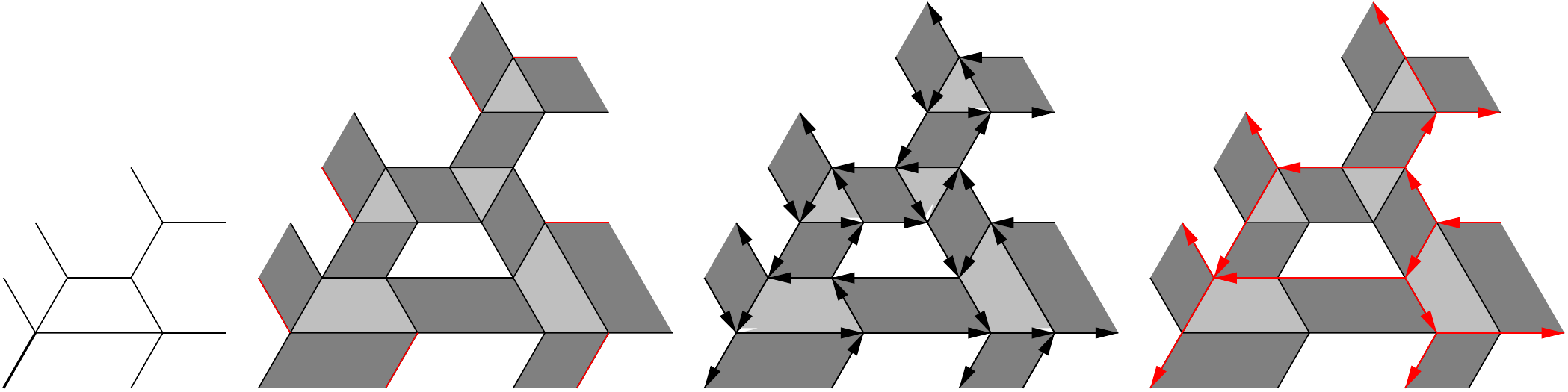}

\caption{\label{fig:puzzle-and-inflation}}

\end{figure}
 of $e$ incident to the acute angle of that inflation(see Figure
\ref{fig:puzzle-and-inflation} for an example with these root edges
colored red). We now orient all edges towards the acute angles of
the parallelogram to which they belong (or away from the obtuse angles
in the case of some of the rays). (Note that the opposite orientation
was used in \cite[Section 4]{KTW} for different purposes.) Thus,
if $f$ is a ray in the support of $\nu$, the boundary of its inflation
has two edges that are rays: an \emph{incoming ray }pointing to the
acute angle of the inflation, and an \emph{outgoing ray. }This orientation
is useful because, except for the root edges, it is precisely the
orientation of the descendant relation, that is, each edge (with the
exception of the root edges) is oriented from an ancestor to a descendant.
The resulting oriented graph contains no loops and, in addition, for
each edge $e$ there is at least one path from one of the root edges
to $e$. It is easy now to see that there actually exists a rigid
tree honeycomb rooted at each of the root edges of the hypothetical
honeycomb $\mu$, and that the union of the supports of these honeycombs
contains all the edges of the puzzle. We summarize these observations
in Theorem \ref{thm:honey in the boundary of puzzles}. The last assertion
in the statement follows by observing that none of the events described
in Remark \ref{rem:calculation of Sigma} is possible for the honeycombs
$\mu_{1}$ and $\mu_{2}$. Figure \ref{fig:puzzle-and-inflation}
shows (in red) the support of one of the honeycombs constructed in
the manner just described. 
\begin{thm}
\label{thm:honey in the boundary of puzzles} Let $\nu\in\mathcal{M}$
be a rigid honeycomb. Then there are exactly ${\rm exit}(\nu)$ rigid
tree honeycombs whose support is contained in the union of the puzzles
edges of $\nu$. The root of the honeycomb corresponding to a ray
$e$ of $\nu$ is the incoming ray in the inflation of $e$. Every
edge of the puzzle of $\nu$ is contained in the support of one of
these tree honeycombs. If $\mu_{1}$ and $\mu_{2}$ are two of these
tree honeycombs, then $\Sigma_{\mu_{1}}(\mu_{2})=0$.
\end{thm}

The following observation about arbitrary compatible honeycombs is
useful in Section \ref{sec:de-degeneration}.
\begin{lem}
\label{lem:compatible puzzle does not leave the edges of the puzzle once there}Let
$\nu\in\mathcal{M}$ be a rigid honeycomb and let $\mu\in\mathcal{M}$
be a honeycomb that is compatible with the puzzle of $\nu$. Suppose
that the support of $\mu$ contains an edge $e$ of the puzzle. Then
the support of $\mu$ contains every descendant of $e$ relative to
the orientation of the puzzle edges of $\nu$.
\end{lem}

\begin{proof}
It suffices to show that, given two edges $e_{1}$ and $e_{2}$ in
the puzzle of $\nu$ such that $e_{1}e_{2}$ is a gentle path and
$e_{1}$ is contained in the support of $\mu$, it follows that $e_{2}$
is also contained in the support of $\mu$. The three possible configurations,
up to rotations and reflections, are shown in Figure
\begin{figure}
\begin{picture}(220,100)
\put(0,0){\includegraphics[scale=2.5]{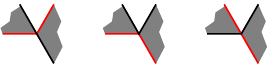}}
\put(-9,20){$e_1$}
\put(65,20){$e_1$}
\put(150,45){$e_1$}

\put(40,45){$e_2$}
\put(100,0){$e_2$}
\put(175,0){$e_2$}

\end{picture}\caption{\label{fig:contained edge}}
Edges of a puzzle in the support of a compatible honeycomb

\end{figure}
 \ref{fig:contained edge}. We observe first that at least part of
$e_{2}$ must be contained in the support of $\mu$ because otherwise
the support of $\mu$ would contain one edge crossing a dark gray
parallelogram in a direction that is not allowed by compatibility.
Similarly, if $I\subset e_{2}$ is the largest segment in the support
of $\mu$ that contains $e_{1}\cap e_{2}$, then $I$ must equal $e_{2}$
because otherwise the endpoint of $I$ (other than $e_{1}\cap e_{2})$
would need to be a branch point for $\mu$ with one edge of $\mu$
crossing the neighboring dark gray parallelogram in the direction
forbidden by compatibility.
\end{proof}
It is of interest to calculate the exit pattern of the tree honeycombs
described in Theorem \ref{thm:honey in the boundary of puzzles},
and we do eventually provide an answer amounting to an inductive procedure
depending on the number of extreme summands of the honeycomb $\nu$
(see Remark \ref{rem:general honeycomb arising from a ray of a rigid honey}).
We start with the special case of an extreme rigid honeycomb. The
direction of the  rays corresponding to the multiplicities $\alpha_{j}$
are not specified in the following statement, but the multiplicity
$\alpha_{1}\alpha_{j}$ corresponds to the outgoing ray in the inflation
of the ray with multiplicity $\alpha_{j}$, while the multiplicities
$\alpha_{1}^{2}-1$ and $1$ correspond to outgoing and incoming rays,
respectively, in the inflation of the ray with multiplicity $\alpha_{1}$.
\begin{prop}
\label{prop:multiple of a rigid tree, original version}Suppose that
$\nu$ is a rigid tree honeycomb and that the nonzero exit multiplicities
of $\nu$ are, in counterclockwise order, $\alpha_{1},\dots,\alpha_{n}$.
Let $\mu$ be the rigid tree honeycomb, supported in the union of
the puzzle edges of $\nu$, and rooted in the incoming ray in the
inflation of the ray with multiplicity $\alpha_{1}$. Then the nonzero
exit multiplicities of $\mu$ are, in counterclockwise order, 
\[
\alpha_{1}^{2}-1,1,\alpha_{1}\alpha_{2},\dots,\alpha_{1}\alpha_{n},
\]
where the first term is omitted if $\alpha_{1}=1$. In particular,
$\omega(\mu)=\alpha_{1}\omega(\nu)$,
\begin{equation}
{\rm exit}(\mu)={\rm exit}(\nu)+1\label{eq:exit(comp)=00003Dexit(nu)+1}
\end{equation}
if $\alpha_{1}>1$, and $\mu$ is homologous to $\nu$ if $\alpha_{1}=1$.
\end{prop}

\begin{proof}
Suppose that the exit multiplicities of $\mu$ are $\beta_{1},1,\beta_{2},\dots,\beta_{n}$,
where $\beta_{j}$ is the multiplicity of the outgoing ray corresponding
to $\alpha_{j}$. The presence of $1$ is explained by the choice
of root edge which happens to be an incoming ray. Deflate now the
puzzle of $\nu$ and obtain thereby a honeycomb $\mu_{\nu}$ with
exit multiplicities $\beta_{1}+1,\beta_{2},\dots,\beta_{n}.$ The
honeycomb $\mu_{\nu}$ has support contained in that of $\nu$ and
is therefore of the form $k\nu$ for some positive integer $k$. We
deduce that the exit multiplicities of $\mu$ are $k\alpha_{1}-1,1,k\alpha_{2},\dots,k\alpha_{n}$
and $\omega(\mu)=k\omega(\nu).$ Thus,
\begin{align*}
\sum_{j=1}^{n}\alpha_{j}^{2} & =\omega(\nu)^{2}+2,\\
(k\alpha_{1}-1)^{2}+1+\sum_{j=2}^{n}(k\alpha_{j})^{2} & =(k\omega(\nu))^{2}+2,
\end{align*}
and subtracting $k^{2}$ times the first equation from the second
we obtain
\[
-2k\alpha_{1}+2=-2k^{2}+2.
\]
The only positive solution of this equation is $k=\alpha_{1}$.
\end{proof}
\begin{cor}
\label{cor:multiple of a rigid tree, mirror image}Suppose that $\nu$
is a rigid tree honeycomb and that the nonzero exit multiplicities
of $\nu$ are, in counterclockwise order, $\alpha_{1},\dots,\alpha_{n}$.
There exists a rigid tree honeycomb $\mu$ whose nonzero exit multiplicities
are, in counterclockwise order, 
\[
1,\alpha_{1}^{2}-1,\alpha_{1}\alpha_{2},\dots,\alpha_{1}\alpha_{n},
\]
where the second term is omitted if $\alpha_{1}=1$. In particular,
$\omega(\mu)=\alpha_{1}\omega(\nu)$.
\end{cor}

\begin{proof}
Let $\nu'$ be a honeycomb obtained by reflecting $\nu$ in a line
parallel to $w_{1}$. Thus the nonzero exit multiplicities of $\nu'$
can be listed, in counterclockwise order, as $\alpha_{1},\alpha_{n},\alpha_{n-1},\dots,\alpha_{2}$.
Proposition \ref{prop:multiple of a rigid tree, original version},
applied to $\nu'$, shows that there exists a rigid tree honeycomb
$\mu'$ whose nonzero exit multiplicities are, in counterclockwise
order,
\[
\alpha_{1}^{2}-1,1,\alpha_{1}\alpha_{n},\alpha_{1}\alpha_{n-1},\dots,\alpha_{1}\alpha_{2}.
\]
The honeycomb $\mu$ is now obtained by reflecting $\mu'$ in a line
parallel to $w_{1}$.
\end{proof}
\begin{rem}
An alternate proof of the preceding corollary uses the $*$-inflation
of $\mu$.
\end{rem}

We next construct compatible honeycombs starting with a rigid overlay
of two rigid tree honeycombs. The following result was proved in \cite[Theorem 5.2]{blt}.
\begin{thm}
\label{thm:stretched honeycomb}Suppose that $\nu_{1}$ and $\nu_{2}$
are rigid tree honeycombs and $\Sigma_{\nu_{2}}(\nu_{1})=0$. Then
there exists a honeycomb $\widetilde{\nu}_{2}$, homologous to $\nu_{2}$
and compatible with the puzzle of $\nu_{1}$, such that the parts
of the support of $\widetilde{\nu}_{2}$ contained in the white puzzle
pieces are simply translates of the corresponding parts of the support
of $\nu_{2}$.
\end{thm}

We recall briefly the construction of $\widetilde{\nu}_{2}$. Choose
a root edge $e$ of $\nu_{2}$ that is not contained in the support
of $\nu_{1}$ and orient all the edges of $\nu_{2}$ away from $e$
(that is, in the direction of the descendance paths from $e$. If
a portion $f$ of an edge of $\nu_{2}$ is contained in an edge of
$\nu_{1}$, attach $f$ to the white puzzle piece \emph{on its right}
relative to this orientation. This way we obtain the part of the support
of $\widetilde{\nu}_{2}$ that is contained in the boundary of a white
puzzle piece. The part of the support of $\widetilde{\nu}_{2}$ in
the interior of the white pieces is obtained simply by translation,
as in the statement above. Finally, one reconnects the edges of $\widetilde{\nu}_{2}$
that are transversal to the support of $\nu_{1}$ and this is done
by adding a number of segments that cross dark gray parallelograms
and are parallel to their light gray sides. The requirement that $\Sigma_{\nu_{2}}(\nu_{1})=0$
makes this construction possible. This process is illustrated in Figure
\ref{fig:honeycomb and stretch} where the supports of $\nu_{2}$
and $\widetilde{\nu}_{2}$ are drawn with dashed lines. (An interesting
feature of the overlay in Figure \ref{fig:honeycomb and stretch}
is that $\nu_{1}$ and $\nu_{2}$ assign nonzero multiplicity to precisely
the same rays.) 
\begin{figure}
\includegraphics{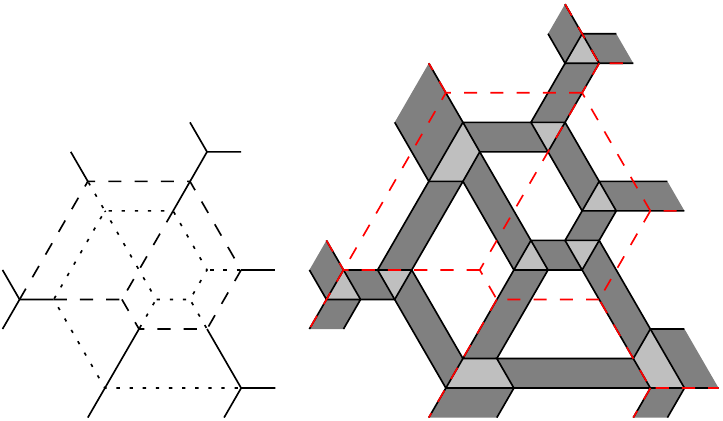}

\caption{\label{fig:honeycomb and stretch}The honeycombs $\nu_{1}$, $\nu_{2}$,
and $\widetilde{\nu}_{2}$}

\end{figure}

\begin{prop}
\label{prop:sigma of stretch vs inflated honey} Suppose that $\nu_{1}$
and $\nu_{2}$ are rigid tree honeycombs such that $\Sigma_{\nu_{2}}(\nu_{1})=0$,
and let $f$ be a ray of $\nu_{1}$ to which $\nu_{1}$ assigns exit
multiplicity $\alpha_{1}$. Let $\alpha_{2}\ge0$ be the exit multiplicity
that $\nu_{2}$ assigns to the same ray. Denote by $\widetilde{\nu}_{1}$
the rigid tree honeycomb whose support is contained in the union of
the puzzle edges of $\nu_{1}$ and with one root edge equal to the
incoming ray of the inflation of $f$. Finally, let $\widetilde{\nu}_{2}$
be the rigid tree honeycomb, homologous to $\nu_{2}$, constructed
in Theorem\emph{ \ref{thm:stretched honeycomb}}. Then $\Sigma_{\widetilde{\nu}_{2}}(\widetilde{\nu}_{1})=0$
and $\Sigma_{\widetilde{\nu}_{1}}(\widetilde{\nu}_{2})=\alpha_{2}+\alpha_{1}\Sigma_{\nu_{1}}(\nu_{2})$.
\end{prop}

\begin{proof}
Suppose that $\nu_{2}^{(j)}(c)\nu_{1}^{(j)}(d)\ne0$ for some $j=1,2,3$
and some $c<d$. Unless the ray $f$ is on the line $x_{j}=d$, there
exist numbers $c'<d'$ such that $\widetilde{\nu}_{2}^{(j)}(c')=\nu_{2}^{(j)}(c)$
and $\widetilde{\nu}_{1}^{(j)}(d')=\alpha_{1}\nu_{1}^{(j)}(d)$. If
the ray $f$ is on the line $x_{j}=d$, we find $c'<d'<d''$ such
that $\widetilde{\nu}_{2}^{(j)}(c')=\nu_{2}^{(j)}(c)$, $\widetilde{\nu}_{1}^{(j)}(d')=\alpha_{1}^{2}-1$,
and $\widetilde{\nu}_{1}^{(j)}(d'')=1$. This describes all the pairs
$c'<d'$ (or $c'<d''$) such that $\widetilde{\nu}_{2}^{(j)}(c')\widetilde{\nu}_{1}^{(j)}(d')\ne0$
with one exception arising from the ray $f$ that corresponds to $\nu_{1}^{(j)}(c)=\alpha_{1}$
and $\nu_{2}^{(j)}(c)=\alpha_{2}$. In this case, we obtain $d'<d''$
(corresponding to the outgoing and incoming rays in the inflation
of $f$) such that $\widetilde{\nu}_{2}^{(j)}(d')=\alpha_{2}$ and
$\widetilde{\nu}_{1}^{(j)}(d'')=1$. We conclude that
\begin{align*}
\sum_{c'<d'}\widetilde{\nu}_{2}^{(j)}(c')\widetilde{\nu}_{1}^{(j)}(d') & =\alpha_{2}+\alpha_{1}\sum_{c<d}\nu_{2}^{(j)}(c)\nu_{1}^{(j)}(d)\\
 & =\alpha_{2}+\alpha_{1}\Sigma_{\nu_{2}}(\nu_{1})+\alpha_{1}\omega(\nu_{1})\omega(\nu_{2}),
\end{align*}
and thus 
\begin{align*}
\Sigma_{\widetilde{\nu}_{1}}(\widetilde{\nu}_{2}) & =\alpha_{2}+\alpha_{1}\Sigma_{\nu_{1}}(\nu_{2})+\alpha_{1}\omega(\nu_{1})\omega(\nu_{2})-\omega(\widetilde{\nu}_{1})\omega(\widetilde{\nu}_{2})\\
 & =\alpha_{2}+\alpha_{1}\Sigma_{\nu_{1}}(\nu_{2})
\end{align*}
because $\omega(\widetilde{\nu}_{1})=\alpha_{1}\omega(\nu_{1})$ and
$\omega(\widetilde{\nu}_{2})=\omega(\nu_{2})$. The calculation of
$\Sigma_{\widetilde{\nu}_{2}}(\widetilde{\nu}_{1})$ is somewhat simpler
because there is no extra pair $c'>d'$ such that $\widetilde{\nu}_{1}^{(j)}(c')\widetilde{\nu}_{2}^{(j)}(d')\ne0$.
\end{proof}
\begin{rem}
The preceding proposition can also be proved using the calculation
of $\Sigma$ described in Remark \ref{rem:calculation of Sigma} though
the extra $\alpha_{2}$ crossings needed for $\Sigma_{\widetilde{\nu}_{1}}(\widetilde{\nu}_{2})$
may be difficult to locate.
\end{rem}

\begin{rem}
\label{rem:sigma between two stretched honeycombs}Suppose that $\nu_{1},\nu_{2},$
and $\nu_{3}$ are three rigid tree honeycombs such that $\Sigma_{\nu_{2}}(\nu_{1})=\Sigma_{\nu_{3}}(\nu_{1})=0$.
Let $\widetilde{\nu}_{1}$ and $\widetilde{\nu}_{2}$ be as in Proposition
\ref{prop:sigma of stretch vs inflated honey}, and construct a third
honeycomb $\widetilde{\nu}_{3}$ that is homologous to $\nu_{3}$
and compatible with the puzzle of $\nu_{1}$ (the same way that $\widetilde{\nu}_{2}$
was constructed). Then the entire overlay $\widetilde{\nu}_{2}+\widetilde{\nu}_{3}$
is homologous to $\nu_{2}+\nu_{3},$ and therefore
\begin{equation}
\Sigma_{\widetilde{\nu}_{2}}(\widetilde{\nu}_{3})=\Sigma_{\nu_{2}}(\nu_{3})\text{ and }\Sigma_{\widetilde{\nu}_{3}}(\widetilde{\nu}_{2})=\Sigma_{\nu_{3}}(\nu_{2}).\label{eq:sigma of two stretched honeys}
\end{equation}
\end{rem}

There is a second way to construct a rigid tree honeycomb from a rigid
overlay of two tree honeycombs. We show in Theorem \ref{thm:there exists a simple simple degeneration}
that every rigid tree honeycomb with weight at least $2$ can be obtained
from this construction applied to an overlay of honeycombs with smaller
weights.
\begin{thm}
\label{thm:basic inductive construction}Suppose that $\nu_{1}$ and
$\nu_{2}$ are rigid tree honeycombs, $\Sigma_{\nu_{1}}(\nu_{2})=\sigma>0$
and $\Sigma_{\nu_{2}}(\nu_{1})=0$. Then there exists a rigid tree
honeycomb $\widehat{\nu}_{1}$ compatible with the puzzle of $\nu_{2}$,
such that the parts of the support of $\widehat{\nu}_{1}$ contained
in the white puzzle pieces are simply translates of the corresponding
parts of the support of $\nu_{1}$. Moreover, the exit pattern of
$\widehat{\nu}_{1}$ is the same as the exit pattern of $\nu_{1}+\sigma\nu_{2}$.
\end{thm}

\begin{proof}
Fix a root edge $e$ of $\nu_{1}$ that is disjoint from the support
of $\nu_{2}$ and orient all the other edges of $\mu_{1}$ away from
$e$, that is, in the direction given by descendance paths. We construct
a collection $\mathcal{C}$ of segments and rays containing:
\begin{enumerate}
\item the edges of the puzzle of $\nu_{2}$,
\item the segments obtained as translates of segments in the support of
$\nu_{1}$ that are contained in the interior of a white puzzle piece,
and
\item the inflation of every branch point of $\nu_{1}+\nu_{2}$ that is
in the interior of a side of a white puzzle piece. That is, if $Z$
is a branch point in the interior of an edge $AB$ of $\nu_{2}$,
and if $A'B',A''B''$ are the two white edges of the inflation of
$AB$, then we include in our collection the segment $Z'Z''$ joining
the points on $A'B'$ and $A''B''$ that are the translates of $Z$.
\end{enumerate}
We orient all the segments in $\mathcal{C}$, with the exception of
the translate of the root edge $e$, as follows:
\begin{enumerate}
\item the edges of every dark gray parallelogram point to the acute angles
of the parallelogram,
\item the translate of a (part of) an edge $f$ of $\nu_{1}$ other than
$e$ is pointing in the same direction as $f$,
\item the segments $Z'Z''$ described above are oriented as in Figure \ref{fig:Orientation-of ZZ}
where the six possible configurations (up to rotations) are shown.
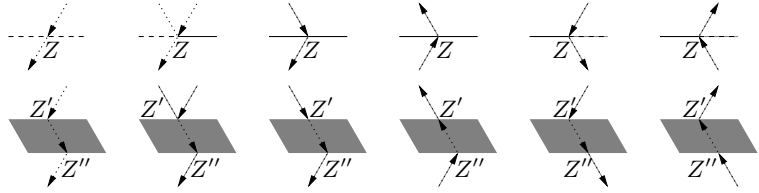
\begin{figure}
\begin{picture}(250,100)
\put(0,0){\includegraphics[scale=.5]{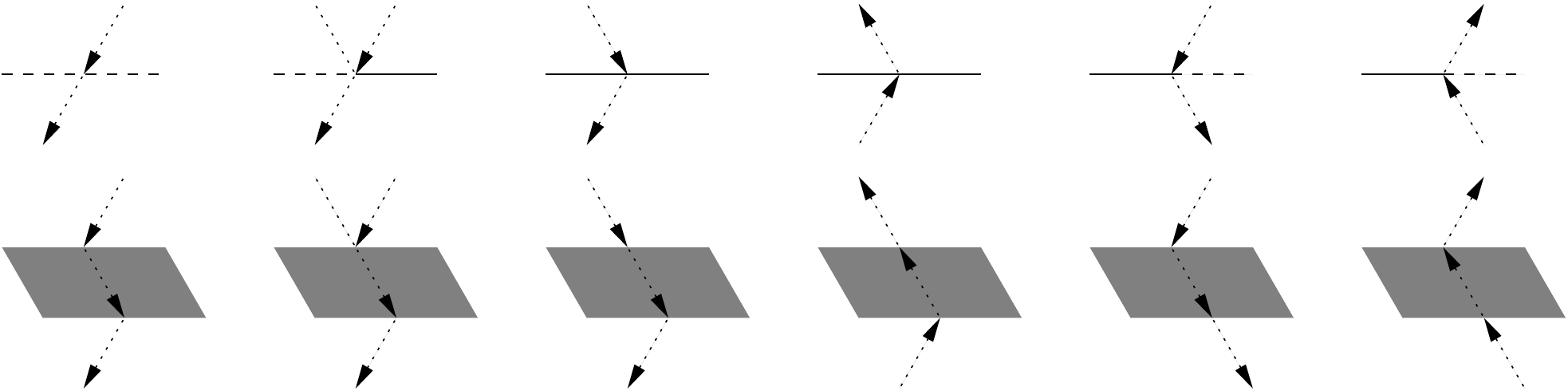}}
{\small
\put(13,49){$Z$}
\put(62,49){$Z$}
\put(111.5,49){$Z$}
\put(162,49){$Z$}
\put(206,49){$Z$}
\put(255,49){$Z$}

\put(8,27){$Z'$}
\put(50,27){$Z'$}
\put(114,27){$Z'$}
\put(163,27){$Z'$}
\put(206,27){$Z'$}
\put(255,27){$Z'$}

\put(20,4){$Z''$}
\put(69,4){$Z''$}
\put(119,4){$Z''$}
\put(169,4){$Z''$}
\put(210,4){$Z''$}
\put(257,4){$Z''$}
}
\end{picture}\caption{\label{fig:Orientation-of ZZ}Orientation of $Z'Z''$}

\end{figure}
(Here, the horizontal edges are in the support of $\nu_{2}$ and are
represented by dashed lines. The support of $\nu_{1}$ is represented
by dotted lines and the solid lines represent common parts of the
two supports. No other configurations are possible because $\Sigma_{\nu_{2}}(\nu_{1})=0$;
see Remark \ref{rem:calculation of Sigma}.)
\end{enumerate}
We treat $\mathcal{C}$ as the prospective support of a honeycomb
and, accordingly, we say that a point where three or more of these
segments meet is a branch point. Denote by $e'=xy$ the translate
of $e$ and construct (as in Section \ref{sec:Puzzles-and-duality})
a tree $T_{0}$ by following the descendance paths in $\mathcal{C}$.
Thus, the vertices of $T_{0}$ are $x,y$, and the gentle paths $e'e_{1}\dots e_{n}$,
where $e_{j}\in\mathcal{C},$ $e_{1}$ points away from $x$ or from
$y$, and $e_{j}$ points away from $e_{j-1}$ for $j=2,\dots,n$.
If $e'e_{1}\dots e_{n}$ is such a path, then it is joined to $e'e_{1}\dots e_{n-1}$
by an edge labeled $j$ if $e_{n}$ is parallel to $w_{j}$. In addition,
$x$ is joined to $y$ by an edge labeled $j$ if $e'$ is parallel
to $w_{j}$. Finally, either $x$ or $y$ is joined to $e_{1}$ by
an edge labeled using the same rule. 

We first argue that the tree $T_{0}$ is finite. In the contrary case,
there would exist an infinite gentle oriented path. Because $\mathcal{C}$
is finite, some edge of $\mathcal{C}$ must appear twice in this path,
and we conclude that there exists an oriented gentle loop $e_{1}\dots e_{n}$
formed by edges in $\mathcal{C}$. Some of the edges $e_{j}$, but
no two consecutive ones, may be of the form $Z'Z''$ as in the case
(3) above. Each other $e_{j}$ is the translate of some edge $f_{j}$
of $\nu_{1}+\nu_{2}$. These edges $f_{j}$ form a loop in the support
of $\nu_{1}+\nu_{2}.$ We claim that this loop is evil, thus contradicting
the rigidity of $\nu_{1}+\nu_{2}$. This is verified by considering
any two, three, or four consecutive edges that form an oriented gentle
path in $\mathcal{C}$. Such a sequence of edges arises from a branch
point of $\nu_{1}+\nu_{2}$ and the relevant evil turns are observed
by examining the situations depicted in Figures \ref{fig:Orientation-of ZZ}
and \ref{fig:11}.
\begin{figure}
\includegraphics[scale=0.6]{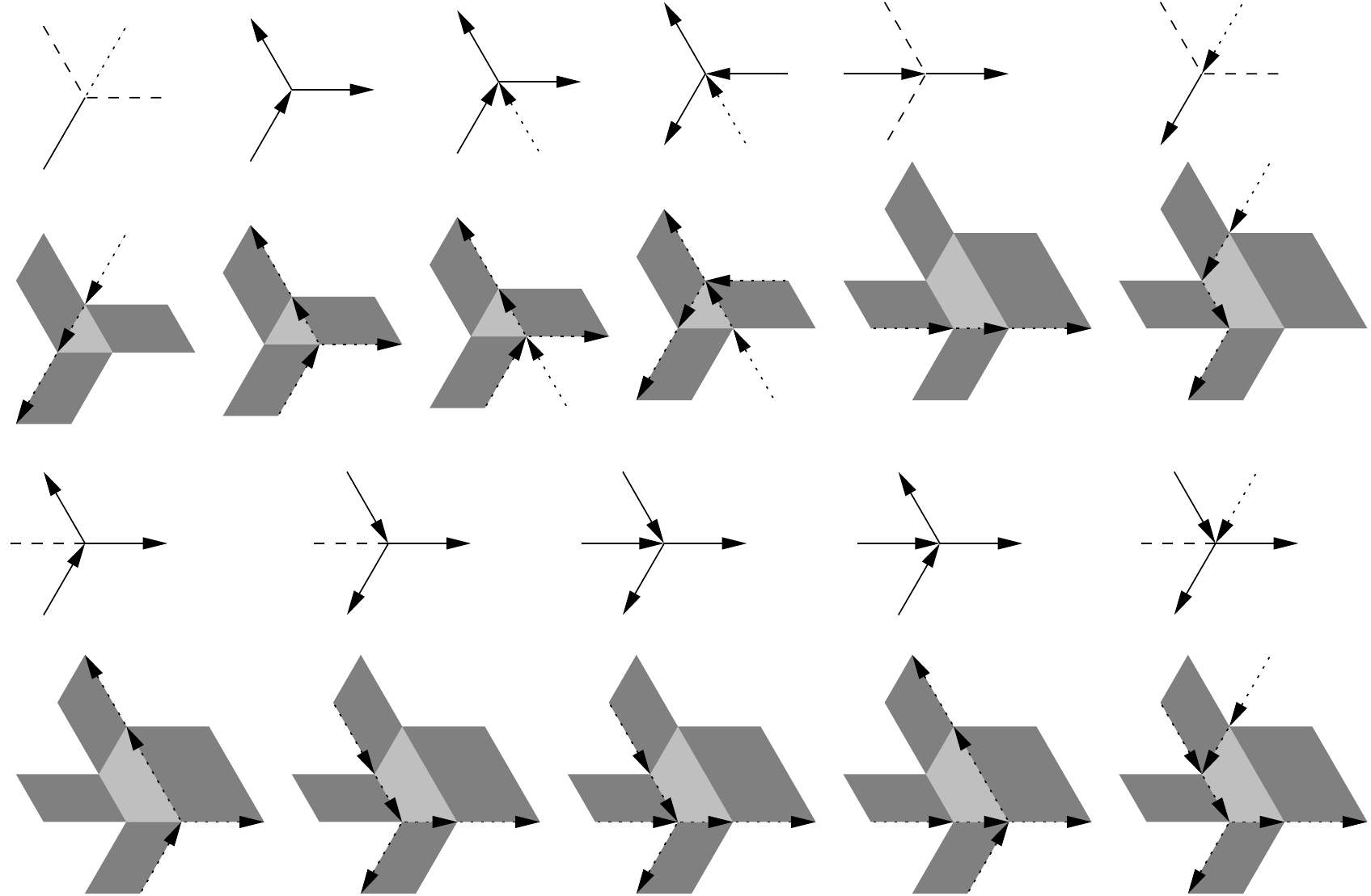}

\caption{\label{fig:11}}

\end{figure}
 All other possibilities for the branch points of $\nu_{1}+\nu_{2}$
are precluded by the condition $\Sigma_{\nu_{2}}(\nu_{1})=0$.

We also note that ends of the tree $T_{0}$ are precisely those paths
$e'e_{1}\dots e_{n}$ such that $e_{n}$ is either an outgoing ray
in the puzzle of $\nu_{2}$ or the translate of a ray of $\nu_{1}$.
We define now multiplicities $\widehat{\nu}_{1}(e)$ for each $e\in\mathcal{C}$
by setting $\widehat{\nu}_{1}(e')=1$ and, for other edges, $\widehat{\nu}_{1}(e)$
is the number of gentle descendance paths $e'e_{1}\dots e_{n}$ satisfying
$e_{n}=e$. 

Define now a honeycomb $\widehat{\nu}_{1}$ by setting $\widehat{\nu}_{1}(e')=1$
and, given an arbitrary edge $f$ in the collection $\mathcal{C}$,
$\widehat{\nu}_{1}(f)$ is the number of distinct oriented descendance
paths from $e'$ to $f$. The balance condition at each branch point
is checked by looking at the various cases from Figures \ref{fig:Orientation-of ZZ}
and \ref{fig:11}. It is clear that $\widehat{\nu}_{1}$ is a tree
honeycomb, and Theorem \ref{thm:rigid immersion} implies that it
is rigid. Since the edges in the support of $\widehat{\nu}_{1}$ are
either translates of the original edges in the support of $\nu_{1}+\nu_{2}$
or segments of the form $Z'Z''$, it follows that the support of the
honeycomb $(\widehat{\nu}_{1})_{\nu_{2}}$ (obtained by deflating
the puzzle of $\nu_{2}$) is contained in the support of $\nu_{1}+\nu_{2}$.
Thus $(\widehat{\nu}_{1})_{\nu_{2}}$ is a rigid honeycomb of the
form $c_{1}\nu_{1}+c_{2}\nu_{2}$ for some $c_{1},c_{2}\ge0$. Clearly,
$(\widehat{\nu}_{1})_{\nu_{2}}$ assigns unit mass to $e'$, so $c_{1}=1$.
To determine $c_{2}$, we note that, since $(\widehat{\nu}_{1})_{\nu_{2}}$
and $\nu_{1}+c_{2}\nu_{2}$ have same exit pattern, relation (\ref{eq:sigms(mu,mu)})
shows that
\begin{align*}
-1 & =\Sigma_{(\widehat{\nu}_{1})_{\nu_{2}}}((\widehat{\nu}_{1})_{\nu_{2}})=\Sigma_{\nu_{1}+c_{2}\nu_{2}}(\nu_{1}+c_{2}\nu_{2})\\
 & =\Sigma_{\nu_{1}}(\nu_{1})+c_{2}\Sigma_{\nu_{1}}(\nu_{2})+c_{2}\Sigma_{\nu_{2}}(\nu_{1})+c_{2}^{2}\Sigma_{\nu_{2}}(\nu_{2})\\
 & =-1+0+c_{2}\sigma-c_{2}^{2}.
\end{align*}
This equality is only satisfied by $c_{2}=0$ and $c_{2}=\sigma$.
If $\sigma>0$, the alternative $c_{2}=0$ can be dismissed because
in this case one of the first two configurations depicted in Figure
\ref{fig:Orientation-of ZZ} or the last configuration from Figure
\ref{fig:11} occurs, thus ensuring that the honeycomb $\widehat{\nu}_{1}$
assigns positive multiplicity to some white edges of the puzzle of
$\nu_{2}$, and so $c_{2}>0$.
\end{proof}
\begin{cor}
\label{cor:pattern of nu1+sigma nu2 comes from rigid tree}Suppose
that $\nu_{1}$ and $\nu_{2}$ are rigid tree honeycombs such that
$\Sigma_{\nu_{2}}(\nu_{1})=0$, and set $\sigma=\Sigma_{\nu_{1}}(\nu_{2})$.
Then there exist a rigid tree honeycomb $\mu_{1}$ \emph{(}respectively,
$\mu_{2}$\emph{) }that has the same exit pattern as $\nu_{1}+\sigma\nu_{2}$
\emph{(}respectively, $\sigma\nu_{1}+\nu_{2}$\emph{)}.
\end{cor}

\begin{proof}
The honeycomb $\mu_{1}=\widehat{\nu}_{1}$ satisfies the requirement.
To prove the existence of $\mu_{2}$, we argue as in Corollary \ref{cor:multiple of a rigid tree, mirror image}.
Thus, we consider the reflections $\nu'_{1}$ and $\nu'_{2}$ of $\nu_{1}$
and $\nu_{2}$ in a line parallel to $w_{1}$. These rigid tree honeycombs
satisfy $\Sigma_{\nu_{1}'}(\nu_{2}')=0$ and $\Sigma_{\nu_{2}'}(\nu_{1}')=\sigma$,
so there exists a rigid tree honeycomb $\mu_{2}'$ with the same exit
pattern as $\sigma\nu'_{1}+\nu_{2}'$. We obtain $\mu_{2}$ as the
reflection of $\mu_{2}'$ in a line parallel to $w_{1}$.
\end{proof}
Once we know that a pattern comes from a rigid tree honeycomb $\mu$,
one can find a honeycomb homologous to $\mu$ using the flat region
construction from Section \ref{sec:Puzzles-and-duality}.  This allows
for a fairly efficient construction of measures $\mu_{1}$ and $\mu_{2}$
with the exit patterns prescribed by Corollary \ref{cor:pattern of nu1+sigma nu2 comes from rigid tree}.
\begin{rem}
\label{rem:every tree appears in some dual}With the notation of Theorem
\ref{thm:basic inductive construction}, relation (\ref{eq:mu* when mu is compatible with the puzzle of nu})
applies to the honeycomb $\widehat{\nu}_{1}$ to yield
\begin{align*}
(\widehat{\nu}_{1})^{*} & =((\widehat{\nu}_{1})_{\nu_{2}})^{*}+((\widehat{\nu}_{1})_{\nu_{2}^{*}})^{*}\\
 & =(\nu_{1}+\sigma\nu_{2})^{*}+((\widehat{\nu}_{1})_{\nu_{2}^{*}})^{*}.
\end{align*}
 The support of $(\widehat{\nu}_{1})_{\nu_{2}^{*}}$ inside the triangle
$\Delta_{*}$ is always contained in the support of $\nu_{2}^{*}$.
In many cases, the support of $(\widehat{\nu}_{1})_{\nu_{2}^{*}}$
contains the entire support of $\nu_{2}^{*}$ (inside $\Delta_{*}$).
When this occurs, the honeycomb $((\widehat{\nu}_{1})_{\nu_{2}^{*}})^{*}$
is in fact a tree honeycomb with the same exit pattern as $\nu_{2}$,
so $((\widehat{\nu}_{1})_{\nu_{2}^{*}})^{*}$ is homologous (after
a $60^{\circ}$ rotation) to $\nu_{2}$. Here is one particular way
to form overlays with this special property. Given an arbitrary rigid
tree honeycomb $\nu_{2}$, find another rigid tree honeycomb $\nu_{1}$
such that $\omega(\nu_{1})=3$, $\nu_{2}$ is clockwise from $\nu_{1}$,
and every exit ray of $\nu_{2}$ crosses an edge of $\nu_{1}$. (See
Figure \ref{fig:Clockwise-overlay-with-bigsig} for a particular case.
The honeycomb $\nu_{2}$ is pictured in black and $\nu_{2}$ in red.
The second part of the figure shows the puzzle of $\nu_{2}$ and the
support of $\widehat{\nu}_{1}$.) Since the support of $\widehat{\nu}_{1}$
contains a part of each incoming ray of the puzzle of $\nu_{2}$,
this support must also contain the descendants of these incoming rays,
and these include all the parallelogram edges other than the incoming
rays. It is clear that the support of $(\widehat{\nu}_{1})_{\nu_{2}^{*}}$
contains all the edges of the light gray puzzle pieces. We record
this fact as follows.
\begin{figure}
\includegraphics[scale=0.7]{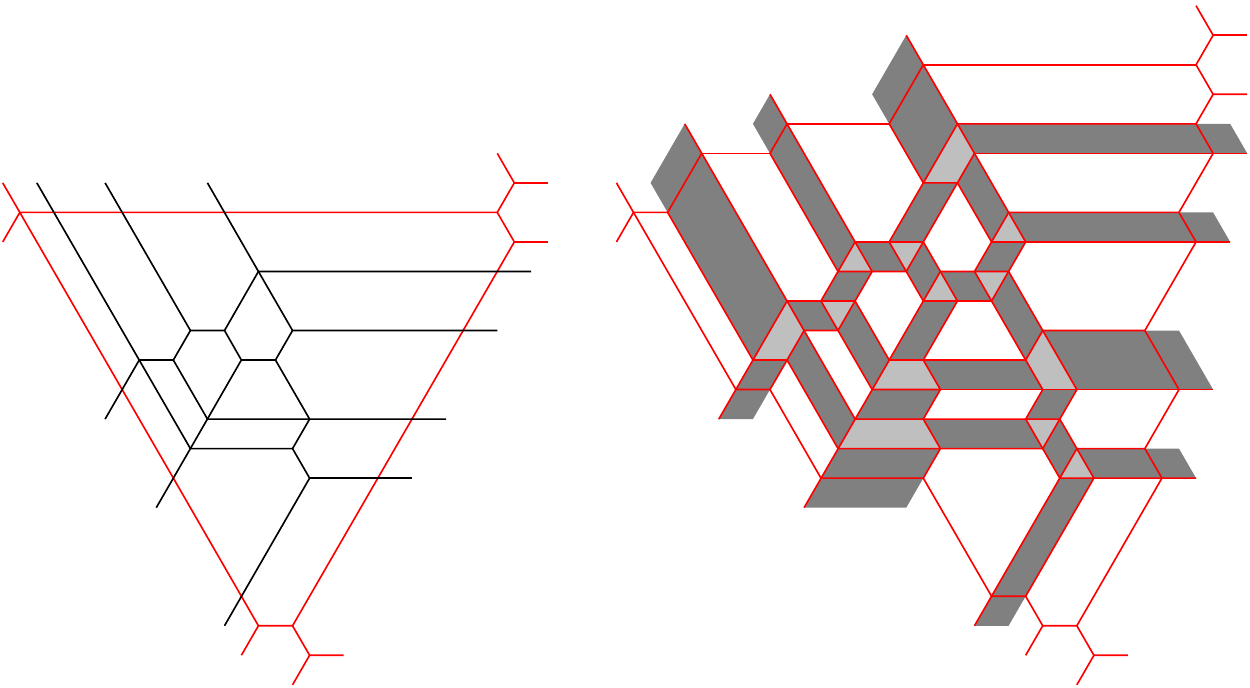}

\caption{\label{fig:Clockwise-overlay-with-bigsig}Clockwise overlay with $\Sigma_{\nu_{1}}(\nu_{2})=\omega(\nu_{1})\omega(\nu_{2})$}
\end{figure}
\end{rem}

\begin{prop}
\label{prop:every mu appears in a dual} Let $\mu\in\mathcal{M}$
be an arbitrary rigid tree honeycomb. There exist rigid tree honeycombs
$\nu\in\mathcal{M}$ and $\rho\in\mathcal{M}_{*}$ such that $\rho\le\nu^{*}$
and $\rho$ is homologous to a $60^{\circ}$ rotation of $\mu$. 
\end{prop}

\begin{rem}
\label{rem:general honeycomb arising from a ray of a rigid honey}
At this point, we have enough information to calculate the exit pattern
of any of the rigid tree honeycombs whose support is contained in
the union of the boundaries of the dark gray parallelograms in a rigid
puzzle. Thus, suppose that $\nu$ is a rigid honeycomb, and write
it as $\nu=\sum_{j=1}^{n}c_{j}\nu_{j}$, where each $\nu_{j}$ is
a rigid tree honeycomb and $\Sigma_{\nu_{i}}(\nu_{j})=0$ for $i>j$.
Replacing $c_{j}$ by $1$ does not alter the structure of the puzzle
of $\nu$, so we assume that $c_{j}=1$ for $j=1,\dots,n$. Let $f$
be a ray in the support of $\nu$, and let $\alpha_{j}\ge0$ be the
corresponding exit multiplicity of $\nu_{j}$. The main observation
is that the rigid honeycomb $\mu$, supported by the edges of parallelograms
in the puzzle of $\nu$ and rooted in the incoming ray in the inflation
of $f$, can be obtained as a result of a sequence of operations that
we now describe. Suppose for the moment that $\alpha_{1}>0$.
\begin{enumerate}
\item Construct rigid tree honeycombs $\nu_{1}(1),\dots,\nu_{n}(1)$, compatible
with the puzzle of $\nu_{1}$ as follows: $\nu_{1}(1)$ is supported
by the puzzle edges and is rooted in the inflation of $f$ while,
for $j\ge2$, $\nu_{j}(1)$ is obtained by translating the support
of $\nu_{j}$ along with the white pieces as in Theorem \ref{thm:stretched honeycomb}.
As seen above (Proposition \ref{prop:sigma of stretch vs inflated honey}
and Remark \ref{rem:sigma between two stretched honeycombs}), these
new tree honeycombs form again a rigid overlay.
\item Construct rigid tree honeycombs $\nu_{2}(2),\dots,\nu_{n}(2)$, compatible
with the puzzle of $\nu_{2}(1)$, as follows: $\nu_{2}(2)$ is obtained
by translating the support of $\nu_{1}(1)$ as in Theorem \ref{thm:basic inductive construction}
and, for $j\ge3$, $\nu_{j}(2)$ is obtained using the procedure of
Theorem \ref{thm:stretched honeycomb}. These tree honeycombs also
form a rigid overlay.
\item For $k=3,\dots,n$, construct rigid tree honeycombs $\nu_{k}(k),\nu_{k+1}(k),\dots,\nu_{n}(k)$.
The construction is done by induction and the inductive step is the
procedure described in (2). Thus, supposing that these honeycombs
are constructed for some $k<n$, we construct the puzzle of $\nu_{k+1}(k)$
and we construct $\nu_{k+1}(k+1)$ by an application of Theorem \ref{thm:basic inductive construction}
to $\nu_{k}(k)$ and we construct $\nu_{k+1}(j)$ by an application
of Theorem \ref{thm:stretched honeycomb}. to $\nu_{k}(j)$, $j=k+2,\dots,n$.
\item The rigid tree honeycomb $\nu_{n}(n)$ is precisely the desired honeycomb
$\mu$. 
\end{enumerate}
One can calculate the exit patterns of all the honeycombs constructed
above, in particular, the exit pattern of $\mu$.
\begin{enumerate}
\item After the first operation described above, the honeycomb $\nu_{j}(1)$
is homologous to $\nu_{j}$ for $j\ge2$ and the exit pattern of $\nu_{1}(1)$
is obtained by multiplying the exit pattern of $\nu_{1}$ by $\alpha_{1}$,
except for the ray $f$ that gives rise to two exit multiplicities
equal to $\alpha_{1}^{2}-1$ and $1$. Moreover, we have $\Sigma_{\nu_{1}(i)}(\nu_{1}(j))=\Sigma_{\nu_{i}}(\nu_{j})$
for $i,j\ge2,$ and $\Sigma_{\nu_{1}(1)}(\nu_{1}(j))=\alpha_{j}+\alpha_{1}\Sigma_{\nu_{1}}(\nu_{j})$
for $j\ge2$ (see Proposition \ref{prop:sigma of stretch vs inflated honey}).
\item After the second operation, the honeycomb $\nu_{2}(j)$ is homologous
to $\nu_{j}$ for $j\ge3$ and the exit pattern of $\nu_{2}(2)$ is
the same as the one for $\nu_{1}(1)+\sigma\nu_{1}(2)$, where $\sigma=\Sigma_{\nu_{1}(1)}(\nu_{1}(2))$
(see Theorem \ref{thm:basic inductive construction}). Moreover, we
have $\Sigma_{\nu_{2}(i)}(\nu_{2}(j))=\Sigma_{\nu_{i}}(\nu_{j})$
for $i,j\ge3,$ and an easy calculation shows that 
\[
\Sigma_{\nu_{2}(2)}(\nu_{2}(j))=\Sigma_{\nu_{1}(1)}(\nu_{1}(j))+\sigma\Sigma_{\nu_{1}(2)}(\nu_{1}(j)),\quad j\ge3.
\]
Further operations follow this model and they eventually yield the
exit pattern of $\mu$.
\end{enumerate}
In case $\alpha_{1}=0$, we look for the first nonzero $\alpha_{j}$
and we can simply discard the honeycombs $\nu_{1},\dots,\nu_{j-1}$
because they do not contribute to the exit pattern of $\mu$.
\end{rem}

\begin{example}
\label{exa:rigid overlay of two}The special case of a rigid overlay
of two tree honeycombs is used in Section \ref{sec:de-degeneration}.
Thus, let $\nu_{1}$ and $\nu_{2}$ be rigid tree honeycombs such
that $\Sigma_{\nu_{2}}(\nu_{1})=0$ and $\Sigma_{\nu_{1}}(\nu_{2})=\sigma>0$.
Let $f$ be a ray to which $\nu_{j}$ assigns multiplicity $\alpha_{j}$,
$j=1,2$. The first operation in the preceding remark provides $\nu_{1}(1)$
and $\nu_{1}(2)$ such that $\Sigma_{\nu_{1}(2)}(\nu_{1}(1))=0$ and
$\Sigma_{\nu_{1}(1)}(\nu_{1}(2))=\alpha_{2}+\sigma\alpha_{1}$. The
exit multiplicities assigned by $\nu_{1}(1)$ to the outgoing and
incoming rays in the inflation of $f$ are $\alpha_{1}^{2}-1$ and
$1$, respectively, while $\nu_{1}(2)$ assigns multiplicity $\alpha_{2}$
to the outgoing ray. If $e$ is another ray such that $\nu_{j}(e)=\beta_{j}$,
$j=1,2$, then the corresponding (outgoing) ray is assigned multiplicities
$\alpha_{1}\beta_{1}$ and $\beta_{2}$, respectively, by $\nu_{1}(1)$
and $\nu_{1}(2)$. After the second operation, we obtain a honeycomb
$\mu=\nu_{2}(2)$ with typical exit multiplicities $\alpha_{1}\beta_{1}+(\alpha_{2}+\sigma\alpha_{1})\beta_{2}$,
and with exit multiplicities $\alpha_{1}^{2}-1+(\alpha_{2}+\sigma\alpha_{1})\alpha_{2}$
and $1$ assigned to the outgoing and incoming rays in the inflation
of $f$ (in the puzzle of $\nu_{1}+\nu_{2}$). This process is illustrated
in Figure \ref{fig:4-14}, where $\nu_{1}$ is drawn in black, $\nu_{2}$
in red, and the numbers represent exit multiplicities. For this example,
$\sigma=2$, $\alpha_{1}=2$, and $\alpha_{2}=1$. For comparison,
we show in Figure \ref{fig:4-14a} the same final honeycomb constructed
directly on the puzzle of $\nu_{1}+\nu_{2}$.
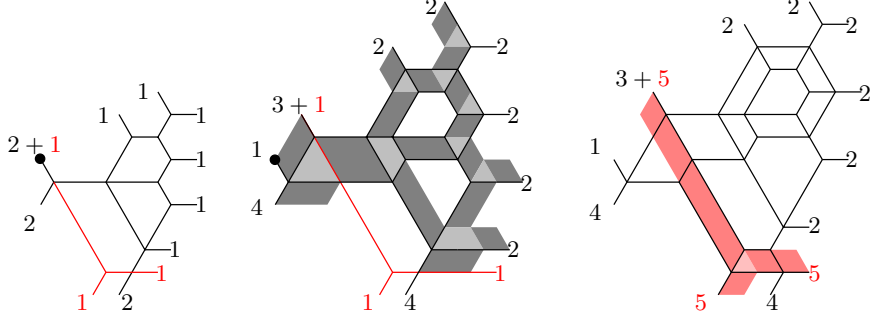
\begin{figure}
\begin{picture}(320,100)
\put(0,0){\includegraphics[scale=1]{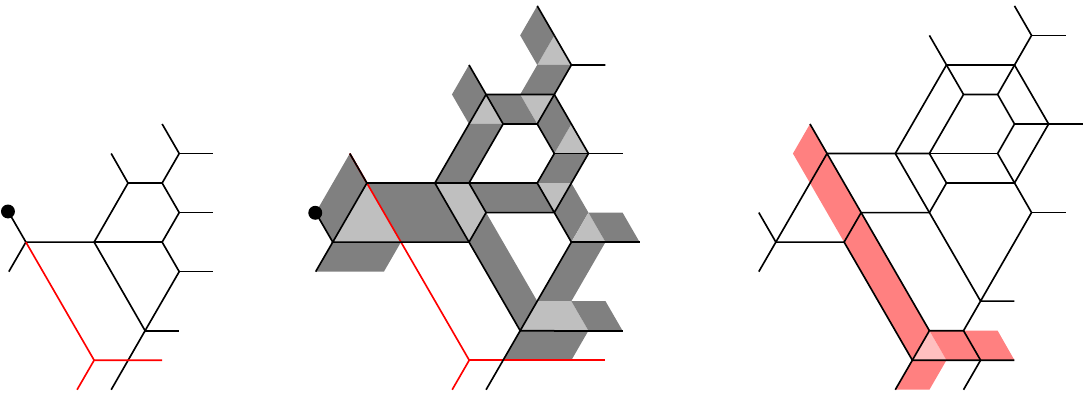}}
{\small
\put(-10,55){$2+\color{red}1$}
\put(-4,25){$2$}
\put(33,-5){$2$}
\put(16,-5){$\color{red}1$}
\put(51,15){$1$}
\put(61,32){$1$}
\put(61,49){$1$}
\put(61,66){$1$}
\put(24,66){$1$}
\put(39,75){$1$}
\put(46,6){$\color{red}1$}

\put(124,-5){$\color{red}1$}
\put(174,6){$\color{red}1$}
\put(90,70){$3+\color{red}1$}
\put(82,53){$1$}
\put(82,31){$4$}
\put(140,-5){$4$}
\put(179,15){$2$}
\put(184,40){$2$}
\put(179,66){$2$}
\put(175,92){$2$}
\put(148,107){$2$}
\put(128,92){$2$}

\put(250,-5){$\color{red}5$}
\put(277,-5){$4$}
\put(293,6){$\color{red}5$}
\put(293,24){$2$}
\put(307,49){$2$}
\put(312,74){$2$}
\put(307,100){$2$}
\put(285,107){$2$}
\put(260,100){$2$}
\put(220,80){$3+\color{red}5$}
\put(210,54){$1$}
\put(210,29){$4$}
}
\end{picture}\caption{The measures $\nu_{1},\nu_{2},\widetilde{\nu}_{1},\widetilde{\nu}_{2}$,
and the final honeycomb for Example \ref{exa:rigid overlay of two}\label{fig:4-14} }

\end{figure}

\begin{figure}
\includegraphics{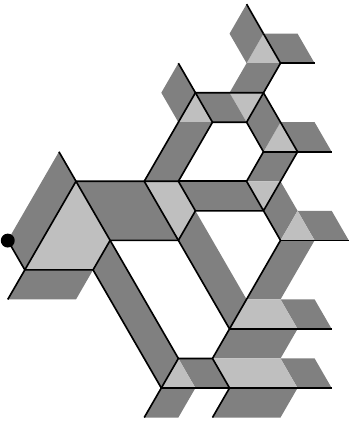}

\caption{The same honeycomb constructed using the puzzle of $\nu_{1}+\nu_{2}$\label{fig:4-14a}}

\end{figure}
\end{example}

\section{Degeneration of a rigid tree honeycomb\label{sec:Degeneration-of-a rigid honey}}

Consider an arbitrary rigid honeycomb $\nu\in\mathcal{M}$ and write
its dual as 
\[
\nu^{*}=c_{1}\rho_{1}+\cdots+c_{n}\rho_{n},
\]
where $n={\rm root}(\nu^{*})={\rm exit}(\nu)-{\rm root(\nu)-2},$
$c_{1},\dots,c_{n}>0$, and each $\rho_{j}$ is a rigid tree honeycomb
in $\mathcal{M}_{*}$ (see Theorem \ref{thm:exit(mu)=00003Droot(mu) etc}).
As we pointed out before, the honeycombs $\mu\in\mathcal{M}$ that
are homologous to $\nu$ are precisely the ones that satisfy
\begin{equation}
\mu^{*}=t_{1}\rho_{1}+\cdots+t_{n}\rho_{n}\label{eq:mu*(local use)}
\end{equation}
for some $t_{1},\dots,t_{n}>0$. Decreasing one of the coefficients
$t_{j}$ amounts to decreasing the lengths of some of the edges of
$\mu$. If we replace some of the coefficients $t_{j}$ by $0$, the
honeycomb $\mu$ satisfying (\ref{eq:mu*(local use)}) is no longer
homologous to $\nu$. We call it a \emph{degeneration }of\emph{ $\nu$.
}A \emph{simple degeneration }of $\mu$ is obtained by replacing exactly
one of the $t_{j}$ by $0$.  All the degenerations of $\nu$ satisfy
$\omega(\mu)=\omega(\nu)$. Moreover, the exit multiplicities of $\mu$
are sums of one or several (consecutive) exit multiplicities of $\nu$
(see Remark \ref{rem:stuff about mu versus mu*}). For instance, if
all the coefficients $t_{j}$ are replaced by $0$, we obtain the
degeneration $\mu=\omega(\nu)\nu_{0}$, where $\nu_{0}$ is a rigid
tree honeycomb of unit weight. Suppose that $\mu$ is an arbitrary
degeneration of $\nu$. An application of Theorem \ref{thm:exit(mu)=00003Droot(mu) etc}
to $\mu$ and to $\nu$ yields
\begin{align*}
{\rm root(\mu)} & ={\rm exit}(\mu)-{\rm root}(\mu^{*})-2,\\
{\rm root}(\nu) & ={\rm exit}(\nu)-{\rm root}(\nu^{*})-2,
\end{align*}
and subtracting these equalities we obtain the equation
\begin{equation}
{\rm root}(\mu)-{\rm root}(\nu)=[{\rm root}(\nu^{*})-{\rm root}(\mu^{*})]-[{\rm exit}(\nu)-{\rm exit}(\mu)],\label{eq:roots and exits of degeneration}
\end{equation}
 where ${\rm root}(\nu^{*})-{\rm root}(\mu^{*})>0$ and ${\rm exit}(\nu)-{\rm exit}(\mu)\ge0$.

Suppose now that $\nu$ is a rigid tree honeycomb and that $\mu$
is a simple degeneration of $\nu$. Then we have ${\rm root}(\nu^{*})-{\rm root}(\mu^{*})=1$
and (\ref{eq:roots and exits of degeneration}) allows for two possibilities:
\begin{enumerate}
\item ${\rm root}(\mu)=2$ and ${\rm exit}(\mu)={\rm exit}(\nu)$, or
\item ${\rm root}(\mu)=1$ and ${\rm exit}(\mu)={\rm exit}(\nu)-1$.
\end{enumerate}
We show that the first situation arises for at least one of the simple
degenerations of $\nu$.
\begin{thm}
\label{thm:there exists a simple simple degeneration}Suppose that
$\nu\in\mathcal{M}$ is a rigid tree honeycomb such that $\omega(\nu)\ge3$.
Then there exist rigid tree honeycombs $\nu_{1},\nu_{2}\in\mathcal{M}$
satisfying $\Sigma_{\nu_{1}}(\nu_{2})=0$ and such that the $\nu$
has the same exit pattern as either $\nu_{1}+\sigma\nu_{2}$ or $\sigma\nu_{1}+\nu_{2}$,
where $\sigma=\Sigma_{\nu_{2}}(\nu_{1})\ne0$.
\end{thm}

\begin{proof}
The support of $\nu^{*}=c_{1}\rho_{1}+\cdots+c_{n}\rho_{n}$ inside
the triangle $\Delta_{*}$ of size $\omega(\nu)$ has at least one
flat segment $e$ of unit length, dual to a root edge of $\nu$. We
can select $j,k,\ell\in\{1,\dots,n\}$ such that $e$ is also an edge
of $\rho_{j}+\rho_{k}+\rho_{\ell}$ (we may actually need three such
honeycombs, one whose support contains $e$ and two more whose supports
contain just the endpoints of $e$.) Observe that $n\ge5$ because
$\omega(\nu)\ge3$, and thus we can choose $i\in\{1,\dots,n\}\backslash\{j,k,\ell\}.$
Consider the simple degeneration $\mu$ of $\nu$ defined by $\mu^{*}=\nu^{*}-c_{i}\rho_{i}$.
The choice of $i$ ensures that $e$ is an edge of $\mu^{*}$, so
$\mu$ has some edges with multiplicity $1$. 

If ${\rm root}(\mu)=1$, it follows that $\mu$, having an edge of
multiplicity $1$, must itself be a rigid tree honeycomb and we are
in the second situation described above, that is, ${\rm exit}(\mu)={\rm exit}(\nu)-1$.
\marginpar{review this proof}It follows that the nonzero exit multiplicities
of $\nu$ can be arranged in counterclockwise order $\alpha_{1},\dots,\alpha_{n+3}$
(see Theorem  \ref{thm:exit(mu)=00003Droot(mu) etc}) so the nonzero
exit multiplicities of $\mu$ are $\alpha_{1}+\alpha_{2},\alpha_{3},\dots,\alpha_{n+3}$.
We see then that
\begin{align*}
\omega(\mu)^{2} & =-2+(\alpha_{1}+\alpha_{2})^{2}+\sum_{j=3}^{n+3}\alpha_{j}^{2}\\
 & =2\alpha_{1}\alpha_{2}-2+\sum_{j=1}^{n+3}\alpha_{j}^{2}\\
 & =2\alpha_{1}\alpha_{2}+\omega(\nu)^{2},
\end{align*}
and this is impossible because $\omega(\mu)=\omega(\nu)$ and $\alpha_{1}\alpha_{2}\ne0$.
We conclude that we must be in the first situation described above,
that is, ${\rm exit}(\mu)={\rm exit}(\nu)$ and $\mu$ is a sum of
two extreme honeycombs. One of these summands is a tree honeycomb
because an edge of $\mu$ has unit multiplicity, so $\mu=\nu_{1}+\sigma\nu_{2}$
for some rigid tree honeycombs $\nu_{1}$ and $\nu_{2}$, and $\sigma>0$.
Since $\mu$ is rigid, one of the numbers $\Sigma_{\nu_{1}}(\nu_{2})$
and $\Sigma_{\nu_{2}}(\nu_{1})$ is equal to zero. Finally, the equality
${\rm exit}(\mu)={\rm exit(\nu)}$ implies that $\mu$ and $\nu$
have the same exit pattern (see Remark \ref{rem:stuff about mu versus mu*}),
in particular $\Sigma_{\mu}(\mu)=\Sigma_{\nu}(\nu)=-1$. Thus,
\[
-1+\sigma\Sigma_{\nu_{2}}(\nu_{1})+\sigma\Sigma_{\nu_{1}}(\nu_{2})-\sigma^{2}=-1,
\]
so $\sigma$ is equal to either $\Sigma_{\nu_{2}}(\nu_{1})$ or $\Sigma_{\nu_{1}}(\nu_{2})$.
The theorem follows.
\end{proof}
It may happen that all the simple degenerations of $\nu$ yield honeycombs
$\nu_{1}$ and $\nu_{2}$ as in the preceding result. For instance,
this is the case for all rigid tree honeycombs $\nu$ of weight at
most $5$. This however is not true for all rigid tree honeycombs.
The smallest examples correspond to the patterns $2,3,2|2,2,3|2,3,2$
and $2,3,1|2,2,2|2,2,2$. They are pictured in Figure \ref{fig:(A)-and-(B)}
\begin{figure}
\includegraphics{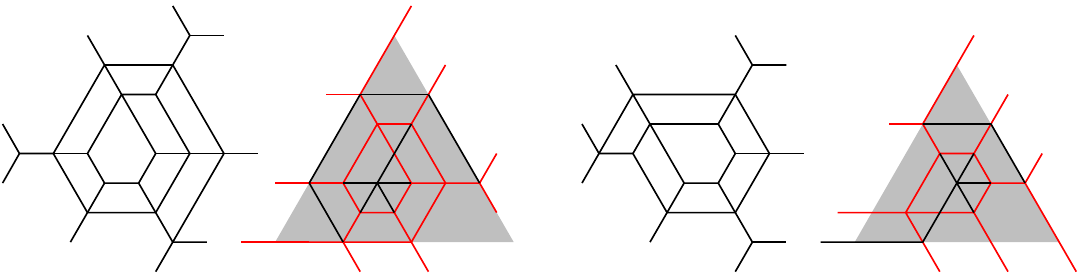}

\caption{(A) and (B) examples\label{fig:(A)-and-(B)}}

\end{figure}
 along with their duals and one particular extreme rigid summand in
the dual (drawn in red) that yields a simple degeneration unlike the
ones envisaged in the proof of Theorem \ref{thm:there exists a simple simple degeneration}.
One of the following situations arises in these examples:
\begin{enumerate}
\item [(A)]A simple degeneration $\mu$ of $\nu$ has ${\rm exit}(\mu)={\rm exit}(\nu)$
and the decomposition $\mu=m_{1}\nu_{1}+m_{2}\nu_{2},$ where $\nu_{1}$
and $\nu_{2}$ are rigid tree honeycombs, has integer coefficients
$m_{1},m_{2}\ge2$.
\item [(B)] A simple degeneration $\mu$ of $\nu$ has ${\rm exit}(\mu)={\rm exit}(\nu)-1$.
In this case, $\mu=c\nu_{0}$ for some rigid tree honeycomb $\nu_{0}$
and some integer $c\ge2$.
\end{enumerate}
In both cases, all the edges of the dual honeycomb $\mu^{*}$ have
length greater than $1$. Later, we describe all the rigid tree honeycombs
$\nu$ that have a degeneration of one of these two types (see Theorems
\ref{thm:overlay with coefficients>1} and \ref{thm:higher multiples of a rigid tree}).
Here we find necessary conditions for such honeycombs. We start with
case (A). In this situation, $\mu$ has the same exit pattern as $\nu$
and so $\Sigma_{\mu}(\mu)=\Sigma_{\nu}(\nu)=-1$. Suppose without
loss of generality that $\Sigma_{\nu_{2}}(\nu_{1})=0$ and calculate
\[
\Sigma_{\mu}(\mu)=-m_{1}^{2}-m_{2}^{2}+\sigma m_{1}m_{2},
\]
where $\sigma=\Sigma_{\nu_{1}}(\nu_{2})$. Thus, the pair $(m_{1},m_{2})$
is a solution of the quadratic Diophantine equation
\begin{equation}
m_{1}^{2}+m_{2}^{2}-\sigma m_{1}m_{2}=1.\label{eq:Diophantine c1 and c2}
\end{equation}
 In case (B), we arrange the nonzero exit multiplicities of $\nu$
in a counterclockwise sequence $\alpha_{1},\dots,\alpha_{m}$ such
that the nonzero exit multiplicities of $\mu$ are $\alpha_{1}+\alpha_{2},\alpha_{3},\dots,\alpha_{m}$.
Note that the exit multiplicities of $\nu_{0}$ are $\beta_{2}=(\alpha_{1}+\alpha_{2})/c,\beta_{3}=\alpha_{3}/c,\dots,\beta_{m}=\alpha_{m}/c$.
We have $\Sigma_{\mu}(\mu)=\Sigma_{c\nu_{0}}(c\nu_{0})=-c^{2}$ and
an application of (\ref{eq:sigms(mu,mu)}) to $\nu$ and to $\mu$
yields
\[
\sum_{j=1}^{m}\alpha_{j}^{2}=\omega(\nu)^{2}+2,\quad(\alpha_{1}+\alpha_{2})^{2}+\sum_{j=3}^{m}\alpha_{j}^{2}=\omega(\mu)^{2}+2c^{2}.
\]
Recalling that $\omega(\mu)=\omega(\nu)$, subtract the two equalities
to obtain
\[
\alpha_{1}\alpha_{2}=c^{2}-1.
\]
Now substitute $\beta_{2}c-\alpha_{1}$ for $\alpha_{2}$ to obtain
\begin{equation}
c^{2}+\alpha_{1}^{2}-\beta_{2}\alpha_{1}c=1,\label{eq:Dioph c alpha1}
\end{equation}
and similarly,
\begin{equation}
c^{2}+\alpha_{2}^{2}-\beta_{2}\alpha_{2}c=1.\label{eq:Dioph c alpha2}
\end{equation}
These equations are of the same kind as ($\ref{eq:Diophantine c1 and c2})$
with $\beta_{2}$ in place of $\sigma$.

We discuss briefly the structure of the nonnegative integer solutions
$(p,q)$ of the equation
\begin{equation}
p^{2}+q^{2}-\sigma pq=1\label{eq:Dioph p^2 +q^2}
\end{equation}
for a given integer $\sigma\ge1$. If $\sigma=1,$ the equation can
be rewritten as 
\[
(p-q)^{2}+pq=1,
\]
and one immediately deduces that the only nonnegative integer solutions
are $(0,1)$, $(1,0)$, and $(1,1)$.

If $\sigma=2$, the equation becomes $(p-q)^{2}=1$, so the nonnegative
integer solutions are $(n,n+1)$ and $(n+1,n)$ for $n=0,1,\dots$. 

If $\sigma\ge3$, a simple substitution transforms (\ref{eq:Dioph p^2 +q^2})
into a standard Pell equation (see \cite[Chapter 8]{levq}). It is
easier however to treat the equation directly. Denote by $\xi$ and
$\xi'$ the two solutions of the quadratic equation $\xi^{2}-\sigma\xi+1=0$,
so $\xi+\xi'=\sigma$ and $\xi\xi'=1$. We set
\[
\mathbb{Z}[\xi]=\{p-\xi q:p,q\in\mathbb{Z}\},
\]
and we define the multiplicative function (sometimes called norm)
$N:\mathbb{Z}[\xi]\to\mathbb{Z}$ by
\[
N(p-\xi q)=(p-\xi q)(p-\xi'q)=p^{2}+q^{2}-\sigma pq,\quad p-\xi q\in\mathbb{Z}[\xi].
\]
 We need to study the multiplicative group $G=\{z\in\mathbb{Z}[\xi]:N(z)=1\}$.
This group is generated, for instance, by $-1$ and $\xi'$. Define
a sequence $\{p_{n}:n=0,1,\dots\}$ of nonnegative integers by setting
\begin{equation}
p_{0}=0,p_{1}=1,p_{n+1}=\sigma p_{n}-p_{n-1},\quad n\in\mathbb{N}.\label{eq:the Pell sequence defined}
\end{equation}
An induction argument shows that
\[
\xi'^{n}=p_{n+1}-p_{n}\xi\text{ and \ensuremath{\xi'^{-n}=-(p_{n-1}-p_{n}\xi),\quad n\in\mathbb{N}.}}
\]
We deduce that the nonnegative integer solutions of (\ref{eq:Dioph p^2 +q^2})
are precisely the pairs $(p_{n},p_{n+1})$ and $(p_{n+1},p_{n})$
for $n=0,1,\dots$.

We note that the numbers $p_{n}=n$ also satisfy the identity $p_{n+1}=2p_{n}-p_{n-1}$.
In other words, the same description of the solutions of (\ref{eq:Dioph p^2 +q^2})
applies to the case $\sigma=2$.

We note for further use some identities that the sequence $\{p_{n}\}_{n=0}^{\infty}$
satisfies. First, consider the column vectors $v_{n}=\left[\begin{array}{c}
p_{n+1}\\
p_{n}
\end{array}\right]$ that satisfy $v_{n+1}=\sigma v_{n}-v_{n-1}$ for $n>0.$ We deduce
that
\[
\det[v_{n+1},v_{n+2}]=\sigma\det[v_{n+1},v_{n+1}]-\det[v_{n+1},v_{n}]=\det[v_{n},v_{n+1}],
\]
and thus $\det[v_{n},v_{n+1}]$ does not depend on $n$. Evaluating
this determinant for $n=0$ we see that
\[
p_{n+1}^{2}-p_{n}p_{n+2}=\det[v_{n},v_{n+1}]=1,\quad n\ge0.
\]
Equivalently,
\begin{equation}
p_{n+1}^{2}-1=p_{n}p_{n+2},\quad n\ge0.\label{eq:p square -1}
\end{equation}
The inductive argument above is easily seen to yield the more general
equality
\[
\det[v_{n},v_{n+k}]=p_{k},\quad k,n\ge0,
\]
or, equivalently,
\[
p_{n+1}p_{n+k}-p_{n}p_{n+k+1}=p_{k},\quad k,n\ge0.
\]
In Section \ref{sec:de-degeneration} we need the special case $k=2$.
Since $p_{2}=\sigma$, this can be rewritten as
\begin{equation}
p_{n}p_{n+1}-\sigma=p_{n-1}p_{n+2},\quad n\ge1.\label{eq:p_np_n+2-sigma}
\end{equation}
These identities show, for instance, that $p_{n}$ and $p_{n+1}$
are relatively prime and that the greatest common divisor of $p_{n}$
and $p_{n+2}$ is $\sigma$ if $n$ is even.

The first few terms of the sequence $p_{n}$ are 
\[
0,1,\sigma,\sigma^{2}-1,\sigma^{3}-2\sigma,\sigma^{4}-3\sigma^{2}+1,\sigma^{5}-4\sigma^{3}+3\sigma,
\]
and a closed formula for these numbers is
\[
p_{n}=\frac{\xi^{n}-\xi'^{n}}{\xi-\xi'},\quad n\ge0.
\]

\begin{rem}
\label{rem:formula for pn applies to all sigma}If $\sigma=2$ and
$p_{n}$ is defined inductively by (\ref{eq:the Pell sequence defined}),
then $p_{n}=n$. Thus, even in this case, the nonnegative solutions
of (\ref{eq:Dioph p^2 +q^2}) are the pairs $(p_{n},p_{n+1})$ and
$(p_{n+1},p_{n})$. If $\sigma=1$, then the sequence $p_{n}$ defined
inductively by (\ref{eq:the Pell sequence defined}) is periodic and
the nonnegative solutions of (\ref{eq:Dioph p^2 +q^2}) are the pairs
$(p_{n},p_{n+1})$ with $n=0,1,2$. 
\end{rem}

Returning now to the discussion of the degenerations of a rigid tree
honeycomb, we have the following result.
\begin{prop}
\label{prop:possible degenerations} Let $\nu$ be a rigid tree honeycomb
and let $\mu$ be a simple degeneration of $\nu$. Then one of the
following cases occurs:
\begin{enumerate}
\item There exist rigid tree honeycombs $\nu_{1}$ and $\nu_{2}$ such that
$\Sigma_{\nu_{1}}(\nu_{2})=1$, $\Sigma_{\nu_{2}}(\nu_{1})=0$, and
$\mu=\nu_{1}+\nu_{2}.$
\item There exist rigid tree honeycombs $\nu_{1}$ and $\nu_{2}$ such that
$\Sigma_{\nu_{2}}(\nu_{1})=0$, $\Sigma_{\nu_{1}}(\nu_{2})=\sigma>1$,
and $\mu=m_{1}\nu_{1}+m_{2}\nu_{2},$ where $m_{1}$ and $m_{2}$
are consecutive terms in the sequence $\{p_{n}\}_{n=1}^{\infty}$
defined by\emph{ (\ref{eq:the Pell sequence defined}).}
\item There exists a rigid tree honeycomb $\nu_{0}$ that has an exit multiplicity
$\sigma>1$ and there exist three consecutive terms $p_{n},p_{n+1},p_{n+2}$
in the sequence $\{p_{n}\}_{n=1}^{\infty}$ defined by\emph{ (\ref{eq:the Pell sequence defined})},
such that the exit multiplicities of $\mu$ are obtained by replacing
each exit multiplicity $\beta$ of $\nu_{0}$ by $p_{n+1}\beta$ except
for the multiplicity $\sigma$ which is replaced by two consecutive
exit multiplicities equal to $p_{n}$ and $p_{n+2}$, respectively.
\end{enumerate}
\end{prop}

\begin{proof}
Parts (1) and (2) follow immediately from the above discussion of
the equation (\ref{eq:Dioph p^2 +q^2}). For (3), set $\sigma=\beta_{2}$
in equations (\ref{eq:Dioph c alpha1}) and (\ref{eq:Dioph c alpha2})
to see that $c,\alpha_{1}$ and $c,\alpha_{2}$ must be consecutive
pairs of elements in the sequence $p_{n}$. Moreover, $\alpha_{2}=\sigma c-\alpha_{1}$,
showing that $\alpha_{1},c,\alpha_{2}$ are three consecutive terms
of this sequence (in either increasing or decreasing order).
\end{proof}
The proof of Theorem \ref{thm:there exists a simple simple degeneration}
shows that all but at most three of the simple degenerations of a
rigid tree honeycomb fall under either case (1) or case (2) above
with $\min\{m_{1},m_{2}\}=1$. Case (A) above corresponds to (2) with
$\min\{m_{1},m_{2}\}>1$ and (B) corresponds to (3). The following
result shows that at most two degenerations are as in (B).
\begin{thm}
\label{thm:at most two type (B) degenerations} Suppose that $\nu$
is a rigid tree honeycomb. Then:
\begin{enumerate}
\item If $\nu$ has two simple degenerations that are extreme honeycombs,
then $\nu$ has no simple degeneration of the form $m_{1}\nu_{1}+m_{2}\nu_{2}$
with $\nu_{1},\nu_{2}$ rigid tree honeycombs and $m_{1},m_{2}\ge2$.
\item At most two of the simple degenerations of $\nu$ are extreme honeycombs.
\item If two of the simple degenerations of $\nu$ are extreme honeycombs,
then there exists a rigid tree honeycomb $\nu'$, with exit multiplicities
$\delta_{1},\dots,\delta_{k}$, listed in counterclockwise order,
such that:
\begin{enumerate}
\item $\delta_{1}>1$.
\item If the sequence $\{p_{n}\}_{n=0}^{\infty}$ is defined by \emph{(\ref{eq:the Pell sequence defined})}
with $\sigma=\delta_{1}$, then there exists $n>1$ such that $\omega(\mu)=p_{n}p_{n+1}\omega(\nu')$,
and the exit multiplicities of $\mu$, arranged in counterclockwise
order, are either
\[
p_{n}^{2}-1,1,p_{n+1}^{2}-1,p_{n}p_{n+1}\delta_{2},\dots,p_{n}p_{n+1}\delta_{k},
\]
or
\[
p_{n+1}^{2}-1,1,p_{n}^{2}-1,p_{n}p_{n+1}\delta_{2},\dots,p_{n}p_{n+1}\delta_{k},.
\]
where the first three exit multiplicities correspond to parallel rays.
\end{enumerate}
\end{enumerate}
\end{thm}

\begin{proof}
Part (1) follows immediately from part (3) because the nonzero exit
multiplicities of $m_{1}\nu_{1}+m_{2}\nu_{2}$, and hence those of
$\nu$, are at least $2$. As observed already, the exit multiplicities
of a simple degeneration that is an extreme honeycomb are obtained
simply by replacing two neighboring exit multiplicities of $\nu$
by their sum and leaving the others multiplicities unchanged. 

If we write $\nu^{*}=\sum_{j=1}^{n}c_{j}\rho_{j}$ as above, this
situation corresponds to the fact that there exists a ray $I$ such
that $\rho_{j_{0}}(I)>0$ but $\rho_{j}(I)=0$ for every $j\ne j_{0}.$
Suppose that $\nu$ has at least two simple degenerations of this
type. Reordering the honeycombs $\rho_{j}$, we may assume that these
degenerations $\mu_{1}$ and $\mu_{2}$ satisfy $\mu_{1}^{*}=\nu^{*}-c_{1}\rho_{1}$
and $\mu_{2}^{*}=\nu^{*}-c_{2}\rho_{2}$. Consider also the degeneration
$\mu_{0}$ such that $\mu_{0}^{*}=\nu^{*}-c_{1}\rho_{1}-c_{2}\rho_{2}$.
Since there are two distinct rays $I_{1},I_{2}$ to which $\nu^{*}$
assigns positive multiplicity such $\mu_{i}^{*}$ assigns zero multiplicity
to $\mu_{i}$, $i=1,2,$ we see that ${\rm exit}(\mu_{0}^{*})\le{\rm exit}(\mu)-2$.
By Theorem \ref{thm:exit(mu)=00003Droot(mu) etc}, we have ${\rm exit}(\mu)=n+3$
and 
\[
{\rm root}(\mu_{0}^{*})+{\rm root}(\mu_{0})={\rm exit(\mu_{0})-2\le{\rm {\rm exit}(\mu)}-4}=n-1.
\]
Since ${\rm root}(\mu_{0}^{*})=n-2$, we see that ${\rm root}(\mu_{0})\le1$,
and thus $\mu_{0}$ is an extreme rigid honeycomb. Thus, there are
rigid tree honeycombs $\nu_{0},\nu_{1},\nu_{2}$ and integers $k_{0},k_{1},k_{2}>1$
such that $\mu_{j}=k_{j}\nu_{j}$ for $j=0,1,2$. Moreover, since
$\mu_{0}$ is also a simple degeneration of $\mu_{1}$, $k_{0}>k_{1}$;
similarly, $k_{0}>k_{2}$ and, in fact, $k_{0}$ is a multiple of
both $k_{1}$ and $k_{2}$. The exit pattern of $\mu_{1}$ is the
same as that of $\nu$, except that two consecutive exit multiplicities
$\alpha_{1},\alpha_{2}$ of $\nu$ are replaced by $\alpha_{1}+\alpha_{2}$.
We next exclude the possibility that the exit pattern of $\mu_{2}$
is the same as that of $\nu$, except that two consecutive exit multiplicities
$\alpha_{3},\alpha_{4}$, \emph{other than }$\alpha_{1},\alpha_{2}$,
are replaced by $\alpha_{3}+\alpha_{4}$. Suppose that this possibility
does arise. In this case, the exit multiplicities of the honeycombs
involved can be listed as follows.

\begin{tabular}{|c|c|c|c|}
\hline 
$\nu$ & $\alpha_{1},\alpha_{2}$ & $\alpha_{3},\alpha_{4}$ & $\alpha_{5},\dots$\tabularnewline
\hline 
\hline 
$\mu_{1}$ & $\alpha_{1}+\alpha_{2}=k_{1}\beta_{1}$ & $\alpha_{3}=k_{1}\beta_{3},\alpha_{4}=k_{1}\beta_{4}$ & $\alpha_{5}=k_{1}\beta_{5},\dots$\tabularnewline
\hline 
$\mu_{2}$ & $\alpha_{1}=k_{2}\gamma_{1},\alpha_{2}=k_{2}\gamma_{2}$ & $\alpha_{3}+\alpha_{4}=k_{2}\gamma_{3}$ & $\alpha_{5}=k_{2}\gamma_{5},\dots$\tabularnewline
\hline 
$\mu_{0}$ & $\alpha_{1}+\alpha_{2}=k_{0}\delta_{1}$ & $\alpha_{3}+\alpha_{4}=k_{0}\delta_{3}$ & $\alpha_{5}=k_{0}\delta_{5},\dots$\tabularnewline
\hline 
\end{tabular}

\noindent Here, $\beta_{j},\gamma_{j},\delta_{j}$ represent the nonzero
exit multiplicities of $\nu_{1},\nu_{2},\nu_{0}$, respectively. Note
that $\beta_{2},\gamma_{4},\delta_{2}$, and $\delta_{4}$ are not
listed because $\mu_{1},\mu_{2},$ and $\mu_{0}$ have fewer nonzero
exit multiplicities than $\nu$. We now apply (\ref{eq:omega square +2})
to the four rigid tree honeycombs involved to obtain
\begin{align*}
\alpha_{1}^{2}+\alpha_{2}^{2}+\alpha_{3}^{2}+\alpha_{4}^{2}+\sum_{j\ge5}\alpha_{j}^{2} & =\omega(\nu)^{2}+2,\\
(\alpha_{1}+\alpha_{2})^{2}+\alpha_{3}^{2}+\alpha{}_{4}^{2}+\sum_{j\ge5}\alpha_{j}^{2} & =k_{1}^{2}(\omega(\nu_{1})^{2}+2)=\omega(\nu)^{2}+2k_{1}^{2},\\
\alpha_{1}^{2}+\alpha_{2}^{2}+(\alpha_{3}+\alpha_{4})^{2}+\sum_{j\ge5}\alpha_{j}^{2} & =\omega(\nu)^{2}+2k_{2}^{2},\\
(\alpha_{1}+\alpha_{2})^{2}+(\alpha_{3}+\alpha_{4})^{2}+\sum_{j\ge5}\alpha_{j}^{2} & =\omega(\nu)^{2}+2k_{0}^{2}.
\end{align*}
Subtract now the first equality from the other three to see that
\[
\alpha_{1}\alpha_{2}=k_{1}^{2}-1,\quad\alpha_{3}\alpha_{4}=k_{2}^{2}-1,\quad\alpha_{1}\alpha_{2}+\alpha_{3}\alpha_{4}=k_{0}^{2}-1,
\]
and thus 
\[
k_{0}^{2}+1=k_{1}^{2}+k_{2}^{2}.
\]
Since $k_{0}$ is a common multiple of $k_{1}$ and $k_{2}$, it follows
that $k_{1}$ and $k_{2}$ are relatively prime, so $k_{0}=mk_{1}k_{2}$
for some $m\in\mathbb{N}$. Therefore
\[
k_{1}^{2}+k_{2}^{2}=m^{2}k_{1}^{2}k_{2}^{2}+1\ge k_{1}^{2}k_{2}^{2}+1,
\]
hence $(k_{1}^{2}-1)(k_{2}^{2}-1)\le0$, and this is impossible because
$k_{j}>1$.

We see right away that it is not possible to have three degenerations
of $\nu$ of this type because two of them would have to involve disjoint
sets of exit multiplicities. This proves (2).

Let now $\mu_{1},\mu_{2},\mu_{0}$, $\nu_{1},\nu_{2},\nu_{0}$, and
$k_{0},k_{1},k_{2}$ be as above. The preceding argument shows that
the exit multiplicities of these honeycombs can be arranged as follows.

\begin{tabular}{|c|c|c|}
\hline 
$\nu$ & $\alpha_{1},\alpha_{2},\alpha_{3}$ & $\alpha_{4},\dots$\tabularnewline
\hline 
\hline 
$\mu_{1}$ & $\alpha_{1}+\alpha_{2}=k_{1}\beta_{1},\alpha_{3}=k_{1}\beta_{3}$ & $\alpha_{4}=k_{1}\beta_{4},\dots$\tabularnewline
\hline 
$\mu_{2}$ & $\alpha_{1}=k_{2}\gamma_{1},\alpha_{2}+\alpha_{3}=k_{2}\gamma_{2}$ & $\alpha_{4}=k_{2}\gamma_{4},\dots$\tabularnewline
\hline 
$\mu_{0}$ & $\alpha_{1}+\alpha_{2}+\alpha_{3}=k_{0}\delta_{1}$  & $\alpha_{4}=k_{0}\delta_{4},\dots$\tabularnewline
\hline 
\end{tabular}

\noindent Of course, $\alpha_{1},\alpha_{2},$ and $\alpha_{3}$ must
be consecutive exit multiplicities corresponding to parallel rays.
As in the situation discussed above, the numbers $k_{1}$ and $k_{2}$
must be relatively prime. Indeed, a common factor $d$ of these integers
must divide $\omega(\nu)$, $\alpha_{1},$$\alpha_{3},$ and $\alpha_{j}$
for $j\ge4$, and therefore $d$ also divides $\alpha_{2}=3\omega(\nu)-\sum_{j\ne2}\alpha_{j}$.
Now, the relation $\sum_{j\ge1}\alpha_{j}^{2}=\omega(\nu)^{2}+2$
shows that $d^{2}$ divides $2$, so $d=1$. Thus $k_{0}=mk_{1}k_{2}$
for some $m\in\mathbb{N}$. We apply again (\ref{eq:omega square +2})
to obtain
\begin{align*}
\alpha_{1}^{2}+\alpha_{2}^{2}+\alpha_{3}^{2}+\sum_{j\ge4}\alpha_{j}^{2} & =\omega(\nu)^{2}+2,\\
(\alpha_{1}+\alpha_{2})^{2}+\alpha_{3}^{2}+\sum_{j\ge4}\alpha_{j}^{2} & =\omega(\nu)^{2}+2k_{1}^{2},\\
\alpha_{1}^{2}+(\alpha_{2}+\alpha_{3})^{2}+\sum_{j\ge4}\alpha_{j}^{2} & =\omega(\nu)^{2}+2k_{2}^{2},\\
(\alpha_{1}+\alpha_{2}+\alpha_{3})^{2}+\sum_{j\ge4}\alpha_{j}^{2} & =\omega(\nu)^{2}+2k_{0}^{2},
\end{align*}
and we subtract the first equality from the others to see that
\[
\alpha_{1}\alpha_{2}=k_{1}^{2}-1,\quad\alpha_{2}\alpha_{3}=k_{2}^{2}-1,\quad\alpha_{1}\alpha_{2}+\alpha_{2}\alpha_{3}+\alpha_{1}\alpha_{3}=k_{0}^{2}-1.
\]
 Thus
\[
\alpha_{1}\alpha_{3}=m^{2}k_{1}^{2}k_{2}^{2}-k_{1}^{2}-k_{2}^{2}+1\ge(k_{1}^{2}-1)(k_{2}^{2}-1),
\]
and since $\alpha_{1}=(k_{1}^{2}-1)/\alpha_{2},$ $\alpha_{3}=(k_{2}^{2}-1)/\alpha_{2}$,
these relations are only possible when $m=\alpha_{2}=1,$ $\alpha_{1}=k_{1}^{2}-1$,
and $\alpha_{3}=k_{2}^{2}-1$. Now set $\sigma=\delta_{1}$ and let
$\{p_{n}\}_{n=0}^{\infty}$ be defined by (\ref{eq:the Pell sequence defined}).
Since $\mu_{0}$ is a simple degeneration of $\mu_{1}$, it follows
from Proposition \ref{prop:possible degenerations}(3) that $\beta_{1},k_{2}=k_{0}/k_{1},$
and $\beta_{3}$ are consecutive terms of this sequence. Finally,
$\beta_{1}=(\alpha_{1}+\alpha_{2})/k_{1}=k_{1}$, and therefore $k_{1}$
and $k_{2}$ are consecutive terms in the sequence $\{p_{n}\}_{n=0}^{\infty}$.
This concludes the proof of (3).
\end{proof}

\section{\label{sec:de-degeneration}Regeneration}

In Section \ref{sec:Degeneration-of-a rigid honey}, we found necessary
conditions for a honeycomb to be one of the simple degenerations of
a rigid tree honeycomb. Our purpose in this section is to show that
these necessary conditions are sufficient as well. The argument requires
a basic construction of rigid tree honeycombs that involves a combination
of an inflation with a partial deflation. 

\subsection*{Construction}

We start with the following data:
\begin{enumerate}
\item a rigid honeycomb $\nu$.
\item a\emph{ }rigid honeycomb $\mu$ that is compatible with the puzzle
of $\nu$. We set $\nu'=\mu_{\nu}$, where $\mu_{\nu}$ is defined
in Section \ref{sec:Honeycombs-compatible-with-a}. It follows from
\cite[Theorem 7.14]{derksen} and \cite[Theorem 1.4]{king-tollu}
(see also the discussion following \cite[Theorem 7.2]{blt}) that
$\nu'$ is also rigid.
\item a ray $f$ in the support of $\mu$.
\end{enumerate}
We proceed in three steps.
\begin{enumerate}
\item [(i)]Construct the puzzle of $\mu$ and preserve the coloring of
the `white' pieces. Color the remainder of the pieces as follows:
\begin{enumerate}
\item The pieces corresponding to branch points of $\mu$ are colored light
red.
\item The inflation of (part of) an edge is colored dark red if it is contained
in the closure of a white piece of the puzzle of $\nu$ or if it crosses
a dark gray parallelogram and is parallel to its white sides. Two
of the sides of a dark red parallelograms are white and the other
two are light red.
\item The inflation of (part of) an edge is colored black if it is contained
in the closure of a light gray piece of the puzzle of $\nu$ or if
it crosses a dark gray parallelogram and is parallel to its light
gray sides.
\end{enumerate}
\item [(ii)]Construct the measure $\widetilde{\mu}$ with support in the
edges of the puzzle of $\mu$ and rooted in the incoming ray corresponding
to $f$ (Theorem \ref{thm:stretched honeycomb}).
\item [(iii)]Remove all the black, dark gray, and light gray areas of the
puzzle constructed in (i), translate the remaining pieces to tile
the plane, and preserve the multiplicities assigned by $\widetilde{\mu}$
for those (parts of) edges that are contained in the closure of a
white piece or in a dark gray parallelogram and are parallel to its
white sides. Add multiplicities in case several such edges are translated
to the same segment. Call $\mu'$ the collection of multiplicities
obtained this way.
\end{enumerate}
\begin{figure}

\includegraphics[scale=0.4]{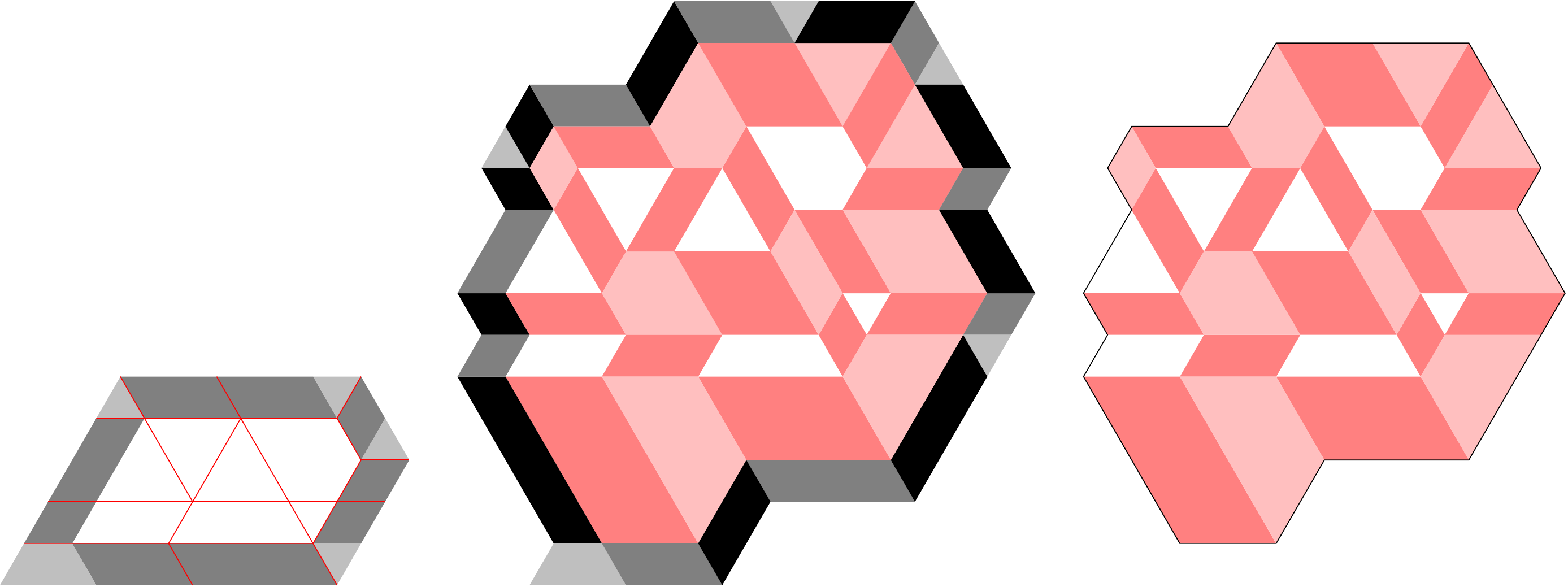}\caption{\label{fig:white puzzle piece}The basic construction inside a white
puzzle piece}

\end{figure}

\begin{figure}

\includegraphics[scale=0.4]{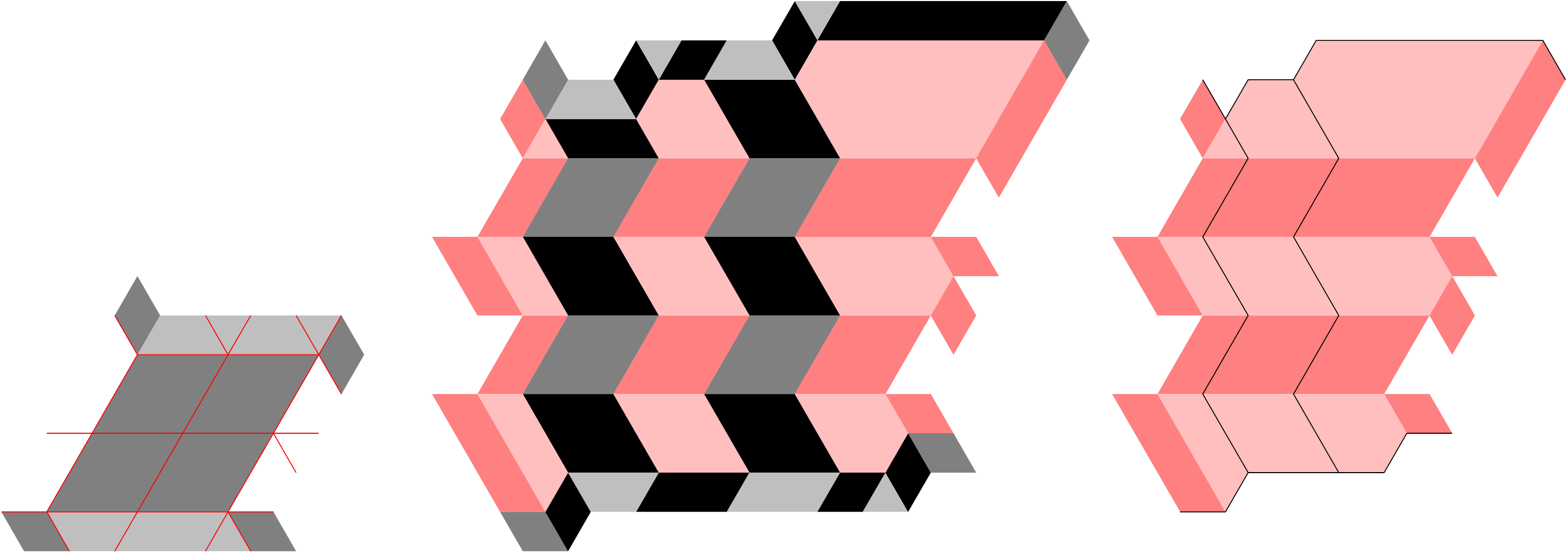}\caption{The basic construction inside a dark gray parallelogram\label{fig:dark parallelogram}}
\end{figure}

\begin{figure}
\includegraphics[scale=0.4]{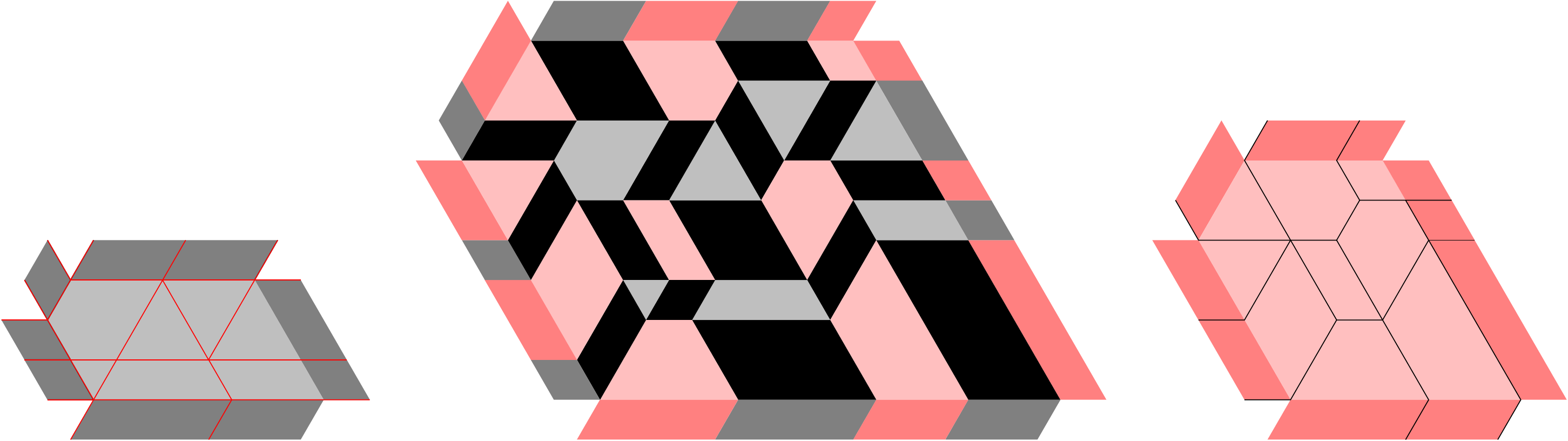}

\caption{The basic construction inside a light gray puzzle piece\label{fig:ligh gray puzzle piece}}

\end{figure}

\begin{prop}
\label{prop:inflation-deflation construction}The construction above
yields:
\begin{enumerate}
\item The puzzle of $\nu'$, with the color gray replaced by red, and
\item A rigid honeycomb $\mu'$ that is compatible with the puzzle of $\nu'$.
\end{enumerate}
\end{prop}

\begin{proof}
The white areas in the puzzle of $\nu$ are divided by the support
of $\mu$ into smaller regions, and it is these regions that are the
white puzzle pieces that remain after the operation described in (iii)
above. Suppose that two such areas $G_{1}$ and $G_{2}$ are separated
by an edge $e$ of $\mu$. Then their counterparts after operation
(iii) are separated by a dark red parallelogram with two edges parallel
to $e$ and two edges of length $\mu(e)$ that are $60^{\circ}$ clockwise
from $e$. Of course, $e$ is the translate of some edge $e'$ of
$\nu'$ with multiplicity $\nu(e')=\mu(e)$, and thus the dark red
parallelogram is the one that normally appears in the construction
of the puzzle of $\nu'$. Another possibility is that $G_{1}$ and
$G_{2}$ are separated by a dark gray parallelogram in the puzzle
of $\nu$. Suppose that $AB$ and $A'B'$ are the two white sides
of that parallelogram, and that $I_{1},\dots,I_{m}$ are the edges
of $\mu$ that cross the parallelogram and are parallel to its white
sides. Then, in the picture produced by (iii), $G_{1}$ and $G_{2}$
are separated by a dark red parallelogram that is decomposed into
$m+2$ parallelograms with light red sides of lengths $\mu(AB),\mu(I_{1}),\dots,\mu(I_{m}),\mu(A'B')$.
The sum of these lengths is precisely $\nu'(e')$, where $e'$ is
the edge of $\nu$ with inflation $ABA'B'$. Similarly, the light
red pieces assemble in step (iii) to complete the puzzle of $\nu'$.
This process is illustrated in Figure \ref{fig:white puzzle piece}
for white pieces in the puzzle of $\nu$, in Figure \ref{fig:ligh gray puzzle piece}
for light gray puzzle pieces, and in Figure \ref{fig:dark parallelogram}
for dark gray puzzle pieces. These figures are sufficiently general
(they deal with branch points of order five contained in the interior,
on the boundary, or in a corner of a puzzle piece) to demonstrate
that the white, dark red, and light red pieces do actually fit together
to form the puzzle of $\nu'$. The multiplicities chosen for these
figures are the smallest positive integers that satisfy the the balance
condition on the (red) edges of $\mu$. Choosing different multiplicities
(possibly zero) produce essentially the same picture (possibly without
some of the dark red parallelograms).

In order to verify that $\mu'$ also satisfies the balance condition
we observe that $\mu'$ is obtained from $\widetilde{\mu}$ by decreasing
the lengths of its edges without changing their multiplicities, except
in those cases in which two edges are translated to the same segment
and their multiplicities are added. In other words, $\mu'$ is a degeneration
of $\widetilde{\mu}$ and is therefore a rigid honeycomb.
\end{proof}
\begin{thm}
\label{thm:overlay with coefficients>1}Suppose that $\nu_{1}$ and
$\nu_{2}$ are rigid tree honeycombs such that $\Sigma_{\nu_{1}}(\nu_{2})=\sigma\ge1$
and $\Sigma_{\nu_{2}}(\nu_{1})=0$. Let $\{p_{n}\}_{n=0}^{\infty}$
be the sequence defined by \emph{(\ref{eq:the Pell sequence defined})}.
Then, for every $n\ge1$, there exists a rigid tree honeycomb $\mu_{n}$
such that:
\begin{enumerate}
\item $\mu_{n}$ and $p_{n}\nu_{1}+p_{n+1}\nu_{2}$ have the same exit pattern,
and
\item $\mu_{n}$ is compatible with the puzzle of $p_{n-1}\nu_{1}+p_{n}\nu_{2}$. 
\end{enumerate}
\end{thm}

\begin{proof}
We proceed by induction. The existence of $\mu_{1}=\widehat{\nu}_{1}$
is a consequence of Theorem \ref{thm:basic inductive construction}.
Suppose that $n\in\mathbb{N}$ and that $\mu_{n}$ has been constructed.
Set $\nu=p_{n-1}\nu_{1}+p_{n}\nu_{2}$ and $\nu'=(\mu_{n})_{\nu}=p_{n}\nu_{1}+p_{n+1}\nu_{2}$.
The nonzero exit multiplicity of $\mu_{n}$ corresponding to a typical
ray $e$ is of the form $p_{n}\alpha+p_{n+1}\beta$, where $\alpha$
and $\beta$ are exit multiplicities of $\nu_{1}$ and $\nu_{2}$,
respectively, and $\alpha+\beta>0$. Choose a particular ray $e_{0}$
of $\mu_{n}$ with exit multiplicity $p_{n}\alpha_{0}+p_{n+1}\beta_{0}$
such that $\beta_{0}>0$. We apply the construction above to the measures
$\nu$, $\mu=\mu_{n},$ and the incoming ray $f$ in the inflation
of $e_{0}$. We obtain a honeycomb $\mu'$ compatible with the puzzle
of $\nu'$. The honeycomb $\widetilde{\mu}$ constructed in step (ii)
has exit multiplicities $(p_{n}\alpha_{0}+p_{n+1}\beta_{0})^{2}-1$
and $1$ corresponding to the outgoing and incoming rays of the puzzle
of $\mu$ corresponding to $e_{0}$. The other nonzero exit multiplicities
correspond to the remaining outgoing rays and they are of the form
$(p_{n}\alpha_{0}+p_{n+1}\beta_{0})(p_{n}\alpha+p_{n+1}\beta)$. It
follows, in particular, that
\[
\mu'_{\nu'}=(p_{n}\alpha_{0}+p_{n+1}\beta_{0})(p_{n}\nu_{1}+p_{n+1}\nu_{2}).
\]
Next, observe that $\mu'$ assigns unit multiplicity to the incoming
ray in the inflation of $e_{0}$ in the puzzle of $\nu'$. Denote
by $\rho$ the rigid tree honeycomb supported by the puzzle edges
of $\nu'$ and rooted in the incoming ray corresponding to $e_{0}$.
Lemma \ref{lem:compatible puzzle does not leave the edges of the puzzle once there}
implies that there exists a rigid honeycomb $\mu''$ such that $\mu'=\rho+\mu''$.
Suppose that $f\ne f_{0}$ is an arbitrary ray of $\nu'$ such that
$\nu_{1}(f)=\alpha$ and $\nu_{2}(f)=\beta.$ Then, according to Example
\ref{exa:rigid overlay of two}, the multiplicity that $\rho$ assigns
to the exit ray corresponding to $e\ne e_{0}$ in the puzzle of $\nu'$
is $\alpha_{0}\alpha+(\beta_{0}+\sigma\alpha_{0})\beta$. Therefore
$\rho\ne\mu'$, so $\mu''\ne0$ and thus ${\rm root}(\mu')\ge2$.
Now, $\nu'$ is a degeneration of $\mu'$ and ${\rm exit}(\mu')={\rm exit}(\nu')$,
so (\ref{eq:roots and exits of degeneration}) implies 
\[
2-{\rm root}(\mu')={\rm root}(\nu')-{\rm root}(\mu')={\rm root}(\mu'^{*})-{\rm root}(\nu'^{*})\ge0.
\]
 We conclude that ${\rm root}(\mu')=2$, and thus $\mu''=\gamma\mu'''$
for some rigid tree measure $\mu'''$ and some $\gamma>0$. We conclude
the proof by showing that $\mu'''$ has the same exit pattern as $p_{n+1}\nu_{1}+p_{n+2}\nu_{2}$
and thus $\mu_{n+1}=\mu'''$ satisfies the conclusion of the theorem
with $n$ replaced by $n+1$. We start with a typical ray $e\ne e_{0}$
in the support of $\nu'$ and calculate the density that $\mu''$
assigns to the outgoing ray in the inflation of $\nu'$:
\[
(p_{n}\alpha_{0}+p_{n+1}\beta_{0})(p_{n}\alpha+p_{n+1}\beta)-\alpha_{0}\alpha-(\beta_{0}+\sigma\alpha_{0})\beta=c\alpha+d\beta.
\]
Here 
\begin{align*}
c & =(p_{n}\alpha_{0}+p_{n+1}\beta_{0})p_{n}-\alpha_{0}\\
 & =(p_{n}^{2}-1)\alpha_{0}+p_{n}p_{n+1}\beta_{0}\\
 & =p_{n-1}p_{n+1}\alpha_{0}+p_{n}p_{n+1}\beta_{0}=(p_{n-1}\alpha_{0}+p_{n}\beta_{0})p_{n+1},
\end{align*}
where we used (\ref{eq:p square -1}). Similarly, using (\ref{eq:p square -1})
and (\ref{eq:p_np_n+2-sigma}) we obtain
\begin{align*}
d & =(p_{n}\alpha_{0}+p_{n+1}\beta_{0})p_{n+1}-\beta_{0}-\sigma\alpha_{0}\\
 & =(p_{n}p_{n+1}-\sigma)\alpha_{0}+(p_{n+1}^{2}-1)\beta_{0}\\
 & =p_{n-1}p_{n+2}\alpha_{0}+p_{n}p_{n+2}\beta_{0}=(p_{n-1}\alpha_{0}+p_{n}\beta_{0})p_{n+2}.
\end{align*}
Similar calculations apply to the outgoing ray in the inflation of
$e_{0}$, while the outgoing ray in that inflation is assigned multiplicity
$1-1=0$ by $\mu''$. Thus, the exit pattern of $\mu''$ is the same
as that of $(p_{n-1}\alpha_{0}+p_{n}\beta_{0})(p_{n+1}\nu_{1}+p_{n+2}\nu_{2})$,
and the desired conclusion follows along with the equality $\gamma=p_{n-1}\alpha_{0}+p_{n}\beta_{0}$.
\end{proof}
\begin{rem}
The proof above would be slightly simpler if one could choose $\alpha_{0}=0$
and $\beta_{0}>0$. This however is not always possible because there
are rigid overlays $\nu_{1}+\nu_{2}$ such that $\nu_{1}$ and $\nu_{2}$
assign positive multiplicity to precisely the same rays, as seen in
the two examples in Figure \ref{fig:same-exits} (the edges of $\nu_{2}$
outside the support of $\nu_{1}$ are drawn with dotted lines).
\begin{figure}
\includegraphics[scale=0.3]{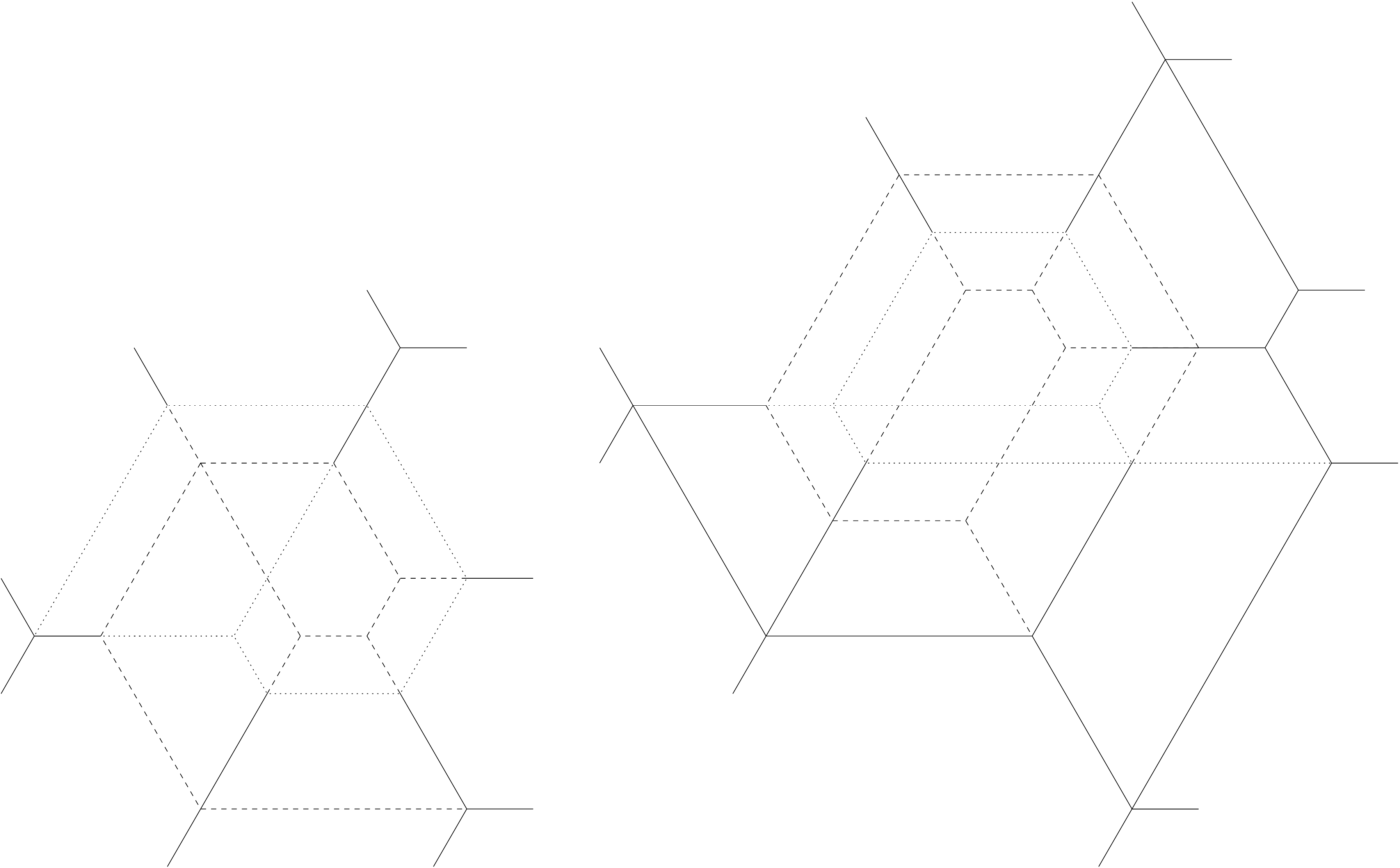}

\caption{\label{fig:same-exits}Rigid overlays where the nonzero exit multiplicities
of $\nu_{1}$ and $\nu_{2}$ correspond to the same rays}

\end{figure}
\end{rem}

\begin{example}
We illustrate the entire process in the proof of Theorem \ref{thm:overlay with coefficients>1}
for the overlay pictured in Figure \ref{fig:overlay example}. In
this example, $\Sigma_{\nu_{1}}(\nu_{2})=1$, and the red lines in
the second part of the picture represent the support of the honeycomb
$\mu_{1}=\widehat{\nu}_{1}$, compatible with the puzzle of $\nu_{2}$.
\begin{figure}
\includegraphics[scale=0.3]{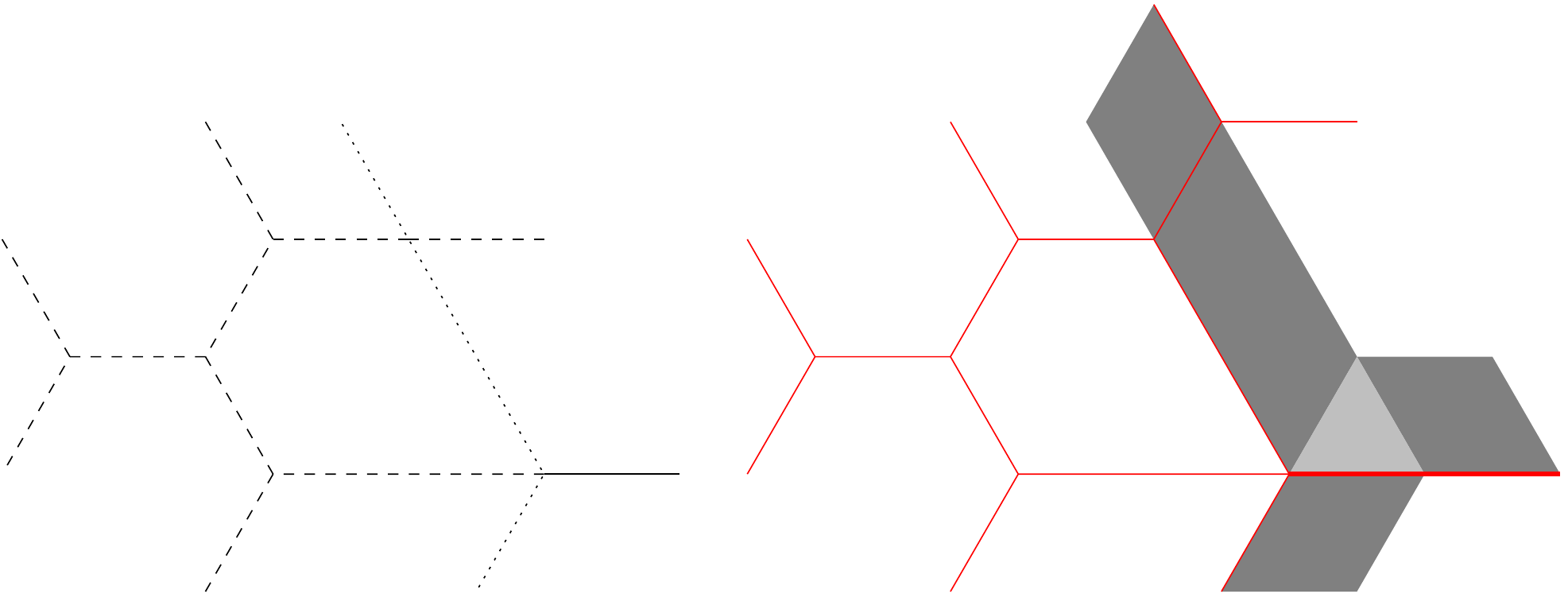}

\caption{\label{fig:overlay example}A simple overlay with $\Sigma_{\nu_{1}}(\nu_{2})=1$}

\end{figure}
 The inflation of $\mu_{1}$, colored as in the proof above is shown
in Figure \ref{fig:existence proof sigma=00003D1}. The support of
$\mu'$ is colored blue and the black dot indicates the incoming ray
$f$. The second part of the picture represents the support of $\mu_{2}$
(in black), and the parts of the support of $\rho$ not covered by
the support of $\mu_{2}$ (in blue). Since $\sigma=1$, the honeycomb
$\mu_{2}$ has the exit pattern of $p_{2}\nu_{1}+p_{3}\nu_{2}=\nu_{1}$,
so $\mu_{2}$ is homologous to $\nu_{1}$. 
\begin{figure}
\includegraphics[scale=0.3]{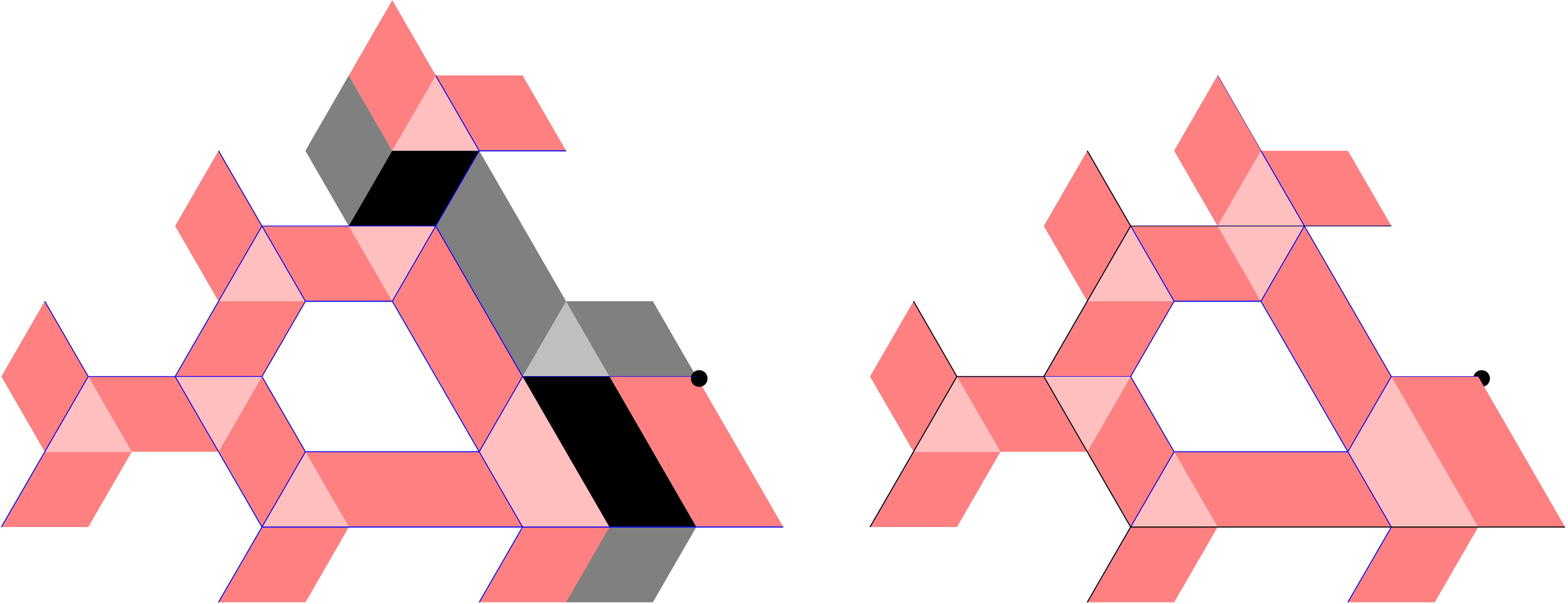}

\caption{\label{fig:existence proof sigma=00003D1}The inflation of $\mu_{1}$,
the support of $\mu''=\widetilde{\mu}_{1}$. Also, the support of
$\mu''$ (black) and the edges of $\rho$ not contained in the support
of $\mu'=\mu_{2}$}

\end{figure}
\end{example}

\begin{cor}
\label{cor:overlay with coef>1 via symmetry}With the notation of
the Theorem\emph{ \ref{thm:overlay with coefficients>1}}, for every
$n\in\mathbb{N}$ there exists a rigid tree honeycomb $\mu_{n}'$
with the same exit pattern as $p_{n+1}\nu_{1}+p_{n}\nu_{2}$.
\end{cor}

\begin{proof}
Reflection in a line parallel to $w_{1}$ changes the roles of $\nu_{1}$
and $\nu_{2}$. Apply Theorem\emph{ }\ref{thm:overlay with coefficients>1}
to these reflected honeycombs to get a honeycomb $\mu{}_{n}$. Finally,
reflect $\mu{}_{n}$ in a line parallel to $w_{1}$ to obtain $\mu'_{n}$.
\end{proof}
\begin{example}
The preceding corollary, applied to the first two overlays in Figure
\ref{fig:overlays with sigma=00003D2} shows that 
\begin{align*}
n,n,n+1|n+1,n,n|n+1,n,n\\
n+1,n+1,n|n,n+1,n+1|n,n+1,n+1\\
n,n+1,n|n,n,n+1|n,n+1,n\\
n+1,n,n+1|n+1,n+1,n|n+1,n,n+1
\end{align*}
are the exit patterns of rigid tree honeycombs. These four families
were already described in \cite[Section 8]{blt}. Similarly, the third
overlay yields the patterns 
\begin{align*}
n,n+1,n+1,n+1|n+1,n+1,n+1,n|2n+2,2n+1\\
n+1,n,n,n|n,n,n,n+1|2n,2n+1.
\end{align*}
 
\begin{figure}
\includegraphics[scale=0.35]{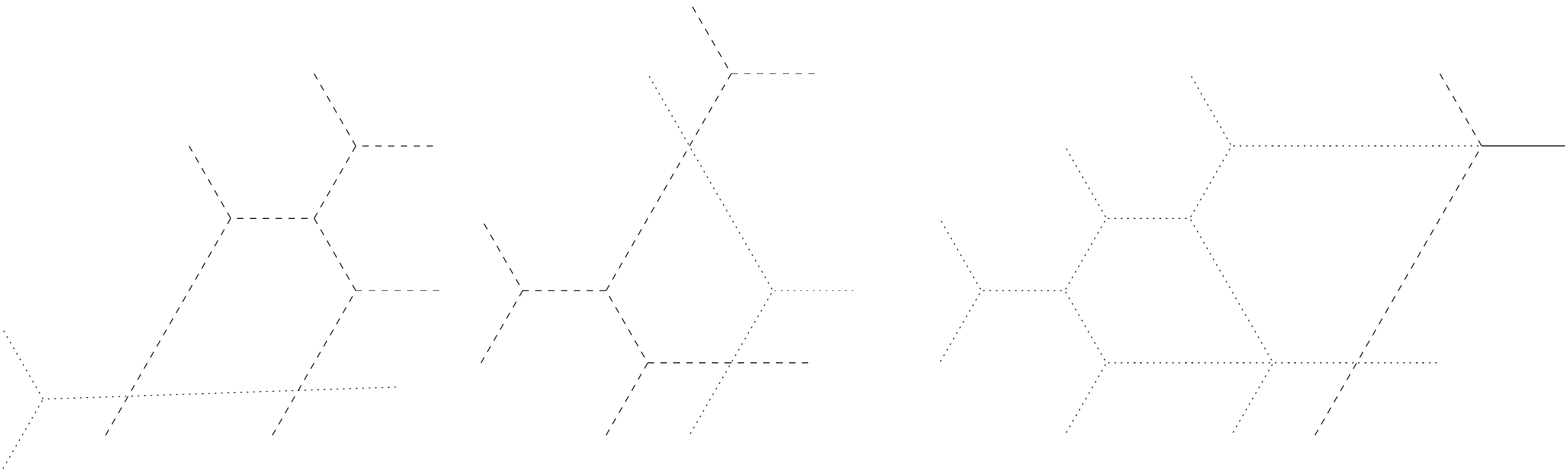}

\caption{\label{fig:overlays with sigma=00003D2}Three rigid overlays with
$\Sigma_{\nu_{1}}(\nu_{2})=2$}

\end{figure}
The overlays in Figure \ref{fig:overlays sigma=00003D3} satisfy $\Sigma_{\nu_{1}}(\nu_{2})=3$.
\begin{figure}
\includegraphics[scale=0.35]{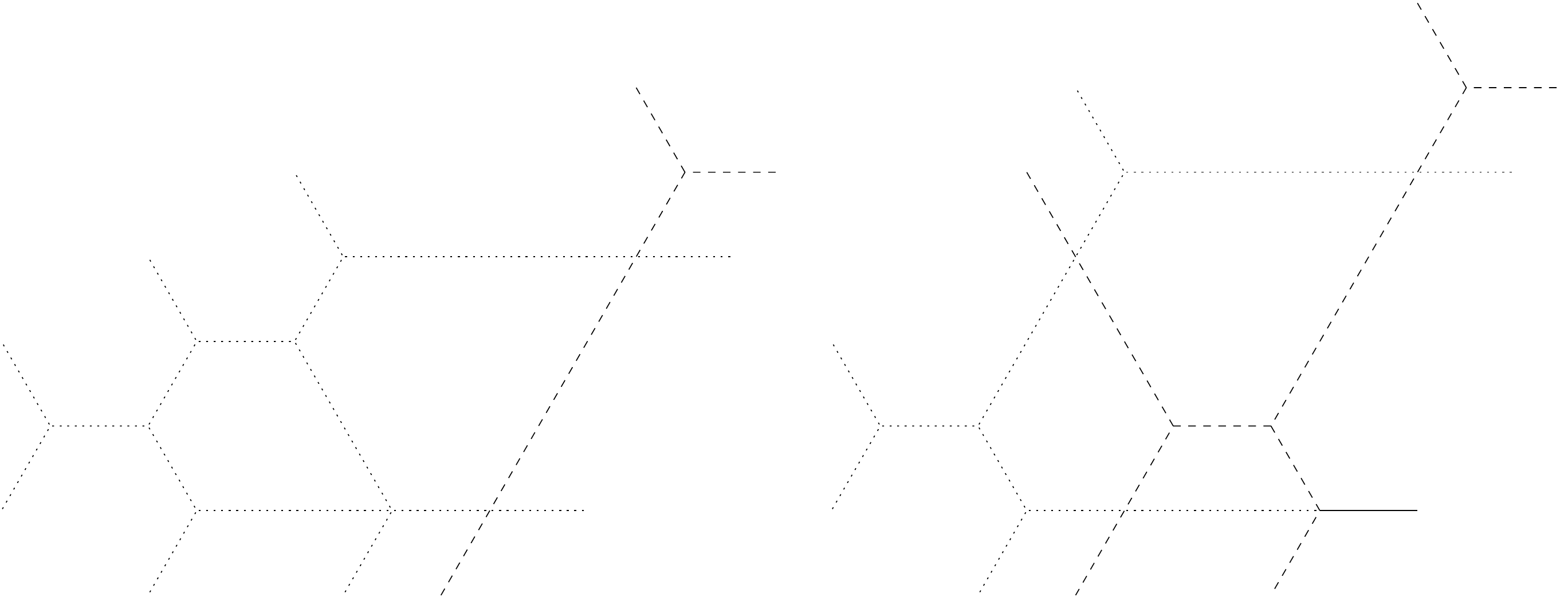}

\caption{\label{fig:overlays sigma=00003D3}Two rigid overlays with $\Sigma_{\nu_{1}}(\nu_{2})=3$}

\end{figure}
 We deduce that 
\begin{align*}
p_{n},p_{n+1},p_{n+1},p_{n+1}|p_{n+1},p_{n+1},p_{n+1},p_{n}|2p_{n+1},p_{n+1},p_{n}\\
p_{n+1},p_{n},p_{n},p_{n}|p_{n},p_{n},p_{n},p_{n+1}|2p_{n},p_{n},p_{n+1}\\
p_{n},p_{n+1},p_{n},p_{n+1}|p_{n+1},p_{n+1},p_{n},p_{n}|p_{n}+p_{n+1},p_{n+1},p_{n} & \text{}\\
p_{n+1},p_{n},p_{n+1},p_{n}|p_{n},p_{n},p_{n+1},p_{n+1}|p_{n}+p_{n+1},p_{n},p_{n+1}
\end{align*}
 are the exit patterns of rigid tree honeycombs, where the sequence
$\{p_{n}\}_{n=0}^{\infty}=\{0,3,8,21,55,144,\dots\}$ is defined by
(\ref{eq:the Pell sequence defined}) with $\sigma=3$.
\end{example}

\begin{rem}
\label{rem:bad conjecture} The preceding two results (combined with
Theorem \ref{thm:basic inductive construction} for $\sigma=1$) could
be paraphrased as follows. Suppose that $\nu$ is a rigid honeycomb
with ${\rm root}(\nu)=2$. If $\Sigma_{\nu}(\nu)=-1$, then there
exists a rigid tree honeycomb $\mu$ that has the same exit pattern
as $\nu$. One may wonder whether one can remove the restriction on
${\rm root}(\nu).$ The answer is negative. If $\nu_{1},\nu_{2}$,
and $\nu_{3}$ are three rigid honeycombs of weight $1$ forming a
clockwise overlay (see Figure \ref{fig:overlay of 3}),
\begin{figure}
\begin{picture}(100,100)
\put(0,0){\includegraphics[scale=.4]{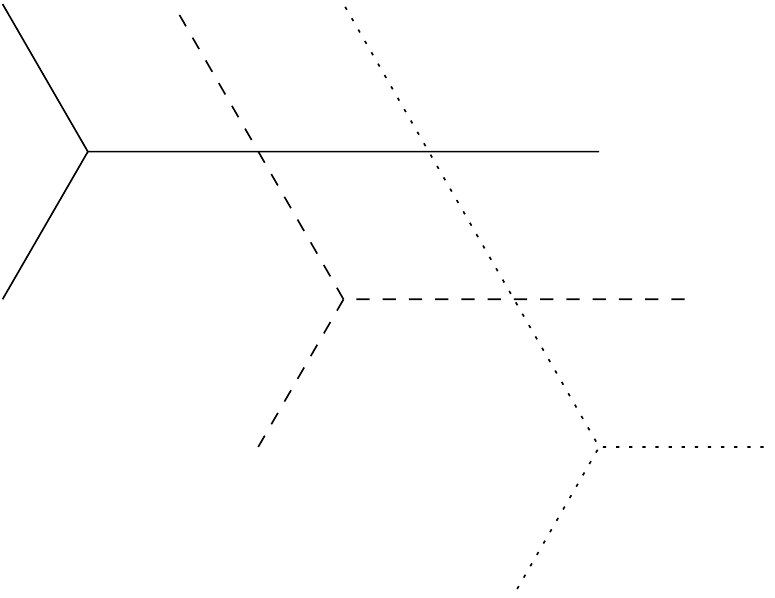}}
{
\put(-3,49){$\nu_1$}

\put(27,31){$\nu_2$}
\put(56,14){$\nu_3$}

}
\end{picture}

\caption{\label{fig:overlay of 3}$\Sigma_{\nu_{j}}(\nu_{i})=0$ for $i<j$
and $\Sigma_{\nu_{j}}(\nu_{i})=1$ for $i>j$}

\end{figure}
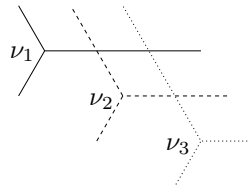
 then $\nu=c_{1}\nu_{1}+c_{2}\nu_{2}+c_{3}\nu_{3}$, $c_{1},c_{2},c_{3}\in\mathbb{N},$
satisfies $\Sigma_{\nu}(\nu)=-1$ precisely when.
\[
c_{1}^{2}+c_{2}^{2}+c_{3}^{2}-c_{1}c_{2}-c_{2}c_{3}-c_{1}c_{3}=1.
\]
Rewriting this equation as
\[
(c_{1}-c_{2})^{2}+(c_{2}-c_{3})^{2}+(c_{1}-c_{3})^{2}=2,
\]
we see that, for each positive solution $(c_{1},c_{2},c_{3})$, two
of the $c_{j}$ must be equal and differ from the third by $1$. Among
the resulting solutions, it seen that $n\nu_{1}+(n+1)\nu_{2}+n\nu_{3}$
and $(n+1)\nu_{1}+n\nu_{2}+(n+1)\nu_{3}$ are not the exit patterns
of rigid tree honeycombs for any $n\in\mathbb{N}$. The other solutions
are the exit patterns of rigid tree honeycombs, as seen from Example
\ref{exa:examples of n times riigid} below.
\end{rem}

\begin{thm}
\label{thm:higher multiples of a rigid tree}Suppose that $\mu$ is
a rigid tree honeycomb and that its nonzero exit multiplicities, listed
in counterclockwise order as $\alpha_{1},\dots,\alpha_{k}$ are such
that $\alpha_{1}=\sigma>1$. Let $\{p_{n}\}_{n=0}^{\infty}$ be the
sequence defined by \emph{(\ref{eq:the Pell sequence defined})}.
Then, for every $n\ge2$, there exists a rigid tree honeycomb $\mu_{n}$
such that:
\begin{enumerate}
\item The exit multiplicities of $\mu_{n}$, listed in counterclockwise
order, are 
\[
p_{n+1},p_{n-1},p_{n}\alpha_{2},\dots,p_{n}\alpha_{k}.
\]
 In other words, $\mu_{n}$ has the same exit pattern as $p_{n}\mu$,
except that $p_{n}\alpha_{1}$ is replaced by $p_{n+1}$ and $p_{n-1}$.
\item $\mu_{n}$ is compatible with the puzzle of $p_{n-1}\mu$.
\end{enumerate}
\end{thm}

\begin{proof}
Let $e_{0}$ be the ray in the support of $\mu$ corresponding to
the multiplicity $\alpha_{0}$. We prove the existence of $\mu_{n}$
satisfying the following additional property: the rays in the support
of $\mu_{n}$ are:
\begin{enumerate}
\item [(i)]the outgoing rays in the puzzle of $p_{n-1}\mu$, and
\item [(ii)]one other ray contained in inflation of $e_{0}$.
\end{enumerate}
The existence of $\mu_{2}$ follows from Theorem \ref{thm:honey in the boundary of puzzles}.
In this case, the additional ray in (ii) is the incoming ray in the
inflation of $e_{0}$. Suppose that $\mu_{n}$ has been constructed
for some $n\ge2$. We apply the basic construction with $p_{n-1}\mu$
in place of $\nu$ and $\mu_{n}$ in place of $\mu$ and with the
outgoing ray in the inflation of $e_{0}$ in place of $f$. The important
observation is that (using notation from part (iii) of the construction,
is that $\widetilde{\mu}$ has consecutive rays with multiplicities
$p_{n+1}^{2}-1,$$1$, and $p_{n+1}p_{n-1}$, and that the rays with
densities $1$ and $p_{n-1}p_{n+1}$ are only separated by a dark
gray strip (a translated part of the inflation of $e_{0}$). It follows
that these three exit multiplicities collapse to two exit multiplicities
$p_{n+1}^{2}-1$ and $p_{n-1}p_{n+1}+1$ of $\mu'$. Since 
\begin{align*}
p_{n+1}^{2}-1 & =p_{n}p_{n+2},\\
p_{n-1}p_{n+1}+1 & =p_{n}p_{n},
\end{align*}
it suffices to show that $\mu'$ is extreme, in which case $\mu_{n+1}=\mu'/p_{n}$
satisfies the conclusion of the theorem with $n+1$ in place of $n$.
To see this, we note that $\nu'=p_{n}\mu$ and $\mu'_{\nu'}=p_{n}p_{n+1}\mu$,
so ${\rm root}(\mu'_{\nu'})=1$ and ${\rm exit}(\mu'_{\nu'})=k$.
An application of (\ref{eq:roots and exits of degeneration}) shows
that
\begin{align*}
1-{\rm root}(\mu') & ={\rm root}(\mu'_{\nu'})-{\rm root}(\mu')\\
 & =[{\rm root}(\mu'^{*})-{\rm root}(\mu_{\nu'}^{\prime*})]-[{\rm exit}(\mu')-{\rm exit}(\mu'_{\nu'})]\\
 & =[{\rm root}(\mu'^{*})-{\rm root}(\mu_{\nu'}^{\prime*})]-1.
\end{align*}
Since ${\rm root}(\mu')\ge1$ and ${\rm root}(\mu'^{*})-{\rm root}(\mu_{\nu'}^{\prime*})>0$,
we conclude that ${\rm root}(\mu')=1$, as desired.
\end{proof}
A reflection argument, like the one used in the proof of Corollary
\ref{thm:overlay with coefficients>1}, immediately yields the following
result.
\begin{cor}
\label{cor:multiple of tree reversed}Under the hypothesis of Theorem
\emph{\ref{thm:higher multiples of a rigid tree}}, for every $n\ge2$
there exists a rigid tree honeycomb $\mu_{n}'$ such that the exit
multiplicities of $\mu_{n}'$, listed in counterclockwise order, are
\[
p_{n-1},p_{n+1},p_{n}\alpha_{2},\dots,p_{n}\alpha_{k}.
\]
\end{cor}

\begin{example}
\label{exa:examples of n times riigid}An application of Theorem \ref{thm:higher multiples of a rigid tree}
and Corollary \ref{cor:multiple of tree reversed} to the rigid tree
honeycomb with exit pattern $1,1,1|2,1|1,1,1$ shows that 
\begin{align*}
n,n,n|n+1,n-1,n|n,n,n\\
n,n,n|n-1,n+1,n|n,n,n
\end{align*}
 are the exit patterns of rigid tree honeycombs. These examples were
first described in \cite[Section 8]{blt}. Another rigid tree honeycomb
has exit pattern $2,2,1|1,3,1|1,1,2,1$. If we apply the results above
with the first $2$ in the role of $\alpha_{1}$, we deduce that 
\begin{align*}
n-1,n+1,2n,n|n,3n,n|n,n,2n,n\\
n+1,n-1,2n,n|n,3n,n|n,n,2n,n
\end{align*}
 are the exit patterns of rigid tree measures. The tree honeycombs
in the first series are represented in Figure \ref{fig:new-series}
\begin{figure}
\includegraphics[scale=0.2]{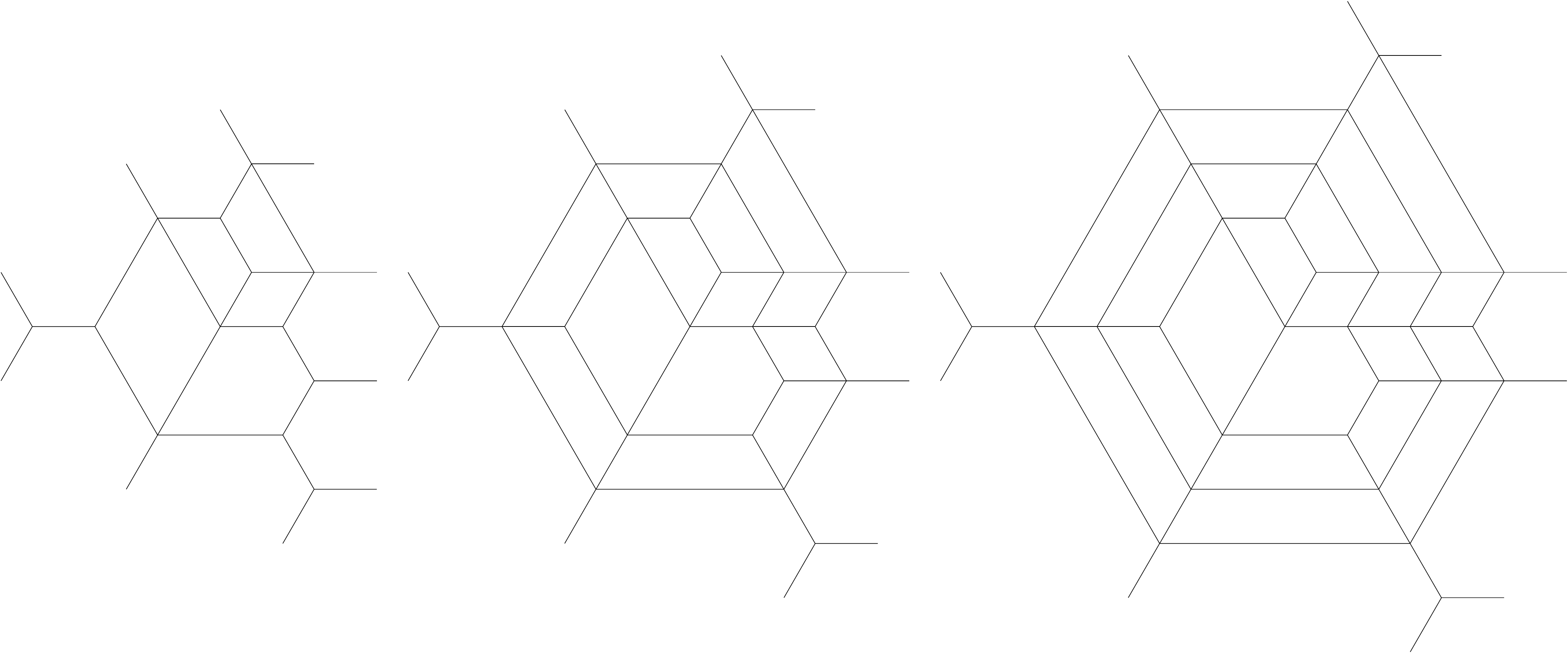}

\caption{\label{fig:new-series}Rigid tree honeycombs with exit pattern $n-1,n+1,2n,n|n,3n,n|n,n,2n,n$,
$n=1,2,3$}

\end{figure}
 for $n=1,2,3$. It is easy to see how further honeycombs in this
series are constructed. The results can also be applied with $3$
in place of $\alpha_{1}$ to generate the exit patterns 
\begin{align*}
2p_{n},2p_{n},p_{n}|p_{n},p_{n+1},p_{n-1},p_{n}|p_{n},p_{n},2p_{n},p_{n},\\
2p_{n},2p_{n},p_{n}|p_{n},p_{n-1},p_{n+1},p_{n}|p_{n},p_{n},2p_{n},p_{n},
\end{align*}
where the numbers $p_{n}$ come from the sequence $0,3,8,21,\dots$.
\end{example}

The following result is extracted from the proof of Theorem \ref{thm:higher multiples of a rigid tree}.
The second part follows from the first using reflection.
\begin{cor}
\label{cor:double multiple}Under the hypothesis of Theorem \emph{\ref{thm:higher multiples of a rigid tree}},
for every $n\ge2$ there exists a rigid tree honeycomb $\tau_{n}$
such that the exit multiplicities of $\tau_{n}$, listed in counterclockwise
order, are 
\[
p_{n+1}^{2}-1,1,p_{n}^{2}-1,p_{n}p_{n+1}\alpha_{2},\dots,p_{n}p_{n+1}\alpha_{k}.
\]
 Similarly, there exists a rigid tree measure $\tau_{n}'$ such that
the exit multiplicities of $\tau_{n}'$, listed in counterclockwise
order, are 
\[
p_{n}^{2}-1,1,p_{n+1}^{2}-1,p_{n}\alpha_{2},\dots,p_{n}\alpha_{k}.
\]
\end{cor}

The following result shows that all the possibilities described in
Theorem \ref{thm:at most two type (B) degenerations} actually arise.
\begin{prop}
\label{prop:two extreme simple degenerations}Each of the honeycombs
$\tau_{n}$ and $\tau'_{n}$ of Corollary\emph{ \ref{cor:double multiple}}
has two simple degenerations that are extreme honeycombs.
\end{prop}

\begin{proof}
Fix $n\ge2$ and consider the honeycomb $\tau_{n}$. We know from
the above results that there exist rigid tree honeycombs $\mu_{n}$
and $\mu_{n+1}$ with exit multiplicities
\[
p_{n+1},p_{n-1},p_{n}\alpha_{2},\dots,p_{n}\alpha_{k}
\]
and
\[
p_{n+2},p_{n},p_{n+1}\alpha_{2},\dots,p_{n+1}\alpha_{k},
\]
respectively. Thus the extreme rigid honeycombs $p_{n+1}\mu_{n}$
and $p_{n}\mu_{n+1}$ have exit multiplicities
\[
p_{n+1}^{2},p_{n+1}p_{n-1},p_{n+1}p_{n}\alpha_{2},\dots,p_{n+1}p_{n}\alpha_{k}
\]
and 
\[
p_{n}p_{n+2},p_{n}^{2},p_{n+1}p_{n}\alpha_{2},\dots,p_{n+1}p_{n}\alpha_{k},
\]
respectively. Since $p_{n+1}^{2}=p_{n+1}^{2}-1+1$, $p_{n+1}p_{n-1}=p_{n}^{2}-1$,
$p_{n}p_{n+2}=p_{n+1}^{2}-1,$ and $p_{n}^{2}=p_{n}^{2}-1+1$, we
see immediately that $p_{n+1}\mu_{n}$ and $p_{n}\mu_{n+1}$ are simple
degenerations of $\tau_{n}$. The case of $\tau'_{n}$ is treated
similarly.
\end{proof}
\begin{example}
The smallest illustration of Proposition \ref{prop:two extreme simple degenerations}
arises from the rigid tree honeycomb $\mu$ with exit pattern $1,1,1|2,1|1,1,1$.
For $n=2,$ we obtain the rigid tree honeycomb with exit pattern $6,6,6|3,1,8,6|6,6,6$.
Figure \ref{fig:double case (B)} 
\begin{figure}
\includegraphics[scale=0.4]{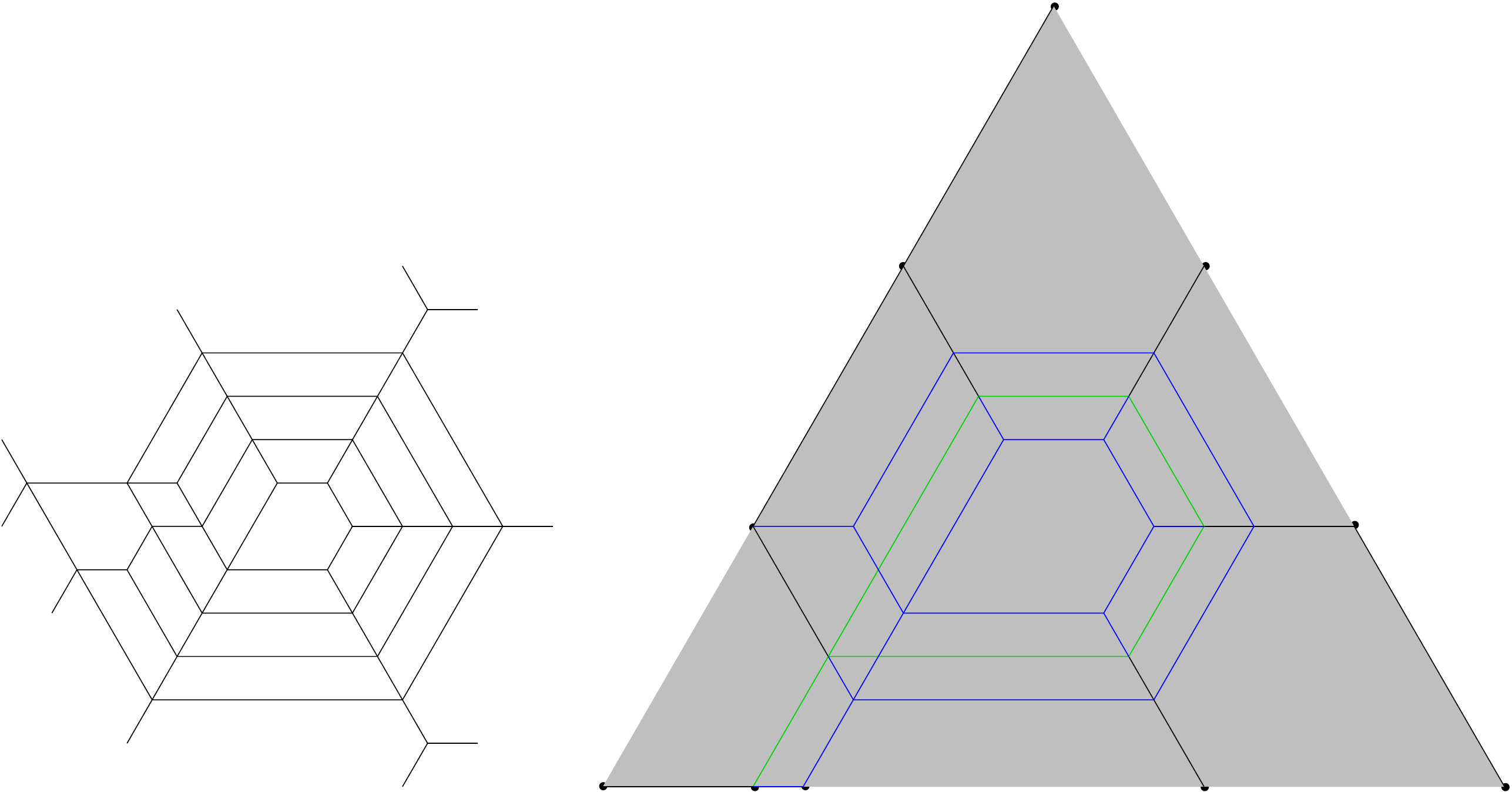}

\caption{\label{fig:double case (B)}A rigid tree honeycomb with exit pattern
$6,6,6|3,1,8,6|6,6,6$ and two extreme summands of its dual that yield
extreme simple degenerations}

\end{figure}
 shows the support of $\tau_{2}^{*}$ along with the supports of the
two extreme summands of $\tau_{2}^{*}$ whose elimination yields the
two extreme simple degenerations of $\tau_{2}$.
\end{example}

\begin{rem}
The simple degenerations $\nu$ of a rigid tree honeycomb fall into
three categories:
\begin{enumerate}
\item [(G)]$\sigma\nu_{1}+\nu_{2}$, where $\nu_{1}$ and $\nu_{2}$ are
distinct tree honeycombs and $\sigma$ is a positive integer. According
to Theorem \ref{thm:there exists a simple simple degeneration}, these
simple degenerations are generic: at most three simple degenerations
are not of this kind.
\item [(A)]$c_{1}\nu_{1}+c_{2}\nu_{2}$, where $\nu_{1}$ and $\nu_{2}$
are distinct tree honeycombs and $c_{1},c_{2}\ge2$ are integers.
\item [(B)]$c\nu$, where $\nu$ is a rigid tree measure and $c$ is a
positive integer.
\end{enumerate}
Thus, in addition to generic simple degenerations, a rigid tree measure
might have
\begin{enumerate}
\item no other degenerations,
\item one degeneration of type (B),
\item one degeneration of type (A),
\item two degenerations of type (B),
\item two degenerations of type (A),
\item one degeneration of type (B) and one of type (A),
\item two degenerations of type (B) and one of type (A),
\item two degenerations of type (A) and one of type (B),
\item three degenerations of type (B),
\item three degenerations of type (A).
\end{enumerate}
We have seen examples in the first four categories, and Theorem \ref{thm:at most two type (B) degenerations}
shows that (7) and (9) are impossible. Extensive experimentation has
not produced any rigid tree honeycombs in categories (5--10) and
we conjecture that no such examples exist.
\end{rem}

\end{document}